\documentclass[a4paper,11pt]{article}

\usepackage{amssymb}
\usepackage{graphicx}
\usepackage{bbm}
\usepackage{mathrsfs}
\usepackage{amsmath}
\usepackage{amsthm}
\usepackage{dsfont}
\usepackage{stmaryrd}
\usepackage[center]{caption}
\usepackage{enumerate}
\usepackage{verbatim}
\usepackage[hmargin=3cm, vmargin=3.5cm]{geometry}
\usepackage{xcolor}
\usepackage{hyperref}
\usepackage{bm}
\usepackage{tikz}
\usepackage{ulem}
\usetikzlibrary{decorations.pathreplacing}

\theoremstyle{plain}
\newtheorem{thm}{Theorem}[]
\newtheorem{cor}[thm]{Corollary}
\newtheorem{lem}[thm]{Lemma}
\newtheorem{prop}[thm]{Proposition}

\theoremstyle{definition}
\newtheorem{ex}[thm]{Example}
\newtheorem{defn}[thm]{Definition}
\newtheorem{rmk}[thm]{Remark}

\renewcommand{\P}{\mathbb{P}}

\newcommand{\R}{\mathbb{R}}
\newcommand{\Rr}{\mathcal{R}}
\newcommand{\N}{\mathbb{N}}

\newcommand{\F}{\mathcal{F}}

\newcommand{\eps}{\varepsilon}
\newcommand{\Exp}{\text{Exp}}

\newcommand{\lesssimP}{\lesssim_{\mathbb{P}}}
\newcommand{\gtrsimP}{\gtrsim_{\mathbb{P}}}

\newcommand{\bp}{\begin{proof}}
	\newcommand{\ep}{\end{proof}}
\def\bal#1\eal{\begin{align*}#1\end{align*}}

\newcommand{\E}[1]{{\mathbb E}\left[#1\right]}

\newcommand{\p}[1]{{\mathbb P}\left(#1\right)}

\def\1{\mathds{1}}

\numberwithin{equation}{section}
\numberwithin{thm}{section}

\usepackage{color}

\author{Colin Desmarais\thanks{Faculty of Mathematics, University of Vienna, Oskar-Morgenstern-Platz 1, 1090 Vienna, Austria}  \thanks{Faculty of Mathematics and Geoinformation, Vienna University of Technology, Wiedner Hauptstraße 8–10, 1040 Vienna, Austria} \and Emmanuel Schertzer\footnotemark[1] \and Zs\'{o}fia Talyig\'{a}s\footnotemark[1]}
\date{}
\title{$K$-Branching Random Walk with Noisy Selection: \\
\Large Large Population Limits and Phase Transitions}

\begin{document}
	\maketitle

 \begin{abstract}
We analyze a variant of the Noisy $K$-Branching Random Walk, a population model that evolves according to the following procedure. At each time step, each individual produces a large number of offspring that inherit the fitness of their parents up to independent and identically distributed fluctuations. The next generation consists of a random sample of all the offspring so that the population size remains fixed, where the sampling is made according to a parameterized Gibbs measure of the fitness of the offspring. Our model interpolates between classical models of fitness waves and exhibits a novel phase transition in the propagation of the wave. By employing a stochastic Hopf-Cole transformation, we show that as we increase the population size, the random dynamics of the model can be described by deterministic operations acting on the limiting population densities. We then show that for fitness fluctuations with exponential tails, these operations admit a unique traveling wave solution with local stability. The traveling wave solution undergoes a phase transition when changing selection pressure, revealing a complex interaction between evolution and natural selection.
\end{abstract}

\tableofcontents

\section{Introduction}\label{sec:Introduction}

Branching particle systems with selection have attracted significant attention in recent years as powerful models for understanding the interplay between reproduction, competition, and evolutionary dynamics. In these systems, the spatial location of a particle encodes its fitness, with particles further to the right corresponding to fitter individuals. Branching corresponds to reproduction: each particle gives rise to a fixed or random number of children, whose positions (fitness values) are random modifications of their parents' positions.  The selection mechanism enforces that only the rightmost (fittest) particles survive and reproduce. We call this procedure truncation selection. A seminal contribution came from Brunet and Derrida~\cite{BrDe:97,BrDe:99}, who introduced a class of such systems, formulated key conjectures about their behavior, and analyzed an exactly solvable version that has since served as a benchmark for both rigorous and heuristic studies.

A natural extension of this framework is to introduce noisy selection, where survival is no longer restricted to the very fittest individuals but instead determined probabilistically. In this setting, one samples from the population in a way that favors fitter particles: those further to the right have a higher chance of being chosen, while still allowing less fit individuals to occasionally persist. This stochastic relaxation of truncation selection better reflects biological and evolutionary contexts where randomness (luck) influences reproductive success. Such models have been studied by Cortines and Mallein~\cite{CoMa:17}, by Schertzer and Wences~\cite{ScWe:23}, and in a companion paper which investigates the biological implications of the present rigorous mathematical results on a slightly different model~\cite{MPST:25}. These works, as well as the present article, highlight how noise affects the long-term behavior of different branching-selection particle systems.

\subsection{The model}\label{sec:model}

We consider an asexual population of fixed size $K$ in discrete time. To each individual in the population, we associate a real number, which represents the individual's evolutionary fitness. The population evolves at each time step according to two substeps: reproduction and selection. Let $X$ be a real-valued random variable with cumulative distribution function $F_X$, let $\beta>0$ and $m\in \mathbb{N}\setminus\{0\}$. We will use the notation $[n]:=\{1,2,\dots,n\}$ for $n\in\N$. We iterate the following steps:

\begin{description}
\item[Reproduction.] Each of the $K$ individuals produces a fixed number of offspring $m$. These children inherit the fitness of their parent up to an independent random fluctuation. The fitness values of the children have independent i.i.d.~displacements from their parents' fitness with distribution $X$. That is, the child of an individual with fitness $x$ has fitness distributed as $x + X$. 
\item[Random selection.] Following reproduction, the population consists of $N:=m K$ individuals. To regulate the size of the population, we sample $K$ individuals without replacement according to the Gibbs sampling weights $(e^{\beta y_i})_{i=1}^{N}$, where $(y_i)_{i=1}^{N}$ denotes the fitness values of the offspring.
\end{description}

One may also think of this process as a particle system with $K$ particles on the real line at all times, where each particle location represents the fitness value of an individual of the population. In order to define the process more formally, we let $(X_{i,j}^t)_{i\in[K],j\in[m],t\in\N}$ be i.i.d.~random variables distributed as $X$. Let the initial particle configuration be given by $\mathcal{Z}_{K}^{0} := (Z_{i}^{0})_{i=1}^{K}$. Now the process $(\mathcal{Z}_{K}^{t})_{t\in\N}$ is defined inductively as follows. Assume that at some time $t\in\N$, the configuration is given by $\mathcal{Z}_{K}^{t} = (Z_{i}^{t})_{i=1}^{K}$. Then, after the \textbf{reproduction} step, the locations of the offspring particles are given by the collection $(Z_i^t+X_{i,j}^t)_{i\in[K],j\in[m]}$. Now in the \textbf{selection} step, out of the $N=Km$ offspring, we sample $K$ without replacement in such a way that the probability of picking an offspring particle at location $y\in\R$ is proportional to $e^{\beta y}$. The configuration $\mathcal{Z}_{K}^{t+1} = (Z_{i}^{t+1})_{i=1}^{K}$ at time $t+1$ is then given by the locations of the sampled particles.

The case $\beta=\infty$ is a version of the $K$-Branching Random Walk ($K$-BRW), see e.g. \cite{BrDe:97,BrDe:99,BDMM:07}, in which only the $K$ fittest offspring  survive after selection. We will refer to the present model as the noisy $K$-BRW with parameter $(\beta,m)$ and reproduction law $F_{X}$. The aim of the present article is to study the case of high reproduction bias ($m>>1$) and show that an interesting phase transition occurs in the large population limit. See the next section for an interpretation of this parameter regime.

Our work relates to another group of fitness wave models, 
considered, for example, in~\cite{DeFi:07,BrRW:08,Schw:17,Schb:17,NeHa:13}, where reproduction and selection happen in one step: a new individual picks a parent with probability proportional to the parent's fitness (or alternatively, fitter individuals reproduce at higher rates). Although the details of this scheme differ from those of the $K$-BRW type models (where reproduction and selection are performed in two substeps), both approaches are believed to capture the same universal large-population behavior. By varying the parameters of the noisy $K$-BRW we will see that it exhibits behaviors similar to the above classical models in one parameter regime, while we find new phenomena in the other.

\subsection{Large population limit}\label{sec:Tail}

Our first theorem says that, in the noisy $K$-BRW model, the fitness distribution (fitness wave) evolves according to a deterministic dynamical system in the large population limit. Before stating this theorem, we specify our assumptions on the parameters of the model ($F_X,\beta,m$) and on the class of functions which will describe the fitness distribution.

\medskip

\noindent
{\bf Tail assumption.} 
Recall that $N=mK\in\N$ denotes the total number of offspring at each time step.
Let $F_X$ be the reproduction law of the noisy $K$-BRW and let 
\begin{equation} 
\label{eq:cN}
c_N := (1-F_X)^{-1}(1/N) := \inf\{x : 1-F_X(x) \geq 1/N\}.
\end{equation}
Assume that  $c_N\to \infty, \text{ as }N\to\infty$,
and for $\rho >0$ and $c_- \in \R_+\cup\{+\infty\}$ we assume the existence of a function 
\begin{equation}\label{eq:defh}
h(x) := \begin{cases}
-x^\rho & x \geq 0,\\
-c_-|x|^\rho & x < 0,
\end{cases}
\end{equation}
such that 
\begin{equation}\label{eq:Karamata}
\begin{aligned}
\forall x \geq 0, \ \ &\lim_{N \to \infty} \frac{\log(1 - F_X(xc_N))}{\log N} = h(x), \ \ \text{and}\\
\forall x<0, \ \ &\lim_{N \to \infty} \frac{\log(F_X(xc_N))}{\log N} = h(x).
\end{aligned}
\end{equation}
This amounts to considering a sub $(\rho < 1)$ or super $(\rho \geq 1)$ exponential tail on the right. When $c_- = +\infty$, the tail is lighter on the left, while the decay is of the same order on both sides when $c_- < \infty$.


\medskip
\noindent
{\bf Log-profiles.}
Let $\beta>0$, $\gamma \in (0,1)$ and $c_N$ as in \eqref{eq:cN}.
We consider a noisy $\lfloor N^{\gamma}\rfloor$-BRW with Gibbs parameter $\beta_N$ and fertility $m_N$ defined as
\begin{equation}\label{eq:betaNrN}
\beta_{N} := \frac{\beta}{c_N}\log N, \ \ m_{N} := N/\lfloor N^{\gamma}\rfloor.
\end{equation}
Let us briefly comment on our choice of parameters. 
First, we note that $m_{N}\to\infty$. Biologically, this means that every individual reproduces a large number of offspring (or gametes) but only a few of those are selected to the next generation (r-selected population~\cite{MaWi:01}). Children of highly fit individuals will have a better chance to survive selection than their peers, leading to high reproductive bias.

\medskip

The biological significance of the parameter $\gamma$ is as follows. For a fixed value of $N$, a lower $\gamma$ entails a higher selection pressure since only a substantially reduced number ($K:=\lfloor N^\gamma \rfloor$) of individuals can reproduce; whereas a high $\gamma$ (close to $1$) corresponds to a mild selection scheme where a larger set of children survive to the next generation. As a consequence, $\gamma$ can be interpreted as capturing the selection pressure in the population. When $\gamma$ is small, selection pressure is high; when $\gamma$ is close to $1$, selection pressure is low.

\medskip

The parameter $\beta$ measures how noisy the system is: large $\beta$ corresponds to low noise, small $\beta$ corresponds to high noise. Furthermore, by~\eqref{eq:betaNrN}, the Gibbs sampling weight of a particle at a location $xc_N$ is a constant power of $N$: $N^{\beta x}$. Our scaling parameter $c_N$ is defined in~\eqref{eq:cN} in such a way that, for any $y>0$, the tail probability $\p{|X|>yc_N}$ is roughly of order $N^{-y^\rho}$; a constant power of $N$.

\medskip

We now describe a stochastic Hopf-Cole transformation of the system, which allows us to capture the time evolution of the fitness distribution in terms of powers of $N$.
In order to justify the upcoming definition, 
let us use the informal notation $a_N\approx b_N$ to say that $a_N$ and $b_N$ are close to each other in some loose sense when $N$ is large. Let $\mathcal{Z}^t_N$ refer to the configuration of particles at time ${t}$. 
Roughly speaking, we will say that $(\mathcal{Z}^t_N)_{N \in \N}$
admits a \textit{limiting log-profile} if and only if there exists a deterministic function $g^t:\R\to\R_+\cup\{-\infty\}$ such that the number of particles in an interval of size $c_N dx$ around $xc_N $
is close to $N^{g^t(x)} dx$, and so the number of particles in an interval $[ac_N,bc_N]$ is close to $\int_a^b N^{g^t(x)}dx \approx N^{\sup_{[a,b]}g^t}$. A more formal definition is given below. Note that when $g^t(x) =-\infty$ there are no particles around $xc_N$. 

\begin{defn}\label{def: D}
A function $g: \R \to \R_+ \cup \{-\infty\}$ belongs to the class $\mathcal{D}$ if and only if 
\begin{itemize}
\item[(a)] $g$ admits a left limit and a right limit for all $x \in \R$, and
\item[(b)] $g$ has bounded support $\mbox{Supp}(g) = \{x \in \R :\: g(x) \neq -\infty \}$, in which case we define $L(g) \leq U(g)$ such that $[L(g), U(g)]$ is the smallest interval containing $\mbox{Supp}(g)$. We call $L(g)$ and $U(g)$ the {\it lower edge} and {\it upper edge} of $g$ respectively.
\end{itemize}
A function $g: \R \to \R_+ \cup \{-\infty\}$ belongs to the class $\mathcal{C}$ if and only if $g \in \mathcal{D}$ and 
\begin{itemize}
\item[(c)] $g$ is concave on its support.
\end{itemize}
\end{defn}
\noindent
We now formally define what it means for a sequence of point measures to have a limiting log-profile. We denote the discontinuity points of $g$ by $D_g$. 

\begin{defn}[limiting log-profile]\label{def: limitlog}
A sequence of (random) point measures $(\mu_N)_{N \in \N}$ is said to have a {\it limiting log-profile} if there exists a deterministic function $g \in \mathcal{D}$ such that 
for all $a,b\in \R\setminus D_g$ with $a<b$, 
\begin{equation}\label{p: MN}
\frac{\log \mu_N((a,b])}{\log N} \xrightarrow{\P} \sup_{x \in (a,b]} g(x), \text{ as }N\to\infty.
 \end{equation}
\end{defn}

Next, we define two operators $r$ and $s_\sigma$ on the space ${\cal D}$, which describe how the limiting log-profile of the configuration changes after a reproduction and a selection step, respectively. We give a heuristic explanation after the statement of Theorem~\ref{thm:main}. Let us define the function $\pi$ by
\begin{equation}\label{eq:pi}
\pi(x) \ := \ \left\{ \begin{array}{cc}  x & \mbox{if $x\geq 0$} \\ -\infty & \mbox{otherwise.} \end{array} \right.
\end{equation}
For all $g\in{\cal D}$ and $x\in\R$, we let
\begin{eqnarray}
& r(g)(x) & := \pi\left(  \sup_{z\in \R}\left( 1-\gamma +  g(z)+ h(x-z) \right)\right ),  \label{def:r}\\
\mbox{and $\forall \sigma>0$}, & s_\sigma(g)(x) & := \pi\left( g(x) + \beta(x-\sigma)_- \right). \label{def:s}
\end{eqnarray}

We can now state the first main result of this paper. Analogous to the notation in Section~\ref{sec:model}, let $(Z_{i,N}^t)_{i=1}^{N^\gamma}$ denote the particle locations of the noisy $\lfloor N^\gamma \rfloor$-BRW at time $t$. Recall the definitions of $c_N$ and $h$ from~\eqref{eq:cN} and~\eqref{eq:defh} and recall $\beta_N$ and $m_N$ from~\eqref{eq:betaNrN}.

\begin{thm}\label{thm:main}
For all $t\geq 0$, we define the sequence of point measures $(M^t_N)_{N\in\N}$ by
\begin{equation}\label{def-MN}
M^t_N := \sum_{i=1}^{N^\gamma} \delta_{Z^t_{N,i}/c_N}.
\end{equation}
Assume that 
\begin{enumerate}
\item[(i):] the sequence of initial configurations  $(M^0_N)_{N \in \N}$ admits a limiting log-profile $g^0 \in {\cal C}$,
\item[(ii):] the function $h$ is concave on its support, that is $\rho\geq 1$.
\end{enumerate}
Then, for all $t\in\N$, $(M^t_N)_{N \in \N}$ admits a limiting log-profile $g^t \in \mathcal{C}$. Further, $(g^t)$ satisfies a recursive equation in terms of the following discrete ``free boundary'' problem,
\begin{eqnarray}
\forall x\in\mathbb{R},  \ g^{t+1}(x) \ = \ s_{\sigma_t}\circ r(g^t)(x),  \label{eq:dynamics}
\end{eqnarray}
where $\sigma^t$ is defined implicitly through the relation 
\begin{equation}\label{eq: sigmat}
\sigma^t := \inf\left\{ \sigma :\: \sup_{z \in \R}s_\sigma(r(g))(z) \leq \gamma\right\}.
\end{equation}
\end{thm}

\medskip

\noindent
\textbf{Heuristics.} Suppose that the sequence of configurations $(M_N^t)_{N\in\N}$ admits a limiting log-profile $g^t$; that is, the number of particles around a location $zc_N$ at time $t$ is roughly $N^{g^t(z)}dz$. Then a first moment argument yields that the averaged number of offspring after reproduction (but before selection) at a location $xc_N$ is 
\begin{equation}\label{eq:ExpRep}
\approx \int N^{g^t(z)}N^{1-\gamma} N^{h(x-z)} dz \approx N^{ \sup_{z}( 1-\gamma +  g^t(z)+ {h(x-z) })},
\end{equation}
where the integral counts the average contribution of every possible parent location, using the fact that we have $\approx N^{g^{t}(z)}dz$ particles at $zc_N$, 
and each of them produces about $N^{1-\gamma}$ offspring with $\approx N^{1 - \gamma + h(x-z)}$ of those offspring around $xc_N$. A Laplace's principle type of argument together with the application of concentration results yields that the stochastic exponent at $x$ converges in probability to the deterministic value $r(g^t)(x)$ defined in~\eqref{def:r}.
The projector $\pi$ encodes the fact that, when the average number of particles at a location is a negative power of $N$, then with high probability, there are no particles at that location. Note that the total number of offspring particles after reproduction is $N$, which is reflected by the fact that $\sup_{x\in\R}r(g^t)(x)=1$ (when $\sup_{x\in\R}g^t(x)=\gamma$).

Next, finding the averaged number of particles sampled in the selection step at every location will lead us to determining the limiting log-profile at time $t+1$.
We need to sample $\lfloor N^{\gamma} \rfloor$ particles without replacement according to the Gibbs weights. For this purpose, we will use an equivalent sampling procedure involving exponential competing clocks, which is convenient for our analysis. We endow each of the $N$ offspring particles with an exponential random variable (clock). If $(x^t_ic_N)_{i=1}^{N}$ is the set of particle locations after reproduction, then the respective rates of the exponential clocks are $e^{\beta_{N} x^t_{i}c_N}=N^{\beta x^t_i}$. We then select the $\lfloor{N^\gamma}\rfloor$ particles with the earliest ring times. We will now argue that the time when $\lfloor{N^\gamma}\rfloor$ clocks have rung, is given by $N^{-\beta\sigma_N}$ for some $\sigma_N\in\R$. 

To determine the limiting log-profile after selection, we first find the limiting log-profile of the configuration of particles, whose clocks rang before time $N^{-\beta\sigma}$ for any $\sigma\in\R$. We show that most particles at a location $x^t_ic_N > \sigma c_N$ will have their clocks ring before time $N^{-\beta \sigma}$, while a particle $x^t_ic_N < \sigma c_N$ has probability $\approx N^{\beta(x^t_i - \sigma)}$ that their clock will ring before time $N^{-\beta \sigma}$. The effect is that, for any fixed $\sigma$, the average number of offspring around $xc_N$, whose clock rang before time $N^{-\beta\sigma}$, is 
\begin{equation}\label{eq:ExpSamp}
\approx N^{r(g^t)(x) + \beta(x - \sigma)_-}dx.
\end{equation}
Applying Laplace's principle and concentration results once more, we conclude that the stochastic exponent of the number of particles at $xc_N$ whose clock rang before time $N^{-\beta\sigma}$ converges in probability to $s_\sigma \circ r (g^t)(x)$ (see~\eqref{def:s}). Thus, we find the total number of such particles by integrating over all locations; and since we need this to be $\sim N^\gamma$, we solve
\begin{equation}\label{eq:ExpSigma}
N^\gamma \approx \int N^{(s_{\sigma^t} \circ r )(g^t)(x)} dx \approx N^{\sup_x (s_{\sigma^t} \circ r )(g^t)(x)}
\end{equation}
for $\sigma_t$, and we will find that $N^{-\beta\sigma_t}$ is indeed the deterministic limit of the times $(N^{-\beta\sigma_N})_N$, when $N^\gamma$ clocks have rung. Therefore, we will be able to conclude that the limiting log-profile after selection, at time $t+1$ is
$$
g^{t+1} \ = \ s_{\sigma_t}\circ r(g^t),
$$
where $\sigma_t$ is given by the implicit relation \eqref{eq: sigmat}.

The heuristics above only apply where the first moment argument yields a non-zero exponent; otherwise we lack the information needed to apply our second moment computation and we cannot determine whether or not there are any particles around that location. This is  illustrated in Section \ref{sec:Examples}. As a benefit of assuming (i) and (ii) of Theorem \ref{thm:main}, we show that the functions in the exponents of \eqref{eq:ExpRep} and \eqref{eq:ExpSamp} after our first moment argument are also concave, and within any interval whose endpoints are continuity points of the functions $r(g^t)$ and $s_\sigma \circ r (g^t)$ respectively, the maximum value attained is non-zero. As we have defined limiting log-profiles in Definition \ref{def: limitlog} outside of discontinuity points, we may indeed apply our second moment computation and prove that $r(g^t) \in \mathcal{C}$ and $s_{\sigma} \circ r(g^t) \in \mathcal{C}$ are the limiting log-profiles of our processes after reproduction and sampling according to exponential clocks respectively.

\medskip

We show in Section \ref{sec:Reproduction} that the argument laid out above for reproduction generalizes to $g^t \in \mathcal{D}$ and $\rho >0$, and a limiting log-profile exists. However, a limiting log-profile after selection cannot be guaranteed in general, as illustrated in Section \ref{sec:Examples}. Under some additional conditions, our argument for selection extends to limiting log-profiles $g^t \in \mathcal{D}$, and $\rho>0$ which is captured in Theorem \ref{thm:mainGeneral}.

\subsection{Traveling wave solutions and the phase transition}

Theorem~\ref{thm:main} establishes the dynamical system that describes the behavior of the fitness wave in the noisy $\lfloor N^\gamma \rfloor$-BRW model. Therefore, in order to characterize this stochastic model, we may analyze the dynamical system~\eqref{eq:dynamics}. We are particularly interested in how the noise parameter $\beta$ and the selection pressure $\gamma$ affect the shape and speed of the fitness wave. Note that the latter can be interpreted as the rate of adaptation of the population to its environment. 

We have explicit results for the particular case when the displacement distribution $F_X$ has exponential tails. Throughout this section, we assume that $\rho = 1$ for $h$ in \eqref{eq:defh}. We say that the dynamics (\ref{eq:dynamics}) admits a traveling wave solution if and only if there exists $(G,\nu)$ such that
if we initialize the profile $g^0$ at $G$, i.e.~$g^0=G$, then 
$$
\forall x\in\R, \forall t\geq0, \ \ g^{t+1}(x+t \nu) = G(x).
$$
Traveling wave solutions are always defined up to translation and provided that $\mbox{Supp}(G)$ is bounded, we take the convention that $U(G) = 0$, where we recall from Definition~\ref{def: D} that $U(G)$ is the upper edge of the profile. For the space of functions in ${\cal D}$, we define the metric
\[ \phi(g_1,g_2) := \sup_{x \in \R}\left\lvert g_1(x)_+ - g_2(x)_+\right\rvert \vee |U(g_1) - U(g_2)| \vee |L(g_1) - L(g_2)|.\]

\noindent
{\bf Selection of the fittest and luckiest phase transition.} 
Theorem~\ref{thm:transition} says that there is a unique traveling wave solution to the dynamics~\eqref{eq:dynamics}. Moreover, later in this manuscript, we provide an explicit expression for this traveling wave solution. We will prove that it can be described as 
concave piecewise linear functions as illustrated in Fig. \ref{fig:G}.
The other message of Theorem~\ref{thm:transition} is that
the noisy $\lfloor N^\gamma \rfloor$-BRW exhibits a phase transition. This phase transition is 
characterized in two ways: by the rate of adaptation and by the geometry of the wave. 

\begin{thm}
\label{thm:transition}
Let $\beta>0$, $\gamma\in(0,1)$ and assume that $\rho=1$ in~\eqref{eq:defh}. Then, the dynamical system~\eqref{eq:dynamics} admits a unique traveling wave solution $G$ (explicitly given by Proposition~\ref{prop:unique}) with speed $\nu$. Furthermore, let
\begin{equation}\label{eq:gam}
k := k(\beta) = \left\{\begin{array}{cc}
\lfloor\frac{1}{\beta}\rfloor & \mbox{if $\frac{1}{\beta}\notin\N$} \\ \frac{1}{\beta}-1 & \mbox{otherwise} 
\end{array}\right. \quad \text{and} \quad
 \gamma_c(\beta) :=  \frac{k(2-(k+1)\beta)}{(k+1)(2 - k\beta)}.
\end{equation}
Then the traveling wave solution exhibits the following phase transition. For a fixed value of $\beta$,
\begin{enumerate}
\item[(a)] {\bf Selection of the luckiest}: If $\gamma<\gamma_{c}(\beta)$,
\begin{enumerate}
\item[(a.1)] the speed of the wave is given by 
$$
\nu(\gamma,\beta)=\frac{2\gamma}{k(2-(k+1)\beta)}. 
$$
In particular, it is increasing in $\gamma$.
\item[(a.2)] If $g^0=G$, then  $\forall x\in \mbox{Supp}(g^1), \ \ r(g^0)(x)> g^1(x).$
\end{enumerate}
\item[(b)] {\bf Selection of the fittest}: If $\gamma \geq \gamma_{c}(\beta)$,
\begin{enumerate}
\item[(b.1)]  the speed of the wave is given by 
$$
\nu(\gamma,\beta)=1-\gamma.
$$
In particular, it is decreasing in $\gamma$.
\item[(b.2)] If $g^0=G$, then
$
\forall x\in [-\chi + 1-\gamma,1-\gamma], \ \ r(g^0)(x)= g^1(x) 
$ where 
\begin{equation}
\chi(\gamma,\beta) := \gamma - \sum_{j=1}^k(1-j\beta)(1-\gamma), \label{eq:defchi}
\end{equation}
which is strictly positive on $(\gamma_c,1)$.
\end{enumerate}
\end{enumerate}
\end{thm}

\noindent
\textbf{Local convergence.}
Our next theorem shows local convergence to the traveling wave solution $G$ of~\eqref{eq:dynamics}. We discuss some of the implications after the statement of Theorem~\ref{thm:speed}. Recall the critical value $\gamma_c(\beta)$ from~\eqref{eq:gam}.

\begin{thm}\label{thm:speed}
Let $\beta>0$, $\gamma\in(0,1)$ and assume that $\rho=1$ in~\eqref{eq:defh}. If we also assume the technical conditions $\gamma\neq\gamma_{c}(\beta)$ and $1/\beta\notin \N$ whenever $\gamma > \gamma_c(\beta)$, then the traveling wave solution $G$ of~\eqref{eq:dynamics} is locally stable in the sense that 
there exists $\delta>0$ such that if $g^0\in{\cal D}$ and 
$$
\phi(g^0,G) \leq \delta
$$
then there exists $c \in \R, C>0$ and $r\in(0,1)$ such that
$$
\forall t\in\N, \ \ \ \phi(g^t(\cdot + \nu t + c), G)\leq  C r^t.
$$
\end{thm}

\begin{figure}[h!]
\centering
\begin{tikzpicture}[xscale = 4, yscale = 3, domain=1:1]
\draw[<->, thick] (-2.1,0) -- (0.75,0);
\draw[<-, thick] (0, 1.2) -- (0,-0.05) node[below]{0};
\draw[dashed] (-2.1,0.4) -- (0.75,0.4) node[right]{$\gamma = 0.4$};
\draw[dashed] (-2.1,1) -- (0.75,1) node[right]{$1$};
\draw[very thick, red] (0,0) -- (-1/3,7/30) -- (-2/3,11/30) -- (-1,12/30) -- (-4/3,10/30) -- (-5/3,5/30)-- (-15/8,0);
\draw ( -1/3,0.05) -- (-1/3,-0.05) node[below]{$-\nu$};
\draw ( -2/3,0.05) -- (-2/3,-0.05) node[below]{$-2\nu$};
\draw ( -3/3,0.05) -- (-3/3,-0.05) node[below]{$-3\nu$};
\draw ( -4/3,0.05) -- (-4/3,-0.05) node[below]{$-4\nu$};
\draw ( -5/3,0.05) -- (-5/3,-0.05) node[below]{$-5\nu$};

\draw(0.6,0.05) -- (0.6,-0.05) node[below]{$1-\gamma$};
\draw[decoration={brace,amplitude=2mm}, decorate, teal, thick] (0,0) -- (1/3,0);
\node at (1/6, 0.08) {$\nu$};

\draw[very thick, black] (0.6,0) -- (0,0.6) -- (-1/3,25/30) -- (-2/3,29/30) -- (-1,1) -- (-4/3,28/30) -- (-5/3,23/30)-- (-15/8,18/30);

\draw[very thick, blue] (1/3,0) -- (0,7/30) -- (-1/3,11/30) -- (-2/3,12/30) -- (-1,10/30) -- (-4/3,5/30)-- (-37/24,0);
\end{tikzpicture}

\begin{tikzpicture}[xscale = 4, yscale = 3, domain=1:1]
\draw[<->, thick] (-2.6,0) -- (0.55,0);
\draw[<-, thick] (0,1.2) -- (0,-0.05) node[below]{$0$};
\draw[dashed] (-2.6,0.6) -- (0.55,0.6) node[right]{$\gamma=0.6$};
\draw[dashed] (-2.6,1) -- (0.55,1) node[right]{$1$};
\draw[very thick, red] (0,0) -- (-12/100,12/100) -- (-52/100,4/10) -- (-92/100, 56/100) --(-132/100, 6/10) -- (-172/100, 52/100) -- (-212/100, 32/100) -- (-252/100, 0);
\draw[dashed] (-12/100,12/100) -- (-12/100,-0.05) node[below] {$-\chi$};
\draw (-52/100,0.05) -- (-52/100,-0.05) node[below] {\small$-\chi-\nu$};
\draw (-92/100,0.05) -- (-92/100,-0.05) node[below] {\small$-\chi-2\nu$};
\draw (-132/100,0.05) -- (-132/100,-0.05) node[below] {\small$-\chi-3\nu$};
\draw (-172/100,0.05) -- (-172/100,-0.05) node[below] {\small$-\chi-4\nu$};
\draw (-212/100,0.05) -- (-212/100,-0.05) node[below] {\small$-\chi-5\nu$};
\draw (-252/100,0.05) -- (-252/100,-0.05) node[below] {\small$-\chi-6\nu$};

\draw[very thick, black] (0.4,0) -- (0,0.4) -- (-12/100,52/100) -- (-52/100,8/10) -- (-92/100, 96/100) --(-132/100, 1) -- (-172/100, 92/100) -- (-212/100, 72/100) -- (-252/100, 4/10);
\draw[very thick, blue] (0.4,0) -- (28/100,12/100) -- (-12/100,4/10) -- (-52/100, 56/100) --(-92/100, 6/10) -- (-132/100, 52/100) -- (-172/100, 32/100) -- (-212/100, 0);

\draw(0.4,0.05) -- (0.4,-0.05) node[below]{$1-\gamma$};
\draw[decoration={brace,amplitude=2mm}, decorate, teal, thick] (0,0) -- (0.4,0);
\node at (0.2, 0.1) {$\nu$};
\end{tikzpicture}

\caption{Two examples of $G$ (in red) where $c_- = \infty$ and $\beta = 0.3$, along with the limiting log-profiles after reproduction $r(G)$ (in black) and the limiting log-profile after selection $s \circ r (G) = G( \cdot -\nu)$ in blue. The first plot corresponds to the selection of the luckiest regime with $\gamma = 0.4$, and the second plot corresponds to the selection of the fittest regime with $\gamma = 0.6$. In both plots, $G$ consists of line segments with slopes $-(1-6\beta), \ldots, -(1-\beta)$, and a further line segment with slope $-1$ in the second plot.}
\label{fig:G}
\end{figure}
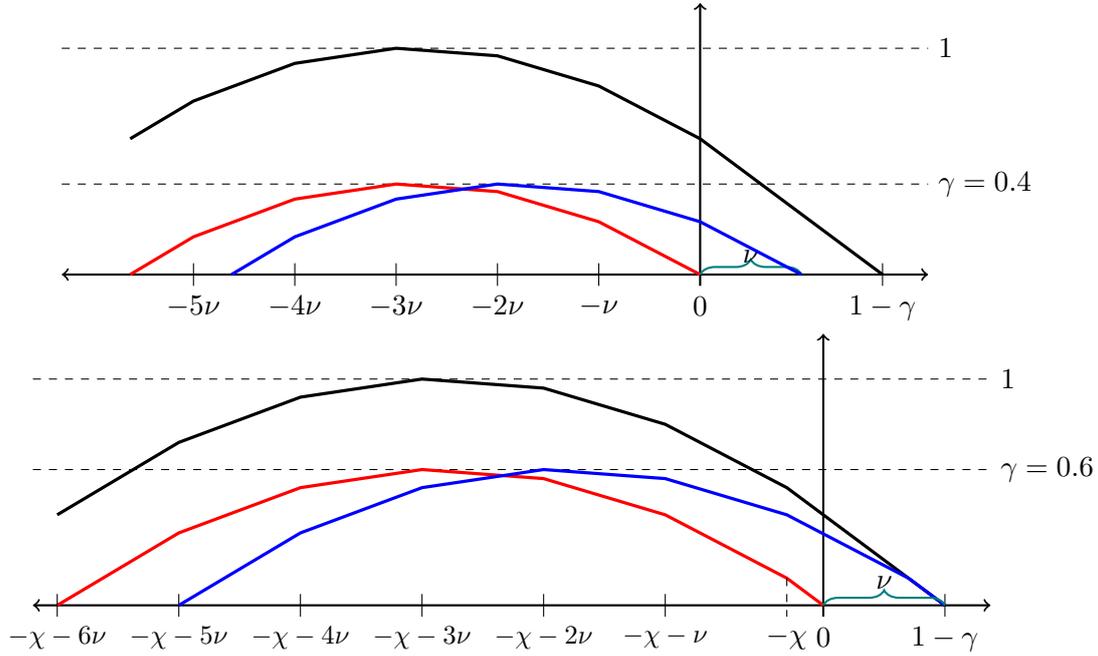

\noindent
\textbf{Interpretation.}
To understand the significance of the theorems above, we recall that
$\gamma$ captures the selection pressure (high $\gamma$ means low selection pressure). Theorem~\ref{thm:transition} says that for any $\beta>0$ there is a critical value of $\gamma$ that maximizes the speed of the traveling wave. That is, the function $\gamma\mapsto\nu(\gamma,\beta)$ is maximized at $\gamma=\gamma_c(\beta)$ from~\eqref{eq:gam} (see also Fig.~\ref{fig:betaCrit}). In the regime to the left of $\gamma_c$, as we decrease $\gamma$ from $\gamma_c$ to $0$, selection becomes stricter: fewer of the $N$ children are allowed to reproduce. We say that selection is `too strict' in this case, in the sense that the individual with the largest fitness (i.e.~the rightmost particle) after reproduction is not selected. This is the message of the statement (a.2) and it is illustrated in the first plot of Fig~\ref{fig:G}. The rightmost selected particle (right edge of the blue curve) is to the left of the rightmost particle after reproduction (right edge of the black curve): the few fittest individuals after reproduction do not have enough weight to survive under high selection pressure. As a result, as Theorem~\ref{thm:transition} and Fig.~\ref{fig:betaCrit} show, higher selection pressure is detrimental to the rate of adaptation, that is, the speed of the wave is increasing in $\gamma$. We call this regime selection of the luckiest, because it is the noise that determines which of the children survive to the next generation rather than the actual fitness of the children.

On the other hand, when $\gamma\geq\gamma_c$, more of the $N$ children are allowed to reproduce and enough for the fittest to be selected, hence, we call this regime selection of the fittest. The second plot of Fig~\ref{fig:G} shows that there is an interval of length $\chi$, in which the log-profile after reproduction (black) agrees with the log-profile after selection (blue). That is, there is an interval at the front of the wave, in which every offspring particle is selected to survive. This is the message of statement (b.2) in Theorem~\ref{thm:transition}. Once $\gamma$ is large enough that the fittest individual survives, further increasing $\gamma$ will mean that selection is `not strict enough': the average fitness of the selected individuals decreases by increasing $\gamma$ and therefore the speed decreases, as shown in Fig.~\ref{fig:betaCrit}.

Comparing the geometry of the traveling wave solutions in the two regimes, we find that the slope of $G$ at the front (near zero) differ in the two cases: in the selection of the luckiest case it is $-(1-\beta)$, whereas in the selection of the fittest regime it is $-1$. This property also reflects the fact that in the latter case there is a segment at the front of the population where every offspring survives selection.

Theorem~\ref{thm:transition} and Fig~\ref{fig:betaCrit} also show that the speed of the traveling wave is monotone increasing in $\beta$, which is not surprising: as noise decreases speed increases. We also observe in this figure that $\gamma_c(\beta)$ decreases in $\beta$, which means that, if we optimize on the speed, we have to be more selective when selection is less noisy.

The paper~\cite{MPST:25} is closely related to our work and provides a more detailed description of the two regimes, focusing on the biological implications for a different version of our model. In that paper the noisy $\lfloor N^\gamma \rfloor$-BRW is studied with the same reproduction step as in the present article, but with a selection step different from Gibbs sampling.
Each individual is assumed to have a genotype and a phenotype, both represented by a real number. The genotype and phenotype of a child are distributed as $x+X$ and $x+X+Y$ respectively, where $x$ denotes the genotype of the parent, $X\sim$ Laplace$(1)$ represents random mutations, and $Y\sim$ Laplace$(1/\mu)$ represents the noise between genotype and phenotype, and where the Laplace distribution with parameter $b>0$ has density $\frac{1}{2b}e^{-|x|/b}$. Then, in the selection step, out of the $N$ offspring, the offspring with the $\lfloor N^\gamma \rfloor$ largest phenotypes are selected. The next generation inherits the genotypes of the selected individuals up to random mutation.

This model also exhibits the phase transition between selection of the fittest and selection of the luckiest regimes, however, as opposed to the present work, the paper~\cite{MPST:25} does not provide a rigorous proof for the convergence to the deterministic dynamics or for the local convergence to the traveling wave. The selection of the luckiest regime in this model has an interesting interpretation; namely, it can be thought of as a Goodhart's law~\cite{Good:84} in evolution. In general, Goodhart's law says that `When a measure becomes a target, it ceases to be a good measure.' That is, when selection is highly noisy, selecting individuals only with the few largest observed phenotypes is not a good way to find individuals with the largest genotypes. One can certainly think of the same interpretation of the present article.

\bigskip

\begin{figure}[h!]
\centering

\begin{tikzpicture}[xscale = 9, yscale = 3, domain=1:1]
\draw[thick, ->] (0, 0) -- (1.1,0) node[below right] {$\gamma$};
\draw[thick, ->] (0, 0) -- (0,1.2) node[above left] {$\nu$};

\draw (0.025, 1) -- (-0.025, 1) node[left] {$1$};
\draw (1,0.075) -- (1,-0.075) node[below] {$1$};

\draw[dashed] (0.0909,0.9091) -- (0.0909,0) node[below] {$\gamma_c(0.9)$};
\draw[teal, samples = 10, thick, domain=0.001:0.0909, variable=\x] plot({\x},{(2*\x)/(2-2*0.9)});
\draw[teal, samples = 10, thick, domain=0.0909:1, variable=\x] plot({\x},{1-\x});

\draw[dashed] (0.2857,0.7143) -- (0.2857,0) node[below] {$\gamma_c(0.6)$};
\draw[red, samples = 10, thick, domain=0.001:0.2857, variable=\x] plot({\x},{(2*\x)/(2-2*0.6)});
\draw[red, samples = 10, thick, domain=0.2857:1, variable=\x] plot({\x},{1-\x});
        
\draw[dashed] (0.54545,0.45455) -- (0.54545,0) node[below] {$\gamma_c(0.3)$};
\draw[blue, samples = 10, thick, domain=0.001:0.54545, variable=\x] plot({\x},{(2*\x)/(3*(2-4*0.3))});
\draw[blue, samples = 10, thick, domain=0.54545:1, variable=\x] plot({\x},{1-\x});

\end{tikzpicture}
\caption{Plots of the speed of the traveling wave $\nu(\beta,\gamma)$ for 3 different values of $\beta$: $\beta = 0.3$ (in blue), $\beta=0.6$ (in red), and $\beta = 0.9$ (in teal).}
\label{fig:betaCrit}
\end{figure}
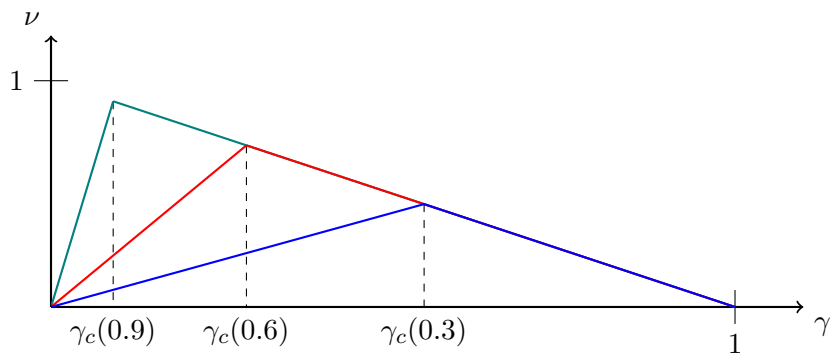

\medskip

\subsection{Conjectures and related work}

Brunet and Derrida~\cite{BrDe:97,BrDe:99} considered a version of a $K$-BRW and (heuristically) argued that the distribution of particles should be given by a deterministic FKPP-type equation, which admits traveling wave solutions; and that the speed of the stochastic system should be given by the minimal velocity solution of the deterministic equation. These conjectures led to several rigorous results, for example on the convergence of branching-selection particle systems to free boundary problems~\cite{DuRe:11} and \cite{BeBP:19,DFPS:17}; on the speed of the traveling wave solution of the FKPP equation with cut-off~\cite{DuPK:07} and with random noise~\cite{MuMQ:08}; and on the limiting velocity of branching-selection particle systems~\cite{BeGo:10,BeMa:14,CoMa:17,ScWe:23}. Our results fit into this literature and leave some open questions, which are subject to future work.

\medskip

{\bf Rate of adaptation.} Our Theorems~\ref{thm:main},~\ref{thm:transition},~\ref{thm:speed} show that under some general conditions on the tails of the reproduction law, the noisy $\lfloor N^{\gamma}\rfloor$-BRW admits a limiting log-profile, which changes according to a deterministic dynamical system. For $\rho=1$ (see~\eqref{eq:defh}), we show local convergence to a unique traveling wave solution of this system, where the shape and speed of the wave are explicitly given. It is then natural to ask what we can say about the speed (rate of adaptation) of the original stochastic particle system. Does it converge to the same speed as the traveling wave solution? A first step towards such a result would be to show that for fixed values of $N$, an asymptotic speed $\nu_N\in\R$ exists, satisfying
\[
\lim_{t\to\infty}  \frac{\max_{i=1,\dots,\lfloor N^\gamma \rfloor}Z_{N,i}^t}{t} = \lim_{t\to\infty}  \frac{\min_{i=1,\dots,\lfloor N^\gamma \rfloor}Z_{N,i}^t}{t} = \nu_N \quad \text{a.s.}
\]
Assuming such an asymptotic speed exists one may investigate its behavior, as $N\to\infty$. Our results (in the case $\rho=1$) suggest that we should see 
\[
\frac{\nu_N}{\log N} \to \nu
\]
as $N\to\infty$, where $\nu$ is defined in Theorem~\ref{thm:transition}. 

\medskip

{\bf Genealogical structure.} Brunet, Derrida, Mueller and Munier~\cite{BDMM:07} predicted that the Bolthausen-Sznitman coalescent~\cite{BoSz:98,Pitm:99} is a universal limiting genealogy for branching-selection particle systems. We conjecture that the noisy $\lfloor N^\gamma \rfloor$-BRW only belongs to this universality class in the selection of the fittest regime; and we see a different coalescent in the limit in the selection of the luckiest regime, when noise has the more significant role. This conjecture is in line with the results of~\cite{ScWe:23}.

In~\cite{ScWe:23}, a variation of the  noisy $\lfloor N^\gamma \rfloor$-BRW model was considered, which was introduced by Cortines and Mallein~\cite{CoMa:17} (who proved convergence to the Bolthausen-Sznitman coalescent for $\beta>1$) as a generalization of the exponential model of Brunet and Derrida~\cite{BrDe:97, BDMM:07}.
It is assumed in~\cite{ScWe:23} that at each reproduction step, every individual produces an infinite number of offspring according to an exponential Poisson point process centered around the parental value. Then, $\lfloor N^\gamma \rfloor$ individuals are selected using a sampling scheme interpolating between truncation selection and Gibbs sampling.

The critical assumption of this model is the choice of the exponential reproduction law. In particular, one crucial observation (already made by Brunet and Derrida) is that, up to translation, the system reaches stationarity after a {\it single} step. (Compared to Theorem~\ref{thm:speed}, the system reaches its travelling wave regime at infinite rate instead of geometric rate.) This makes the model fully integrable and allows 
to make precise predictions on the large-population limit and in particular, on the genealogical structure of the model.
In~\cite{ScWe:23}, the existence of a critical $\gamma_{c}(\beta)\in[0,1]$ was proved, segregating between two different limiting regimes.
\begin{itemize}
\item If $\gamma<\gamma_{c}(\beta)$, then the genealogy of the population is given by the Poisson-Dirichlet coalescent with parameter $(0,\beta)$ \cite{PiYo:97,Schw:03}. 
\item If $\gamma>\gamma_{c}(\beta)$, then, after rescaling time by $c\log(N)$ (for some explicit choice $c\equiv c(\beta,\gamma)$), the genealogy of the population is given by the Bolthausen-Sznitman coalescent. 
\end{itemize}
We conjecture that the same transition occurs in the noisy $\lfloor N^\gamma \rfloor$-BRW. However, we emphasize that the proof in \cite{ScWe:23} heavily relies on the integrability of the exponential model 
and more specifically, on the property that the system reaches stationarity after only a single step. Thus, the geometric rate of convergence in Theorem \ref{thm:speed} appears to be a first important step in the direction 
of proving the universality of the phase transition observed in the model.

\subsection{Outline of the paper}

This work proceeds as follows: Section \ref{sec:SingleGen} is dedicated to proving Theorem \ref{thm:main} and the more general Theorem \ref{thm:mainGeneral}. The proof relies on investigating the limiting log-profile after reproduction in Section~\ref{sec:Reproduction} and after selection in Section~\ref{sec:Sampling}.
We examine an example of where a limiting log-profile after selection fails to materialize in Section~\ref{sec:Examples}. 

In Sections~\ref{sect:pwllogprof} and~\ref{sec:tws} we prove that the deterministic system~\eqref{eq:dynamics} admits a unique traveling wave solution, which is given by a piecewise linear function, which exhibits a phase transition as we change the selection pressure $\gamma$. In Section~\ref{sect:pwllogprof} we introduce the space $\mathcal T$ of concave piecewise linear functions, where the traveling wave solution will belong. In Section~\ref{sec:Toperators} we describe how the operators defined in~\eqref{def:r} and~\eqref{def:s} act on such piecewise linear functions. In Section~\ref{sec:pwl} we prove that the space $\mathcal T$ is strongly attractive with respect to~\eqref{eq:dynamics}; starting from any $g^0:\R \to \R_+\cup\{-\infty\}$ with bounded support such that $\sup_{x}g^0(x) = \gamma$, after finitely many iterations of~\eqref{eq:dynamics}, every $g^t$ is a concave piecewise linear function. Using the results of Section~\ref{sect:pwllogprof} we manage to translate~\eqref{eq:dynamics} to a finite-dimensional system describing the evolution of the ``break-points'' where the slopes of the linear segments change. Analyzing this finite-dimensional problem, we prove Theorems~\ref{thm:transition} and~\ref{thm:speed} in Sections~\ref{sec:proofthmTransition} and~\ref{sec:proofthmSpeed}.

\subsection{Some notation}\label{sec:mathIntro}

Throughout the rest of this work, for ease of reading, we will assume that $N^\gamma$ and $N^{1-\gamma}$ are integers. Specifically, we will write $N^\gamma$ and $N^{1-\gamma}$ for $\lfloor{N^\gamma}\rfloor$ and $N/\lfloor{N^{\gamma}}\rfloor$ respectively. For an integer $K$, we use the notation $[K] = \{1, \ldots, K\}$, and so we write $[N^\gamma] = \{1, \ldots, \lfloor{N^\gamma}\rfloor\}$ and $[N^{1-\gamma}] = \{1, \ldots, N/\lfloor{N^{\gamma}}\rfloor\}$. 

\medskip

We will write 
$\bar \R=\{-\infty\}\cup \{+ \infty\} \cup \R$.

For a function $f: \R \to \R \cup \{-\infty\}$, we will say that $f$ has a left limit at $x$ if and only if $\lim_{z \to x^-}f(z)$ exists and is real, or if there exists $\delta > 0$ such that $f(z) = -\infty$ for all $z \in (x-\delta, x)$, in which case we say that $\lim_{z \to x^-} f(x) = -\infty$ exists. Then $f$ is left continuous at $x$ if and only if $\lim_{z \to x^-}f(z) = f(x)$. Similarly we define right limits and right continuity, and say that $f$ is continuous at $x$ if and only if $f$ is both left and right continuous at $x$, and discontinuous otherwise. We will denote the set of discontinuity points of $f$ by $D_f$. Recall that any function $f$ which admits a left and right limit at every $x \in \R$ has at most countably many discontinuities.

In a similar vein as c\`adl\`ag functions, we say that a function $f$ is {\it c\`all\`al} (continue \`a l'un, limite \`a l'autre) if and only if for all $x \in \R$, $f$ is continuous on one side of $x$ and has a limit on the other. 

We extend the definition of concavity to functions $f\in{\cal D}$. Notice that such a function that is concave on a compact support $\mbox{Supp}(f) := \{x : f(x) \neq -\infty\}$ is also concave on the real line; if $x$ or $y$ is taken outside of $\mbox{Supp}(f)$, then for all $\alpha \in [0,1]$, 
\[ f(\alpha x + (1-\alpha) y ) \geq \alpha f(x) + (1-\alpha)f(y) = -\infty.\]

For any point measure $M_N$,
we write
$$
\forall a<b, \ \ \ M_N(a,b) = M_N((a,b]).
$$

For sequences of random variables $(X_N)_{N\in\N}$ and $(Y_N)_{N\in\N}$ we write $X_N\lesssimP Y_N\in\R$ as $N\to\infty$, if and only if, for all $\eps>0$,
\[
\p{X_N-Y_N>\eps} \to 0,
\]
as $N\to\infty$. We often use the notation with $X_N$ or $Y_N$ being simply a deterministic constant.
In particular, if $X_N\lesssimP c\in\R$ and $X_N\gtrsimP c$ as $N\to\infty$, then $X_N\xrightarrow{\mathbb P} c$ as $N\to\infty$.

We will also use the following property. 
Suppose that $X_N\lesssimP Y_N$ as $N\to\infty$, and that \[Y_N\xrightarrow{\mathbb P} Y^*\]
as $N\to\infty$ for some random (or deterministic) variable $Y^*$. Then we have
\[
X_N \lesssimP Y^*,
\]
as $N\to\infty$.  If we assume $X_N\gtrsimP Y_N$ instead, then it follows that $X_N \gtrsimP Y^*$.

\section{Dynamics of the model as operators on limiting log-profiles}\label{sec:SingleGen}

This section is divided as follows. We relegate our technical lemmas to Section \ref{sec:technical}. In Section \ref{sec:Reproduction}, we prove that a deterministic operator $r$ describes the change of the limiting log-profile for point measures after reproduction, while a similar result for the selection step under technical conditions is proven in Section \ref{sec:Sampling}. The conditions are needed since though we get a quite detailed perspective of the process when looking at a logarithmic scale, there are some exceptional cases where the scale does not offer enough precision to describe a limiting log-profile after sampling. These minor conditions, however, are always satisfied for limiting log-profiles in $\mathcal{C}$ and when $h$ is concave. We then provide a proof of Theorem \ref{thm:main} in Section \ref{sec:ProofMain}.

\subsection{Technical Lemmas}\label{sec:technical}

\subsubsection{A Laplace's principle type Lemma}

We prove a version of Laplace's principle for the type of Riemann-Stieltjes integrals that we encounter throughout this work. Recall the definition of the space of functions $\mathcal{D}$ (Definition~\ref{def: D}), and recall that a function $f$ is c\`all\`al if and only if for every $x \in \R$, the function $f$ is continuous on one side and has a limit on the other. Before stating the main result of this section, we will need to prove an easy lemma.

\begin{lem}\label{lem:M_Nto0}
Assume $(M_N)_{N \in \N}$ has limiting log-profile $g \in \mathcal{D}$ with lower and upper edges $L,U$. For any $a,b \in {\bar R}\setminus D_g$ with $a<b$, if $\sup_{x \in (a,b]} g(x) = -\infty$, then as $N \to \infty$,
\[ M_N((a,b]) \xrightarrow{\P} 0.\]
Furthermore, for the leftmost point $\rho_{N,L}$ and rightmost point $\rho_{N,U}$ of $M_N$, as $N \to \infty$,
\[ \rho_{N,L} \xrightarrow{\P} L, \quad \text{and} \quad \rho_{N,U} \xrightarrow{\P} U.\]
\end{lem}

\begin{proof}
Since $M_N((a,b])$ takes integer values, 
\[ \P(M_N((a,b]) > 0) = \P\left(\frac{\log M_N((a,b])}{\log N} \geq 0\right).\]
Thus if $\log M_N((a,b])/\log N \xrightarrow{\P} -\infty$, then $M_N((a,b]) \xrightarrow{\P} 0$. 

Now we examine $\rho_{N,U}$. Using the above result and definition of $\mathcal{D}$, for any $\delta > 0$, 
\begin{equation}\label{eq:rhoUupper}
 \P\left(\rho_{N,U} > U+\delta\right) = \P(M_N((U+\delta, \infty)) > 0) \xrightarrow{N \to \infty} 0.
 \end{equation}
Since $g$ has countably many discontinuity points, we may always find $0 < \delta' \leq \delta$  such that $g$ is continuous at $U - \delta'$, and so 
\[\frac{\log M_N((U-\delta', \infty))}{\log N} \xrightarrow{\P} \sup_{x \in (U-\delta', \infty)}g(x) > 0\]
and as such,
\begin{equation}\label{eq:rhoUlower}
 \P( \rho_{N,U} < U - \delta) \leq \P(M_N((U-\delta', \infty)) = 0) \xrightarrow{N \to \infty} 0.
 \end{equation}
Then $\rho_{N,U} \xrightarrow{\P} U$ follows from \eqref{eq:rhoUupper} and \eqref{eq:rhoUlower}. A symmetric argument holds for $\rho_{N,L}$. 
\end{proof}

\begin{lem}\label{lem:integral}
Let $f_N: \R \to \R\cup\{-\infty\}$ be a sequence of functions and let $f:\R \to \R\cup\{-\infty\}$ be a c\`all\`al function with discontinuity set $D_f$ such that for all $x \in \R$,
\begin{align}
\lim_{\delta \to 0}\lim_{N \to \infty}\sup_{z \in B(x,\delta)}f_N(z) &=\lim_{\delta \to 0} \sup_{z \in B(x, \delta)}f(z), \label{eq:IntLemmaSupCond}\\
\lim_{\delta \to 0}\lim_{N \to \infty}\inf_{z \in B(x, \delta)}f_N(z) &=\lim_{\delta \to 0} \inf_{z \in B(x, \delta)}f(z).\label{eq:IntLemmaInfCond}
\end{align}
Let $(M_N)_{N \in \N}$ be a sequence of (random) point measures on $\R$ which has a limiting log-profile $g \in \mathcal{D}$ with discontinuity set $D_g$. If $D_f \cap D_g = \emptyset$, then for all $a,b \in \overline{\R} ~\setminus~(D_f~\cup~D_g)$, as $N \to \infty$, 
\[ \frac{\log\left(\int_{(a,b]}N^{f_N(x)} dM_N(x)\right)}{\log N} \xrightarrow{\P} \sup_{x \in (a,b]}\left(f(x) + g(x)\right)\in \R \cup \{-\infty\}.\]
\end{lem}

\begin{rmk}\label{rmk:unifIntLemma}
The conditions \eqref{eq:IntLemmaSupCond} and \eqref{eq:IntLemmaInfCond} hold trivially if $f_N$ converges locally uniformly to $f$, while these conditions imply pointwise convergence at all continuity points of $f$. 
\end{rmk}

\begin{proof}
Note throughout that \eqref{p: MN} can be rewritten as 
\begin{equation}\label{p: MNint}
 \frac{\log\left(\int_{(a,b]} dM_N(x)\right)}{\log N} \xrightarrow{\P} \sup_{x \in (a,b]} g(x)
 \end{equation}
for $a,b \not\in D_g$. \\

We will first assume that all $f_N$ and $f$ are real-valued functions.
Let $L$ and $U$ be the lower and upper edges of $g$. First suppose that $a = -\infty$ and $b < L$. From Lemma~\ref{lem:M_Nto0}, $M_N((-\infty, b]) \xrightarrow{\P} 0$, and so in this case, 
\[ \P\left(\int_{(-\infty,b]} N^{f_N(x)}dM_N(x) > 0 \right) \leq \P\left(M_N(-\infty, b) > 0\right) \to 0\]
as $N \to \infty$, and so 
\[ \frac{\log\left( \int_{(-\infty, b]}N^{f_N(x)}dM_N(x)\right)}{\log N} \xrightarrow{\P} \sup_{x \in (-\infty, b]}g(x) = -\infty.\]
 A similar argument holds if $a > U$ and $b = \infty$, completing the proof in these cases.\\

Now suppose $a,b \in \R\setminus(D_f \cup D_g)$ and fix $\eps > 0$. From \eqref{eq:IntLemmaSupCond} and \eqref{eq:IntLemmaInfCond}, for all $x \in \R$, there exists a $\delta_x' > 0$ such that for all $\delta_x\leq \delta_x'$,
\begin{equation}\label{eq:SupInfCloseEnough}
\sup_{z \in B(x, \delta_x)}f_N(z) \leq \sup_{z \in B(x, \delta_x)}f(z) + \eps \qquad \text{and} \qquad \inf_{z \in B(x, \delta_x)}f_N(z) \geq \inf_{z \in B(x, \delta_x)}f(z) - \eps
\end{equation}
holds for all $N$ large enough. Since $f$ is c\`all\`al, $g$ has left and right limits for every $x \in \R$, and $D_f \cap D_g = \emptyset$, then for every $x \in [a,b]$, there exists $0 < \delta_x \leq \delta_x'$ such that for the open neighbourhood $B(x, \delta_x)$, one of the two holds:
\begin{itemize}
\item[(a)] for all $y,z \in B(x, \delta_x)$, $|f(y)-f(z)| < \eps$ and if $(y-x)(z-x)>0$, then $|g(y) - g(z)| < \eps$, or 
\item[(b):] for all $y,z \in B(x, \delta_x)$, $|g(y)-g(z)| < \eps$ and if $(y-x)(z-x) > 0$, then $|f(y) - f(z)| < \eps$.
\end{itemize}
The condition $(y-x)(z-x) > 0$ ensures $y$ and $z$ are on the same side of $x$. Since $g$ and $f$ have left and right limits for all $x \in \R$, $D_g \cup D_f$ is countable, and so we may also specify that $\delta_x$ is chosen in such a way that in addition to the above, both $f$ and $g$ are continuous at $x - \delta_x$ and $x + \delta_x$. 
Since $[a,b]$ is compact and $\{B(x, \delta_x)\}_{x \in[a,b]}$ is an open cover of $[a,b]$, there is a finite subcover $\{ B(x_i, \delta_{x_i})\}_{i=1}^K$ of $[a,b]$. Let 
\[ B_i := (x_i - \delta_{x_i}, x_i + \delta_{x_i}] \cap (a,b],\]
 so that $\bigcup B_i = (a,b]$. We show that for all $i = 1, \ldots, K$, 
\begin{equation}\label{eq:IntBoundEps}
\sup_{x \in B_i}(f(x) + g(x)) -5\eps \lesssimP \frac{\log\left(\int_{B_i}N^{f_N(x)}dM_N(x)\right)}{\log N} \lesssimP \sup_{x \in B_i}( f(x) + g(x)) + 3\eps.
\end{equation}
Let $s_i^f := \sup_{x \in B_i}f(x)$ and $s_i^g := \sup_{x \in B_i}g(x)$. By the definition of $\delta_{x_i}$ and the continuity of $f$ and $g$ at $b \wedge (x_i + \delta_{x_i, \eps})$, we see that 
\begin{equation}\label{eq:SupCont}
 \left\lvert \sup_{x \in B_i}(f(x) + g(x)) - \left(s_i^f + s_i^g\right)\right\rvert \leq 2\eps.
 \end{equation}
 For the upper bound of \eqref{eq:IntBoundEps}, using \eqref{eq:SupInfCloseEnough}, since $a\vee (x_i - \delta_{x_i})$ and $b \wedge (x_i + \delta_{x_i})$ are continuity points of $g$, \eqref{p: MNint} guarantees that for all $N$ large enough
\[ \frac{\log\left(\int_{B_i}N^{f_N(x)}dM_N(x)\right)}{\log N}  \leq \frac{\log\left(N^{s^f_i + \eps}\int_{B_i}dM_N(x)\right)}{\log N} 
\xrightarrow{\P} s^f_i + s^g_i + \eps,\]
which along with \eqref{eq:SupCont} proves the upper bound of \eqref{eq:IntBoundEps}.

 For the lower bound, we consider the two cases (a) and (b) for $B(x_i, \delta_{x_i})$. Note that if $a=x_i$ or $b=x_i$, since $f$ is continuous at $a$ and $b$, we may assume $B(x_i, \delta_{x_i})$ is of type (a) in this case. 
 
 If $B(x_i, \delta_{x_i})$ is of type (a), then $\inf_{x \in B(x_i, \delta_{x_i})}f(x) \geq s^f_i - \eps,$
 which along with \eqref{eq:SupInfCloseEnough} guarantees that for all $N$ large enough, $f_N(x) \geq s^f_i - 2\eps$ for all $x \in B(x_i, \delta_{x_i})$. By using \eqref{p: MNint} again, 
\[ \frac{\log\left(\int_{B_i}N^{f_N(x)}dM_N(x)\right)}{\log N}  \geq \frac{\log\left(N^{s^f_i - 2\eps}\int_{B_i}dM_N(x)\right)}{\log N} 
\xrightarrow{\P} s^f_i -2\eps + s^g_i,\]
which along with \eqref{eq:SupCont} completes the lower bound of \eqref{eq:IntBoundEps} in this case. 

Now suppose $B(x_i, \delta_{x_i})$ is of type (b). We define a new value $y_i$ in the following way; if $f(x) \geq s^f_i - \eps$ for all $x < x_i$ in $B_i$, let $y_i = x_i - \delta_{x_i}/2$, while if $f(x) \geq s^f_i - \eps$ for all $x > x_i$ in $B_i$, let $y_i = x_i + \delta_{x_i}/2$ (at least one of these two cases must hold from the definition of $B(x_i, \delta_{x_i})$ and since $f$ is either right continuous or left continuous at $x_i$). Repeating similar arguments as above, from \eqref{eq:IntLemmaInfCond}, we may find $\delta_{y_i} \leq \delta_{x_i}/2$ such that $y_i - \delta_{y_i}, y_i + \delta_{y_i} \not\in D_g$ and for all $N$ large enough, $f_N(x) \geq s^f_i - 2\eps$ for all $x \in B(y_i, \delta_{y_i})$. 
Since $B(x_i, \delta_{x_i})$ is of type (b), 
\begin{equation}\label{eq:weirdgsup}
\sup_{x \in B(y_i, \delta_{y_i})} g(x) \geq s^g_i - \eps,
 \end{equation}
and so from \eqref{p: MNint} once more, 
\[ \frac{\log\left(\int_{B_i}N^{f_N(x)}dM_N(x)\right)}{\log N} \geq \frac{\log\left( N^{s^f_i - 2\eps}\int_{B(y_i, \delta_{y_i})}dM_N(x)  \right)}{\log N} \xrightarrow{\P} s^f_i - 2\eps +\! \sup_{x \in B(y_i, \delta_{y_i})} \!g(x).\]
Substituting \eqref{eq:weirdgsup} and \eqref{eq:SupCont} into the above equation completes the lower bound of \eqref{eq:IntBoundEps} in this case as well. 

Suppose $j$ is such that 
\[\sup_{x \in B_j}(f(x) + g(x)) = \max_{i=1}^K \left\{\sup_{x \in B_i}(f(x) + g(x))\right\}.\]
 With \eqref{eq:IntBoundEps} in hand and since there are only finitely many $B_i$, we see that 
\begin{align*}
\frac{\log\left(\int_{(a,b]}N^{f_N(x)}dM_N(x)\right)}{\log N} &\leq \frac{\log\left( \sum_{i=1}^K \int_{B_i}N^{f_N(x)}dM_N(x)\right)}{\log N} \\
&\leq \frac{\log(K) + \log\left(\max_{i=1}^K \int_{B_i} N^{f_N(x)}dM_N(x)\right)}{\log N} \\
&\lesssimP \sup_{x \in B_j}(f(x) + g(x)) + 3\eps,
\end{align*}
while 
\[\frac{\log\left(\int_{(a,b]}N^{f_N(x)}dM_N(x)\right)}{\log N} \geq \frac{\log \left(\int_{B_j} N^{f_N(x)}dM_N(x)\right)}{\log N} \gtrsimP \sup_{x \in B_j}(f(x) + g(x)) - 5\eps.\]
Since $\bigcup B_i = (a,b]$, we see that 
\[\sup_{x \in (a,b]}(f(x) + g(x)) = \sup_{x \in B_j}(f(x) + g(x)),\]
 and so letting $\eps \to 0$ proves the statement of the lemma in this case. \\

Now consider any $a,b \in \overline{\R} \setminus(D_f \cup D_g)$ such that $a < b$. Let $x_L < L$ and $x_U > U$ be continuity points of $f$ and $g$. Then splitting the integrals 
\begin{multline*}
\int_{(a,b]}N^{f(x)}dM_N(x) = \int_{(a\wedge x_L, x_L]} N^{f(x)}dM_N(x) \\
 + \int_{(a\vee x_L, b\wedge x_U]}N^{f(x)}dM_N(x)+ \int_{(x_U, b\vee x_U]}N^{f(x)}dM_N(x),
\end{multline*}
and applying the previous cases proves the lemma for real-valued functions. \\

Now we allow $f_N:\R \to \R\cup\{-\infty\}$ and $f: \R\to \R\cup\{-\infty\}$. The statement of the lemma holds by truncating the functions and applying the argument above. For all $M \in \R$, let $f_N^M(x) = -M \vee f_N(x)$ and $f^M(x) = -M\vee f(x)$. It is evident that the conditions \eqref{eq:IntLemmaSupCond} and \eqref{eq:IntLemmaInfCond} hold for the truncated functions as well, and so in particular 
\begin{equation}\label{eq:IntLemM} \frac{\log\left(\int_{(a,b]}N^{f^M_N(x)} dM_N(x)\right)}{\log N} \xrightarrow{\P} \sup_{x \in (a,b]}(f^M(x) + g(x)) = \sup_{x \in (a,b]}(-M\vee f(x) + g(x)).
\end{equation}
Since $g(x)$ is bounded above, if $\sup_{x \in (a,b]}(f(x) + g(x)) = -\infty$, then \eqref{eq:IntLemM} tends to $-\infty$ as $M$ increases, while if $\sup_{x \in (a,b]}(f(x) + g(x)) \in \R$, then \eqref{eq:IntLemM} is precisely equal to this value for all $M$ large enough. In either case, letting $M \to \infty$ completes the proof.
\end{proof}

\subsubsection{Concentration inequalities}

We adapt classic concentration inequalities to be applied for our results. We will use the following version of McDiarmid's inequality.

\begin{lem}[McDiarmid, \cite{MCDI:98}]\label{thm: McD}
Let $X_1,\dots,X_n$ be independent random variables with ${X_k\in[0,1]}$ for all $k\in[n]$. Let $S_n:=\sum_{k=1}^{n}X_k$ and let $\mu=\E{S_n}$. Then for any $\epsilon\in(0,1)$,
\[
\p{|S_n-\mu| > \epsilon\mu} < e^{-\frac{1}{4}\epsilon^2\mu}.
\]
\end{lem}

Let $(\ell_N)_{N\in\N}$ and $(a_j^N)_{N,j\in\N}$ be sequences of random variables. For all $N\in\N$, let $\F_N$ be a $\sigma$-algebra such that $\ell_N,a_1^N,\dots,a_{\ell_N}^N\in\F_N$. Let $(\xi_j^N))_{N,j\in\N}$ be random variables with $\xi_j^N\in\{0,1\}$ for all $j,N\in\N$. Suppose that, conditioned on $\F_N$, the random variables $(\xi_j^N))_{N,j\in\N}$ are independent, and they are distributed as Ber$(a_j^N\wedge 1)$. We also let $S_{\ell_N} := \sum_{j=1}^{\ell_N}\xi_j^N$ and $\E{S_{\ell_N}\;|\;\F_N} =: \mu_N$ for all $N\in\N$.

\begin{lem}\label{lem:concentration}
\begin{enumerate}[(a)]
\item 
If there exists a constant $C_1>0$ such that $\frac{\log \mu_N}{\log N} \xrightarrow{\mathbb P} C_1$, then
\[
\frac{\log S_{\ell_N}}{\log N} \xrightarrow{\mathbb P} C_1
\]
as $N\to\infty$.
\item 
If there exists a constant $C_2\in\R$ such that $\frac{\log \mu_N}{\log N} \lesssimP C_2$, then, as $N\to\infty$,
\[
\frac{\log S_{\ell_N}}{\log N}  \lesssimP C_2, \quad \text{if }C_2\geq 0,
\]
and
\[
S_{\ell_N} \xrightarrow{\mathbb{P}} 0, \quad \text{if }C_2<0.
\]

\end{enumerate}	
	
\end{lem}

\begin{proof}
We first prove part (a). Take $0<\eps<1\wedge C_1/2$. It is enough to show that the probability of the event $\{S_{\ell_N}\notin [N^{C_1-2\eps},N^{C_1+2\eps}]\}$ converges to zero as $N\to\infty$. Splitting this event based on the value of $\mu_N$ we get
\begin{multline}\label{eq: PSlN}
\p{S_{\ell_N}\notin [N^{C_1-2\eps},N^{C_1+2\eps}]} \leq \p{S_{\ell_N}\notin [N^{C_1-2\eps},N^{C_1+2\eps}], \mu_N\in [N^{C_1-\eps},N^{C_1+\eps}]} \\
+ \p{\mu_N\notin [N^{C_1-\eps},N^{C_1+\eps}]}.
\end{multline}

Notice that for $N$ sufficiently large we have
\[
\left\{ S_{\ell_N}+1\notin [N^{C_1-2\eps},N^{C_1+2\eps}], \mu_N\in [N^{C_1-\eps},N^{C_1+\eps}] \right\} \subseteq \left\{ |S_{\ell_N} - \mu_N| > \frac{\mu_N}{2} \right\}.
\]
Note furthermore, that the second term on the right-hand side of~\eqref{eq: PSlN} converges to zero as $N\to\infty$ by the assumption of the lemma on $\mu_N$.
Hence, continuing~\eqref{eq: PSlN}, and first conditioning on $\F_N$, and then applying Lemma~\ref{thm: McD}, for $N$ sufficiently large we obtain
\begin{align*}
\p{S_{\ell_N}\notin [N^{C_1-2\eps},N^{C_1+2\eps}]} & \leq \E{\p{|S_{\ell_N} - \mu_N| > \frac{\mu_N}{2}\;|\; \F_N}\1_{\{ \mu_N\in [N^{C_1-\eps},N^{C_1+\eps}\}}} + o(1)
\\ &
\leq
\E{e^{-\frac{\mu_N}{16}}\1_{\{ \mu_N\in [N^{C_1-\eps},N^{C_1+\eps}\}}} +o(1)
\\ &
\leq
\E{e^{-N^{C_1-\eps}/16}} + o(1),
\end{align*}
which converges to zero as $N\to\infty$, showing part (a). 

We now move on to part (b). Take $0<\eps<1\wedge |C_2|/2$. Similarly to part (a), in this case we have
\begin{align*}
\p{S_{\ell_N} > N^{C_2+2\eps}} \leq \E{\p{S_{\ell_N} > N^{C_2+2\eps}\;|\; \F_N }\1_{\mu_N<N^{C_2+\eps}} } + \p{\mu_N > N^{C_2+\eps}}.
\end{align*}
By Markov's inequality and our assumption on $\mu_N$, and recalling that $\mu_N=\E{S_{\ell_N}\;|\;\F_N}$, we arrive at
\begin{align}\label{eq: SlN0}
\p{S_{\ell_N} > N^{C_2+2\eps}} \leq \E{\frac{\mu_N \1_{\{\mu_N<N^{C_2+\eps}\}} }{N^{C_2+2\eps}}} + o(1) \leq N^{-\eps} + o(1) \to 0, 
\end{align}
as $N\to\infty$. Note that~\eqref{eq: SlN0} holds for any $C_2\in\R$. If $C_2\geq 0$, then this shows the first statement of part (b). If $C_2<0$, then the second statement follows from~\eqref{eq: SlN0} and from the fact that $S_{\ell_N}$ takes integer values. 
\end{proof}

\subsection{Limiting log-profiles after reproduction}\label{sec:Reproduction}

For a function $g \in \mathcal{D}$, define $\bar{r}(g)$ and $r(g)$ by
\begin{equation}\label{eq: Rg}
\bar{r}(g)(x) :=
1 - \gamma + \sup_{z\in\R} (g(z) + h(x-z) ), \qquad \text{ and } \qquad r(g)(x):= \pi\big(\bar{r}(g)(x)\big).
\end{equation}

\begin{figure}[h!]
\begin{center}
\begin{tikzpicture}[scale=2.5]
\draw[->] (0,0) -- (0,1.2) node[above]{$y$};
\draw[dashed] (-4.5,0.6) -- (1, 0.6) node[right]{$\gamma$};
\draw[dashed] (-4.5, 1) -- (1, 1) node[right]{$1$};
\draw[red, smooth,very thick, domain=0:-4, variable = \x] plot({\x}, {-(0.15)*((\x+2)^2-4)});
\node at (-4.2,0.2) {\color{red} $g$};
\draw[black, smooth,very thick, domain=0.58:-4, variable = \x] plot({\x}, {-(0.15)*((\x+2)^2-3)+0.55});
\node at (-4.2,0.75) {$r(g)$};

\draw[<->,thick] (-4.5,0) -- (1,0) node[right]{$x$};
\end{tikzpicture}
\end{center}
\caption{A potential outcome of applying the operator $r$ to the log-profile $g$ (in red), producing the functions $r(g)$ (in black).}
\end{figure}
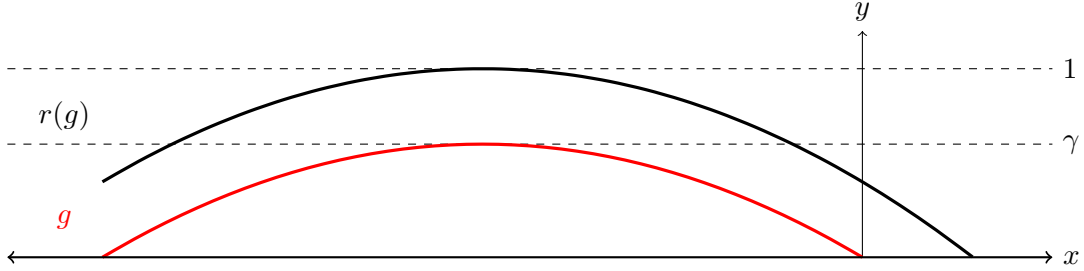

Let $\big(X_{i,j} :\: i \in [N^{\gamma}], j \in [N^{1 - \gamma}]\big)$ be an array of i.i.d.~random variables  distributed as $X$. Consider a sequence of point measures $({\cal Z}_N = ({Z}_{N,i})_{i=1}^{N^\gamma})_{N \in \N}$. We assume that the sequence of measure 
 $$
 M_{N} \ := \ \sum_{i=1}^{N^\gamma} \delta_{Z_{N,i}/c_N}
 $$
has a limiting log-profile $g \in \mathcal{D}$. Finally, we define the profile after reproduction as 
$$
M^R_N \ = \ \sum_{i=1}^{N^\gamma} \sum_{j=1}^{N^{1-\gamma}} \delta_{(Z_{N,i} +X_{i,j})/c_N}.
$$
The goal of this section is to prove the following Lemma: 

\begin{lem}\label{lem:Reproduction}
The sequence of measures $(M_N^R)_{N \in \N}$ has limiting log-profile $r(g) \in \mathcal{D}$. 
\end{lem}

To prove Lemma \ref{lem:Reproduction} we make use of Lemma \ref{lem:integral} and Lemma \ref{lem:concentration} to prove a convergence in probability after scaling for $M_N^R(a,b)$. However, to apply Lemma \ref{lem:integral}, certain continuity conditions at $a$ and $b$ must be satisfied, while we require that the supremum of $\bar{r}$ in the interval $(a,b]$ be nonzero to apply Lemma \ref{lem:concentration}. Therefore, we first apply our probabilistic argument where these conditions are met to achieve Lemma \ref{lem:ReproductionProb} below, and finally use some technical continuity arguments to finish the proof of Lemma \ref{lem:Reproduction}. 

\begin{lem}\label{lem:ReproductionProb}
 For all $a,b \in \R\setminus D_{g}$, if $\sup_{x \in (a,b]} \bar{r}(g)(x) \neq 0$, then 
 \[ \frac{\log M_N^R(a,b)}{\log N} \xrightarrow{\P} \sup_{x \in (a,b]} r(g)(x).\]
 \end{lem}
 
\begin{proof}
Set $\bar{r} := \bar{r}(g)$ and $r := r(g)$. 
We have 
\begin{equation}\label{eq: MRN}
M_N^R(a,b) = 
\int_{(ac_N,bc_N]} \sum_{i=1}^{N^\gamma} \sum_{j=1}^{N^{1-\gamma}} \delta_{X_{ij}+{Z}_{N,i}}(dx) =
\sum_{i=1}^{N^\gamma} \sum_{j=1}^{N^{1 - \gamma}} \1_{\{X_{i,j} \in (ac_N -{Z}_{N,i}, bc_N - {Z}_{N,i}]\}}.
\end{equation}
We first show that for all $a,b \in \R\setminus D_{g}$ with $a<b$, 
\begin{equation}\label{eq:RepA}
A_N(a,b) := \frac{\log(\E{M_N^R(a,b)\;|\; \sigma( \mathcal{Z}_N)})}{\log N} \xrightarrow{\P} \sup_{x \in (a,b]} \bar{r}(x).
\end{equation}
For all $a<b$, define
\[ p_N(a,b) := \frac{\log\left(F_X(bc_N) - F_X(ac_N)\right)}{\log N}.\]
From our assumptions of the tail of $X$, one can verify that for all $z \in \R$, 
\[ p_N(a-z, b-z) \xrightarrow{N \to \infty} H_{a,b}(z) := \!\! \sup_{x \in (a-z,b-z]} \!\! h(x),\]
and that for $a-z, b-z \neq 0$, 
\begin{equation}\label{eq:IntLemCond}
\begin{aligned}
\lim_{\delta \to 0}\lim_{N \to \infty} \sup_{x \in B(z,\delta)} p_N(a-x, b-x) 
&= \lim_{\delta \to 0} \sup_{x \in B(z,\delta)} H_{a,b}(x),\\
\lim_{\delta \to 0}\lim_{N \to \infty} \inf_{x \in B(z,\delta)} p_N(a-x, b-x) 
&= \lim_{\delta \to 0} \inf_{x \in B(z,\delta)} H_{a,b}(x).
\end{aligned}
\end{equation}
The function $H_{a,b}(z)$ has at most two points of discontinuity (when $z=a$ or $z=b$) and is c\`all\`al. Since all $X_{i,j}$ are i.i.d.,~\eqref{eq: MRN} gives us
\begin{align*}
A_N(a,b) &= \frac{1}{\log N}\log \left(\sum_{i=1}^{N^\gamma} N^{1-\gamma} \p{X \in (ac_N - Z_{N,i}, bc_N - Z_{N,i}]\;|\;\sigma( \mathcal{Z}_N)}\right)
\\ &
=\frac{1}{\log N} \log\left(\int_{\R}N^{1-\gamma} N^{p_N(a-z, b-z)}dM_N(z)  \right). 
\end{align*}
From \eqref{eq:IntLemCond}, so long as $a,b \not\in D_g$ (and so $D_{H_{a,b}}\cap D_g = \emptyset$), we may apply Lemma \ref{lem:integral} and see that 
\[ A_N(a,b) \xrightarrow{\P} \sup_{z \in \R}\left( 1- \gamma + g(z) + H_{a,b}(z)\right).\]
Observe that by the definition of the function $H_{a,b}$ and by swapping the two suprema, we have
\begin{align*}
\sup_{z\in\R}(1 - \gamma+ g(z) + H_{a,b}(z) ) & = \sup_{z\in\R}\left(1-\gamma+ g(z) + \sup_{x\in(a-z,b-z]}h(x)\right) \\
& = \sup_{x\in(a,b]}\left(1 - \gamma+ \sup_{z\in\R} \{g(z) + h(x-z) \}\right)
\\ &
= \sup_{x\in(a,b]} \bar r(x),
\end{align*}
and so \eqref{eq:RepA} holds.\\

Now apply Lemma~\ref{lem:concentration} with $\ell_N=N$, $a_{i,j}^N = \p{X \in (ac_N - Z_{N,i}, bc_N - Z_{N,i}]}$ for all $i\in[N^\gamma]$ and $j\in[N^{1 - \gamma}]$, $\xi_{i,j} = \1_{\{X_{i,j}\in (ac_N - Z_{N,i}, bc_N - Z_{N,i}]\}}$, with $\F_N =\sigma({\cal Z}_N)$. By 
Lemma~\ref{lem:concentration}(a) we conclude that
\begin{equation*}
\frac{\log M^R_N(a,b)}{\log N}\overset{\mathbb P}{\to} \sup_{x\in(a,b]} \bar r(x) = \sup_{x\in(a,b]} r(x),
\end{equation*} 
if $\sup_{x\in(a,b]} \bar r(x)>0$; and by Lemma~\ref{lem:concentration}(b) we see that
\begin{equation*}
M^R_N(a,b) \overset{\mathbb P}{\to} 0,
\end{equation*}
if $\sup_{x\in(a,b]} \bar r(x)<0$, which concludes the proof.
\end{proof}

\begin{proof}[Proof of Lemma \ref{lem:Reproduction}]
Once more let $\bar{r} := \bar{r}(g)$ and $r := r(g)$. 
It is an elementary argument to show that $r(g) \in \mathcal{D}$, that is left to the reader. We now use the previous lemma to prove that $r$ is the limiting log-profile of $M_N^R$.

 \medskip

Assume that $a,b \in \R$ such that $\sup_{x \in (a,b]}\bar{r}(x) = 0$. We show that either $a$ or $b$ must be a discontinuity point of $r$. Recalling $h$ from \eqref{eq:defh}, $h^{-1}(\{x \geq -(1-\gamma)\})$ is a closed interval containing $0$, call it $[-p,q]$. Note that $q > 0$ and that $p> 0$ if $c_- \neq \infty$ and $p=0$ otherwise, while $h(x) > -(1-\gamma)$ for all $x \in (-p,q)$. Therefore, for any $x \in \R$,
\[\text{if } g(z) \geq 0 \text{ for some } z \in (x-q,x+p) \text{, then } \bar{r}(x) \geq 1-\gamma +\!\! \sup_{z \in (x-q,x+p)} \!\! \left(g(z) + h(x-z)\right) > 0.\]
Fix $x \in (a,b]$. Since $\bar{r}(x) \leq 0$, the above implies $g(z) = -\infty$ for all $x \in (x-q,x+p)$ and
\begin{equation}\label{eq:Whererbar} 
\bar{r}(x) =1 - \gamma + \sup_{z \leq x - q}\left(g(z) + h(x-z)\right) \,\, \text{ or } \,\, \bar{r}(x) = 1 - \gamma + \sup_{z \geq x+p}\left(g(z) + h(x-z)\right).
\end{equation}
Suppose $\bar{r}(x) \neq -\infty$ and is given by the first expression of \eqref{eq:Whererbar}. Since the function $h$ is strictly decreasing on $\R_+$ and $g$ has bounded support, we see that for all $q > \delta > 0$, 
\[ \sup_{z \in \R}\left(g(z) + h(x-\delta - z)\right) \geq \sup_{z \leq x-q}\left(g(z) + h(x - \delta - z)\right) > \sup_{z \leq x- q}\left(g(z) + h(x-z)\right),\]
and so $\bar{r}(x-\delta) > \bar{r}(x)$. If $\bar{r}(x) \neq -\infty$ and is given by the second expression of \eqref{eq:Whererbar}, then $c_- \neq \infty$, and since the function $h$ is strictly increasing on $\R_-$, a similar argument shows that $\bar{r}(x+\delta) > \bar{r}(x)$ for all $p > \delta > 0$. In all cases, we see that since $\sup_{x \in (a,b]}\bar{r}(x) = 0$,
\begin{equation}\label{eq:abDiscR}
\forall \delta > 0 \text{ small enough}, \qquad \sup_{x \in (a+\delta,b-\delta)}\bar{r}(x) < 0 \qquad \text{and} \qquad \sup_{x \in (a-\delta,b+\delta)}\bar{r}(x) > 0.
\end{equation}
The function $\pi$ sends negative values to $-\infty$, and so \eqref{eq:abDiscR} implies that at least one of $a$ or $b$ is a discontinuity point of $r = \pi(\bar{r})$. Therefore, 
\begin{equation}\label{eq:rCondno0}
a,b \in \R\setminus D_r \,\, \Longrightarrow\,\, \sup_{x \in (a,b]} \bar{r}(x) \neq 0.
\end{equation}

\medskip

 Now take any continuity points $a,b \in \R\setminus D_{r}$. Since $D_g \cup D_r$ is countable, we may find $a^+ >a$ and $b^- < b$ such that $a^+, b^- \not\in D_g \cup D_r$ and
\[ \sup_{x \in (a^+, b^-]}r(x) = \sup_{x \in (a,b]} r(x) = -\infty, \qquad \text{or} \qquad \sup_{x \in (a,b]} r(x) - \sup_{x \in (a^+, b^-]} r(x) < \eps.\]
Since $\sup_{x \in (a^+,b^-]}r(x) \neq 0$, we use Lemma \ref{lem:ReproductionProb} and get 
\[ \frac{\log M_N^R(a,b)}{\log N} \geq \frac{\log M_N^R(a^+,b^-)}{\log N} \xrightarrow{\P} \sup_{x \in (a^+,b^-]}r(x).\]
A similar argument for an upper bound, and letting $\eps \to 0$, we see that indeed 
\[ \frac{\log M_N^R(a,b)}{\log N} \xrightarrow{\P} \sup_{x \in (a,b]}r(x).\]
This holds for all $a,b \in \R\setminus D_r$, and so along with \eqref{eq:rCondno0}, $r$ is the limiting log-profile of $(M_N^R)_{N\in\N}$, as desired.
\end{proof}

\subsection{Limiting log-profiles after sampling}\label{sec:Sampling}

Let $r\in{\cal D}$.
For any $\sigma \in \R$, define $\bar{s}_\sigma(r)$ and $s_\sigma(r)$ by 
\begin{equation}\label{eq: Ssigmar}
\bar{s}_\sigma(r)(x) := r(x) + \beta(x - \sigma)_-, \qquad \text{ and } \qquad s_\sigma(r)(x) := \pi \big( \bar{s}_\sigma(r)(x)\big).
\end{equation}
We then define 
\begin{equation}\label{eq: sigma*}
\sigma^*(r) := \inf\left\{ \sigma :\: \sup_{z \in \R}s_\sigma(r)(z) \leq \gamma\right\}
\end{equation}
and finally, we define $\bar{s}(r)$ and $s(r)$ by 
\begin{equation}\label{eq: S}
\bar{s}(r) := \bar{s}_{\sigma^*(r)}(r), \qquad \text{ and } \qquad s(r) := s_{\sigma^*(r)}(r) = \pi(\bar{s}(r) ).
\end{equation}

\begin{figure}[h!]
\centering

\begin{tikzpicture}[scale=2.5]
\draw[<->] (-4.5,0) -- (1,0) node[right]{$x$};
\draw[->] (0,0) -- (0,1.2) node[above]{$y$};
\draw[dashed] (-4.5,0.6) -- (1, 0.6) node[right]{$\gamma$};
\draw[dashed] (-4.5, 1) -- (1, 1) node[right]{$1$};
\draw[black, smooth,very thick, domain=0.58:-4, variable = \x] plot({\x}, {-(0.15)*((\x+2)^2-3)+0.55});
\node at (-4.2,0.8) { $r$};
\draw[orange, smooth,very thick, domain=0.26:-2.58, variable = \x] plot({\x}, {-(0.15)*((\x+2)^2-3)+0.55 +min(0,0.25*(\x-1.2)});
\node at (-1.3,0.1) {\color{orange} $s_{\sigma_1}(r)$};
\draw[blue, smooth,very thick, domain=0.4:-3.18, variable = \x] plot({\x}, {-(0.15)*((\x+2)^2-3)+0.55 +min(0,0.25*(\x-0)});
\node at (-3.3,0.1) {\color{blue} $s(r)$};
\draw[violet, smooth,very thick, domain=0.58:-3.62, variable = \x] plot({\x}, {-(0.15)*((\x+2)^2-3)+0.55 +min(0,0.25*(\x+1.2)});
\node at (-3.9,0.1) {\color{violet} $s_{\sigma_2}(r)$};

\draw[<->,thick] (-4.5,0) -- (1,0) node[right]{$x$};
\end{tikzpicture}
\caption{Applications of the operator $s_\sigma$ to $r$ (in black). The functions $s_\sigma(r)$ are portrayed for three values $\sigma_1 > \sigma^{\ast}(r) > \sigma_2$, where $s_{\sigma_1}(r)$ is in violet, $s(r) = s_{\sigma^\ast(r)}(r)$ is in blue, and $s_{\sigma_2}(r)$ is in orange.}

\end{figure}
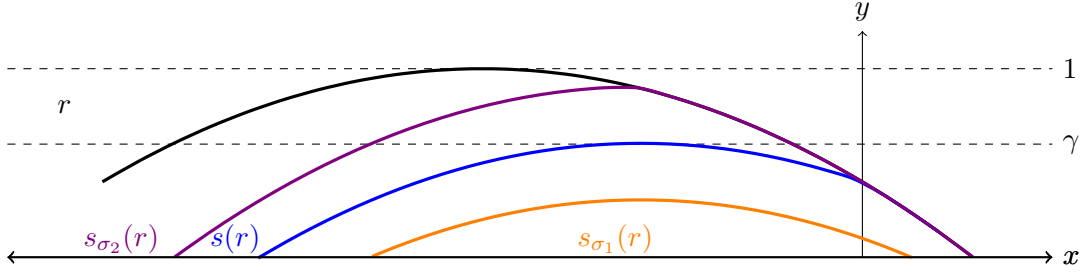

Let $\Rr_N = (R_{N,j})_{j \in [N]}$ denote a random set of points. We think of $\Rr_N$ as the point set after the reproduction step and
we wish to sample $N^\gamma$  particles $\mathcal{S}_N = (S_{N,i})_{i \in[N^\gamma]}$ from this set without replacement, where each particle $R_{N,j}$ has relative weight 
\begin{equation}\label{eq:weights}
w_{N,j} := e^{\beta_N R_{N,j}},
\end{equation}
and $\beta_N$ is given by~\eqref{eq:betaNrN}.
We define the sequences of measures
\[ M_N^R := \sum_{j=1}^N \delta_{R_{N,j}/c_N},\qquad \text{and} \qquad M_N^S := \sum_{i=1}^{N^\gamma} \delta_{S_{N,i}/c_N},\]
and assume that $(M_N^R)_{N\in\N}$ has limiting log-profile $r \in \mathcal{D}$. Our goal in this section is to prove the following:

\begin{lem}\label{lem:Sampling}
Assume that $\sigma^*:= \sigma^\ast(r)$
is the unique value of $\sigma$ such that $\sup s_{\sigma}(r)= \gamma$.
For all $a,b \in \R\setminus D_{r}$ such that $\sup_{x \in (a,b]}\bar{s}(r)(x) \neq 0$,
\begin{equation}\label{eq:SamplingLemmaEq} \frac{\log M_N^S(a,b)}{\log N} \xrightarrow{\P} \sup_{x \in (a,b]} s(r)(x).
\end{equation}
\end{lem}

The conditions imposed on $\sigma^\ast$ and $\bar{s}$ in Lemma \ref{lem:Sampling} are required to guarantee the existence of a limiting log-profile for $M_N^S$. We explore what can fail if these conditions are not satisfied in Section \ref{sec:Examples}.

\bigskip

To prove Lemma \ref{lem:Sampling}, we perform our sampling without replacement using exponential clocks: a clock associated to a particle at location $R_{N,j}$ is an exponential random variable $\tau_{N,j}\sim\Exp(w_{N,j})$ for $w_{N,j}$ defined in \eqref{eq:weights}. Then the particles corresponding to the $N^\gamma$ smallest values of the $\tau_{N,j}$'s (i.e.~the $N^\gamma$ particles whose clocks rang the fastest) are selected. In analyzing the measures $M^S_N$ we will first investigate the number of clocks that have `gone off' by a certain time. For any $\sigma \in \R$, define 
\[ M^S_{N, \sigma} := \sum_{j=1}^{N} \delta_{R_{N,j}/c_N}\1_{\{\tau_{N,j} \leq N^{-\beta\sigma}\}}.\]

\begin{lem}\label{prop:clocks}
Suppose the sequence of point measures $(M^R_N)_{N \in \N}$  has limiting log-profile $r \in \mathcal{D}$. For all $a,b \in \R\setminus D_{r}$, if $\sup_{x \in (a,b]}\bar{s}_\sigma(r)(x) \neq 0$, then 
\[ \frac{\log M_{N,\sigma}^S(a,b)}{\log N} \xrightarrow{\P} \sup_{x \in (a,b]} s_\sigma(r)(x).\]
\end{lem}

\begin{proof}
Throughout we let $\bar{s}_\sigma := \bar{s}_\sigma(r)$ and $s_\sigma := s_\sigma(r)$. For a particle at location $R_{N,j} = x c_N$, we have $w_{N,j}= N^{\beta x}$, and
\begin{equation}\label{eq:tauprob}
\P\left(\tau_{N,j} \leq N^{-\beta\sigma}\,|\,\sigma({\cal R}_N)\right) = 1 - \exp(-w_{N,j} N^{-\beta\sigma}) = 1 - \exp\left(-N^{\beta (x - \sigma)}\right).
\end{equation}
For all $a,b \in \R$, if we let $B_N(a,b) :=  \mathbb{E}[M^S_{N,\sigma}(a,b) | \sigma({\cal R}_N)]$, then 
\begin{align}\label{eq: EMSsig}
B_N(a,b) = \!\!\sum_{R_{N,j} \in (ac_N,bc_N]}\!\! \P\left(\tau_{N,j} \leq N^{-\beta\sigma}| \sigma({\cal R}_N)\right) = \int_{(a,b]} \left(1 - \exp\left(-N^{\beta (x - \sigma)}\right)\right) dM^R_N(x).
\end{align}
It is easy to see that
\[ \frac{\log\left(1- \exp(-N^{\beta(x - \sigma)})\right)}{\log N} \xrightarrow{N \to \infty} \beta(x-\sigma)_-\]
uniformly in $x$. 
From Remark \ref{rmk:unifIntLemma}, we may apply Lemma \ref{lem:integral}, and so for all $a,b \in \R\setminus D_r$,
\begin{equation}\label{eq:convAsigma}
B_N(a,b) \xrightarrow{\P} \sup_{x \in (a,b]} \bar{s}_\sigma(x), \text{ as }N\to\infty.
\end{equation}

Now we apply Lemma \ref{lem:concentration} with $\ell_N = N$, $a_j^N = (1-\exp(-w_{N,j}N^{-\beta\sigma}))\1_{ \{ R_{N,j} \in (ac_N, bc_n] \} }$, $\xi_{j}^N=\1_{\{\tau_{N,j}\leq N^{-\beta\sigma}\}}$, and $\F_N = \sigma({\cal R}_N)$. By \eqref{eq:convAsigma} and Lemma \ref{lem:concentration}(a), we conclude that
\[ \frac{\log M^S_{N,\sigma}(a,b)}{\log N} \xrightarrow{\P} \sup_{x \in (a,b]}\bar{s}_{\sigma}(x) = \sup_{x \in(a,b]} s_\sigma(x), \text{ as $N\to\infty$},\]
if $\sup_{x \in (a,b]}\bar{s}_\sigma(x)> 0$, and by Lemma \ref{lem:concentration}(b), we see that 
\[ M^S_{N,\sigma} \xrightarrow{\P} 0, \text{ as $N\to\infty$},\]
if $\sup_{x \in(a,b]}\bar{s}_{\sigma}(x) < 0$, which concludes the proof. 
\end{proof}

\bigskip

For all $N\in\N$, let us write
\begin{equation}\label{eq:sigma^*-d}
\sigma^\ast_N := \sigma^\ast_N(\Rr_N) = \inf\Big\{ \sigma :\: M^S_{N,\sigma}(\R) \leq  N^\gamma\Big\}
\end{equation}
as an analog to $\sigma^\ast(r)$ from \eqref{eq: sigma*}.
Then as discussed above,
$M^S_{N, \sigma^\ast_N}(a,b) = M^S_N(a,b)$. 

\begin{proof}[Proof of Lemma~\ref{lem:Sampling}]
Let $s_\sigma := s_\sigma(r)$ and $s := s(r)$. Since we assume $\sigma^*$ is the unique value for which $\sup s_{\sigma^\ast}=\gamma$, then for all $\eps > 0$ and by Lemma~\ref{prop:clocks}, 
\[ \frac{\log M^S_{N,\sigma^{*}-\eps}(-\infty,\infty)}{\log N} \xrightarrow{\mathbb P} \sup s_{\sigma^{*}-\eps}> \gamma\]
and so from the definition of $\sigma^\ast_N$, as $N \to \infty$, 
\begin{align*}
\p{\sigma^{*}-\eps > \sigma^\ast_N} &\leq \p{M^S_{N,\sigma^*-\eps}(-\infty,\infty) \leq M^S_{N}(-\infty,\infty) = N^\gamma } 
\\ &
\leq \p{\frac{\log M^S_{N,\sigma^{*}-\eps}(-\infty,\infty)}{\log N} \leq \gamma} \to 0.
\end{align*}
By a symmetric argument, $\P(\sigma^\ast_N > \sigma^\ast + \eps) \xrightarrow{N \to \infty} 0$, and so $\sigma^\ast_N \xrightarrow{\P} \sigma^{*}$ as $N\to\infty$. As a consequence, for all $a,b \in \R$, 
\begin{equation}\label{eq:inclusionprob}
\P\left(M^S_{N, \sigma^\ast+\eps}(a,b) \leq M^S_{N}(a,b) \leq M^S_{N, \sigma^\ast-\eps}(a,b)\right) \geq \p{ \sigma^\ast-\eps \leq \sigma^\ast_N \leq \sigma^\ast+\eps  } 
 \xrightarrow{N \to \infty} 1.
\end{equation}

Note that for any $x, \sigma \in \R$ and any $\eps >0$, $|\bar{s}_\sigma(x) - \bar{s}_{\sigma\pm \eps}(x)| \leq \eps$, and so it is evident that for all $a,b \in \R\setminus D_{r}$ such that $\sup_{x \in (a,b]} \bar{s}(x) \neq 0$, the function $\sigma \mapsto \sup_{x \in (a,b]} s_\sigma(x)$ is continuous around $\sigma^*$. Then applying Lemma \ref{prop:clocks} and letting $\eps \to 0$ in \eqref{eq:inclusionprob} completes the proof. 
\end{proof}

\subsection{Proof of Theorem \ref{thm:main}}\label{sec:ProofMain}

Recall that $M_N^t$ is the point measure associated with $\mathcal{Z}_N^t = (Z^t_{N,i})_{i\in[N^\gamma]}$, see (\ref{def-MN}). For $t\in\N$, we introduce $M^{R,t}_{N}$ and $M^{S,t}_{N} = M_N^{t+1}$ as the analog point measures after reproduction and selection (between times $t$ and $t+1$) as defined in Sections \ref{sec:Reproduction} and \ref{sec:Sampling}.
We also recall the definitions of the operators, $\bar{r},r$ defined in \eqref{eq: Rg} and $\bar{s},s$ defined in \eqref{eq: S}. First, we prove a more general theorem with technical conditions, which we then apply to prove Theorem~\ref{thm:main}.

\medskip

\begin{thm}\label{thm:mainGeneral}
Assume that $(M_N^0)_{N \in \N}$ has limiting log-profile $g^0\in{\mathcal D}$ and let  $(g^t)_{t\geq0}$ be the dynamics as defined in Theorem \ref{thm:main}. 
Suppose further that 
\begin{enumerate}
\item[(i)] $g^t\in{\cal D}$,
\item[(ii)] for all $t \geq 0$,
\begin{align}
\forall a,b \in \R\setminus D_{g^{t+1}}, \ \ \sup_{x \in (a,b]} \bar{s}(r^t)(x) &\neq 0, \label{eq:noZeroConds}
\end{align}
\item[(iii)]  $\sigma^t$ is the unique $\sigma$ for which $\sup s_\sigma(r^t) = \gamma$. 
\end{enumerate}
Then for all $t \in \N$, the sequences of point measures $(M_N^t)_{N \in \N}$ have limiting log-profile $g^t$.
\end{thm}

\begin{proof}
Let $t\in\N$. Suppose the result is true for all $0\leq n\leq t$. 
By Lemma \ref{lem:Reproduction}, $(M_N^{R,t})_{N \in \N}$ has limiting log-profile $r^t$. Let $a,b \in \R \setminus D_{g^{t+1}}$, then \eqref{eq:noZeroConds} guarantees that $\sup_{x \in (a,b]} \bar{s}(r^t)(x) \neq 0$. For any $\eps >0$, we may find $a^-,b^+\in \R\setminus D_{r^t}$ such that $a^- \leq a$, $b \leq b^+$, 
\begin{equation}\label{eq:boundSups}
\sup_{x \in (a^-, b^+]} \bar{s}(r^t)(x) \neq 0, \qquad \text{and} \qquad \sup_{x \in (a^-, b^+]} g^{t+1}(x) -\sup_{x \in (a, b]} g^{t+1}(x) \leq \eps.
\end{equation}
We assumed that $\sigma^t$ is the unique value for which $\sup s_{\sigma}(r^t) = \gamma$, and so from \eqref{eq:boundSups} and Lemma \ref{lem:Sampling}, 
\[\frac{\log M_N^{t+1}(a,b)}{\log N}\leq \frac{\log M_N^{t+1}(a^-,b^+)}{\log N} \xrightarrow{\P} \sup_{x \in (a^-, b^+]} g^{t+1}(x).\]
A symmetric argument holds for a lower bound, and letting $\eps \to 0$, we conclude that 
\[ \frac{\log M_N^{t+1}(a,b)}{\log N} \xrightarrow{\P} \sup_{x \in (a,b]}g^{t+1}(x).\]
The above holds for all $a,b \in D_{g^{t+1}}$, and so the sequence of point measures $(M_N^{t+1})_{N \in \N}$ has limiting log-profile $g^{t+1}$. Recursively applying the above argument completes the proof. 
\end{proof}

In order to prove Theorem \ref{thm:main} we will need one more technical lemma. Recall the definition of the function class $\cal C$ from Definition~\ref{def: D}.

\begin{lem}\label{prop:WellDefined}
Suppose $h$ from~\eqref{eq:defh} is concave ($\rho\geq1$).  If $g \in \mathcal{C}$ and $\sup g = \gamma$, then $\bar{r}(g), \bar{s}(r(g))$ are concave and 
\[r(g), \,\, s(r(g)) \in \mathcal{C}.\]
\end{lem}

\begin{proof}
Set $\bar{r} := \bar{r}(g)$ and $r := r(g)$. For any $x, x' \in \R$, by using the concavity of $g$ and $h$, 
\begin{align*}
\bar{r}(g)(\alpha x  + (1-\alpha)x') &= 1 - \gamma + \sup_{z \in \R}\big\{ g(z) + h(\alpha x + (1-\alpha) x' - z)\big\} \\
&= 1 - \gamma + \sup_{z,z' \in \R} \big\{ g(\alpha z + (1-\alpha)z') + h(\alpha(x-z) + (1-\alpha)(x' - z')\big\} \\
& \geq 1- \gamma + \alpha\sup_{z \in \R}\big\{ g(z) + h(x - z)\} + (1-\alpha) \sup_{z' \in \R}\big\{ g(z') + h(x' - z')\big\} \\
&= \alpha\bar{r}(x) + (1-\alpha)\bar{r}(x'),
\end{align*}
and so $\bar{r}$ is concave.
Since $\pi$ is non-decreasing and concave, then $r = \pi(\bar{r})$ is also concave. Finally, the function $x\mapsto\beta(x - \sigma^\ast(r))_-$ is concave, and so $\bar{s}(r)(x) = r(g)(x) + \beta(x - \sigma^\ast(r))_-$ is also concave, while once more, the composition $\pi(\bar{s}(r)) = s(r)$ is also concave.  
\end{proof}

\begin{proof}[Proof of Theorem \ref{thm:main}]
By Lemma \ref{prop:WellDefined}, for all $t \geq 0$, we have that $r^t,g^t \in \mathcal{C} \subset \mathcal{D}$. To apply Theorem \ref{thm:mainGeneral}, we then only need to show that \eqref{eq:noZeroConds} holds and that there is a unique $\sigma^t$  for which $\sup s_{\sigma_t}(r^t) = \gamma$. For $t\in\N$ let $L^t:=L(g^t)$ and $U^t:=U(g^t)$ from Definition~\ref{def: D}.\\

From Lemma~\ref{prop:WellDefined}, $\bar{s}(r^t)$ and $g^{t+1}$ are concave, and are therefore continuous on the interior of their support, so $D_{g^{t+1}} = \{L^{t+1}, U^{t+1}\}$. By the concavity of $\bar{s}(r^{t})$, we also see that $\bar{s}(r^{t})(x) < 0$ for all $x \in \R \setminus [L^{t+1}, U^{t+1}]$, and since $\bar{s}(r^t) > 0$ on $(L^{t+1}, U^{t+1})$ we see that 
\begin{equation*}
a,b \in \R\setminus\{L^{t+1}, U^{t+1}\} \Longrightarrow \sup_{x \in (a,b]} \bar{s}(r^t)(x) \neq 0, 
\end{equation*}
and so \eqref{eq:noZeroConds} holds. \\

We now prove the uniqueness of $\sigma^t$. Suppose otherwise that there exists two values $\sigma_1 < \sigma_2$ such that $\sup s_{\sigma_1}(r^t) = \sup s_{\sigma_2}(r^t) = \gamma$. By the definition of the operator $s_{\sigma_1}$, 
\[\gamma \geq \sup_{x \in (\sigma_1, \infty)}s_{\sigma_1}(r^t)(x) = \sup_{x \in (\sigma_1, \infty)} r^t(x).\]
Since $r^t$ is concave and $\sup r^t = 1 > \gamma$, the above implies that $r^t$ is strictly decreasing for $x \in (\sigma_1, \infty)\cap \text{Supp}(r^t)$. Therefore, if we let $\hat{\sigma}$ be such that $\sigma_1 < \hat{\sigma} < \sigma_2$, then
\begin{equation}\label{eq:sigmaUnique1} \sup_{x \in (\hat{\sigma}, \infty)}s_{\sigma_2}(r^t)(x) \leq \sup_{x \in (\hat{\sigma}, \infty)} r^t(x) < \gamma
\end{equation}
and
\begin{equation}\label{eq:sigmaUnique2} \sup_{x \in (\sigma_1, \hat{\sigma}]}s_{\sigma_2}(r^t)(x) = \sup_{x \in (\sigma_1, \hat{\sigma}]}\pi\left(r^t(x) + \beta (x - \sigma_2)\right) \leq \sup_{x \in (\sigma_1,\hat{\sigma}]}r^t(x) + \beta(\hat{\sigma} - \sigma_2) < \gamma.
\end{equation}
Now let $x \in (-\infty, \sigma_1]$. Since $\beta(x - \sigma_2) < \beta(x - \sigma_1)$, then $s_{\sigma_2}(r^t)(x) < s_{\sigma_1}(r^t)(x) \leq \gamma$ on $\text{Supp}(s_{\sigma_2}(r^t))$, while $s_{\sigma_2}(r^t)(x) = -\infty$ otherwise. Along with \eqref{eq:sigmaUnique1} and \eqref{eq:sigmaUnique2}, we see that $\sup s_{\sigma_2}(r^t) < \gamma$, and a contradiction is reached. This completes  the proof of Theorem~\ref{thm:mainGeneral}.
\end{proof}

\subsubsection{Example with no limiting log-profile after selection}\label{sec:Examples}

\begin{ex}\label{ex:Sup0} 

We now construct a counterexample where 
${\cal S}_N$ -- the point measure after selection step -- does not admit a limiting log-profile coinciding with the prediction of Theorem \ref{thm:main}. To do so, we consider a population whose log-profile after selection has an isolated $0$, so that it becomes impossible to deduce the presence or absence of particles from the log-profile alone.  

\medskip

Let $\rho \in (0,1)$, 
and $X$ be a random variable with probability density function $f_X = c(\rho)\exp\left(-|x|^\rho\right)$ for some normalizing constant $c(\rho)$, so that $h(x) = -|x|^\rho$. We see that $c_N \approx (\ln N)^{1/\rho}$. Starting with $N^\gamma$ particles at $0$, the population after reproduction $\Rr_N$ is given by $N$ i.i.d. copies of $X$. In turn, the  limiting log-profile is given by $r(x) := \pi(1 + h(x))$.

Suppose $\beta,\gamma$ satisfy $\beta = 1 - \gamma  \in (\rho,1)$. An easy computation yields $\sigma^\ast := \sigma^\ast(r) = 1$, and there exist $a< b < 1$ such that
$$
 {\bar s}_{\sigma^*}(x) = r(g)(x) + \beta(x - \sigma)_- = 1-x^\rho + \beta(x-1) 
$$
is greater or equal to $0$ on $[a,b]\cup \{1\}$ and negative otherwise, see Figure \ref{fig:no-profile}. Since $\pi(0)=0$, 
\begin{equation}
\label{eq:supp-g}
\mbox{Supp}\left(s(r)\right) \ = \ [a,b]\cup \{1\}.
\end{equation}

Let us now show that the limiting log-profile (if there exists any) after  selection can not coincide with $s(r)$. We only provide an outline of the proof and leave the details to the reader. 
Let ${\cal S}_N$ be the point measure after 
the  selection step.
The key point is that for $\eps>0$ small enough, the random variable  
$$
\#\{ x \in{\cal S}_N  \in c_N(1-\eps,1+\eps) \} 
$$
is tight and any subsequential limit has a positive probability 
to be $0$. If the limiting r.v. is $0$, the limiting log-profile 
must be $-\infty$ at $1$, in contradiction with (\ref{eq:supp-g}).

\begin{figure}[h!]
\centering

\begin{tikzpicture}[xscale = 5, yscale = 2, domain = 1:1]
\draw[<->,thick] (-1.2,0) -- (1.35,0) node[right]{$x$};
\draw[->] (0,0) -- (0,1.2) node[above]{$y$};
\draw[dashed] (-1.2,0.6) -- (1.35, 0.6) node[right]{$\gamma$};
\draw[dashed] (-1.2, 1) -- (1.35, 1) node[right]{$1$};
\draw[dashed] (1,1) -- (1,-0.05) node[below]{$1$};
\draw (-0.0987,0.05) -- (-0.0987,-0.05) node[below]{$a$};
\draw (0.2285,0.05) -- (0.2285,-0.05) node[below]{$b$};

\draw[black, smooth, samples=500, very thick, domain = 0:1, variable = \x] plot({\x},{1-(\x)^(1/4)});
\draw[black, smooth, samples=500,very thick, domain = -1:0, variable = \x] plot({\x},{1-(-\x)^(1/4)});

\draw[orange, smooth, samples=200,very thick, domain = 0:0.2285, variable = \x] plot({\x},{1-(\x)^(1/4) + 0.4*(\x-1)});
\draw[orange, dashed, smooth,very thick, domain = 0.2285:1, variable = \x] plot({\x},{1-(\x)^(1/4) + 0.4*(\x-1)});
\draw[orange, smooth, samples=200,very thick, domain = -0.0987:0, variable = \x] plot({\x},{1-(-\x)^(1/4)+ 0.4*(\x-1)});
\draw[orange, dashed, smooth,very thick, domain = -1:-0.0987, variable = \x] plot({\x},{1-(-\x)^(1/4)+ 0.4*(\x-1)});

\node at (-0.5,0.45) {$r$};
\node at (-0.5,-0.25) {\color{orange} $\bar{s}(r)$};
\node at (1,0) {\color{orange} \large $\bullet$};

\end{tikzpicture}

\caption{After applying the operator $\bar{s}$ to $r(x) = \pi(1-|x|^{1/4})$ (in black) with $\sigma^\ast(r) = 1$, the function $\bar{s}(r)$ (in orange) has a local maximum $\bar{s}(r)(1) = 0$, and $s(r)$ (in solid orange) is not the limiting log-profile of the process after selection.}
\label{fig:no-profile}
\end{figure}

\end{ex}

\section{Piecewise linear log-profiles}\label{sect:pwllogprof}

From now on, we consider the special class of limiting log-profiles when the function $h$ has the form 
\begin{equation}\label{eq:hCond3}
h(x) = \begin{cases}
-x & x \geq 0 \\
c_-x & x < 0,
\end{cases}
\end{equation}
where $c_- \in \R_+\cup \{+\infty\}$. 
Define the parameters 
\begin{equation}\label{eq:defK}
k := k(\beta) = \left\{\begin{array}{cc}
\lfloor\frac{1}{\beta}\rfloor & \mbox{if $\frac{1}{\beta}\notin\N$} \\ \frac{1}{\beta}-1 & \mbox{otherwise} 
\end{array}\right., \ \ \  K := K(\beta) = \left\lfloor \frac{1+c_-}{\beta}\right\rfloor + 2,
\end{equation}
where we take the convention that $K=+\infty$ if $c_- = \infty$.
 Since $c_- >0$, we have that $K \geq k+2$, while we also note that when $c_- \neq +\infty$ then 
\begin{equation}\label{eq:BoundsC-}
(K-1)\beta - 1 \, \leq \, c_- + \beta \, \leq\, K\beta - 1.
\end{equation}

Throughout the rest of the article, we will use the convention that $\R^{K+1} := \R^\N$ when $K = +\infty$. We are interested in the following class of functions:

\begin{defn}\label{def:T}
Suppose $\bm{x} \in \R^{K+1}$ is a non-increasing vector, that is,
\begin{equation}\label{eq:defTCond}
x_0 \geq x_1 \geq x_2 \geq \cdots \geq x_{K}. 
\end{equation} 
Define the function $f_{\bm{x}} \in \mathcal{C}$ by 
\begin{equation}\label{eq:defT}
f_{\bm{x}}(x) := \pi\left(\int_x^{x_0} (-c_- - \beta)\vee \left(1-\beta\sum_{j=1}^{K}\1_{\{z < x_j\}}(z)\right)dz\right).
\end{equation}
\end{defn}

\begin{figure}[h!]
\centering
\begin{tikzpicture}[scale=1.5]
\draw[<->,  thick] (0.4, 0) -- (-5.9, 0);
\draw (-0,0.1) -- (-0,-0.1) node[below]{$x_0$} ;

\draw[very thick, red] (0,0) --(-0.5,0.5) -- (-1.1,0.968) -- (-1.6, 1.248) -- (-2.2,1.452) -- (-2.6, 1.5)-- (-3.1, 1.45) -- (-3.5, 1.322)-- (-4.1,0.998) -- (-5.486,0);

\draw[dashed] (-0.5,0.5) -- (-0.5,-0.1) node[below]{$x_1$};
\draw[dashed] (-1.1,0.968) -- (-1.1,-0.1) node[below]{$x_2$};
\draw[dashed] (-1.6,1.248) -- (-1.6,-0.1) node[below]{$x_3$};
\draw[dashed] (-2.2,1.452) -- (-2.2,-0.1) node[below]{$x_4$};
\draw[dashed] (-2.6,1.5) -- (-2.6,-0.1) node[below]{$x_5$};
\draw[dashed] (-3.1,1.45) -- (-3.1,-0.1) node[below]{$x_6$};
\draw[dashed] (-3.5,1.322) -- (-3.5,-0.1) node[below]{$x_7$};
\draw[dashed] (-4.1,0.998) -- (-4.1,-0.1) node[below]{$x_8$};

\node at (-5.5,0.75) {\Large $f_{\bm{x}}:$};

\end{tikzpicture}
\begin{tikzpicture}[scale=1.5]
\draw[<->,  thick] (0.4, 0) -- (-5.9, 0);
\draw (-0,0.1) -- (-0,-0.1) node[below]{$x_0$} ;

\draw[very thick, red] (0,0) --(-0.5,0.5) -- (-1.1,0.968) -- (-1.6, 1.248) -- (-2.2,1.452) -- (-2.6, 1.5)-- (-3.1, 1.45) -- (-3.5, 1.322)-- (-4.1, 0.998)-- (-4.5,0.694) -- (-5.1, 0.106) -- (-5.18833,0);

\draw[dashed] (-0.5,0.5) -- (-0.5,-0.1) node[below]{$x_1$};
\draw[dashed] (-1.1,0.968) -- (-1.1,-0.1) node[below]{$x_2$};
\draw[dashed] (-1.6,1.248) -- (-1.6,-0.1) node[below]{$x_3$};
\draw[dashed] (-2.2,1.452) -- (-2.2,-0.1) node[below]{$x_4$};
\draw[dashed] (-2.6,1.5) -- (-2.6,-0.1) node[below]{$x_5$};
\draw[dashed] (-3.1,1.45) -- (-3.1,-0.1) node[below]{$x_6$};
\draw[dashed] (-3.5,1.322) -- (-3.5,-0.1) node[below]{$x_7$};
\draw[dashed] (-4.1,0.998) -- (-4.1,-0.1) node[below]{$x_8$};
\draw[dashed] (-4.5,0.694) -- (-4.5,-0.1) node[below]{$x_9$};
\draw[dashed] (-5.1,0.106) -- (-5.1,-0.1) node[below]{$x_{10}$};
\draw[dashed] (-5.5,0.1) -- (-5.5,-0.1) node[below] {$x_{11}$};

\node at (-5.5,0.75) {\Large $f_{\bm{x}}:$};

\end{tikzpicture}
\caption{Two examples of functions $f_{\bm{x}}$. In the first plot, $c_- = 0.5$ and so $k=4, K=8$, and the function $f_{\bm{x}}$ has slopes $-1, -(1-\beta), \ldots, -(1-7\beta)$ and $c_- + \beta$. In the second plot, $c_- = \infty$ and so $k=4, K=\infty$, and the function $f_{\bm{x}}$ may have slope $-1, -(1-\beta), -(1-2\beta), \ldots$.}
\end{figure}
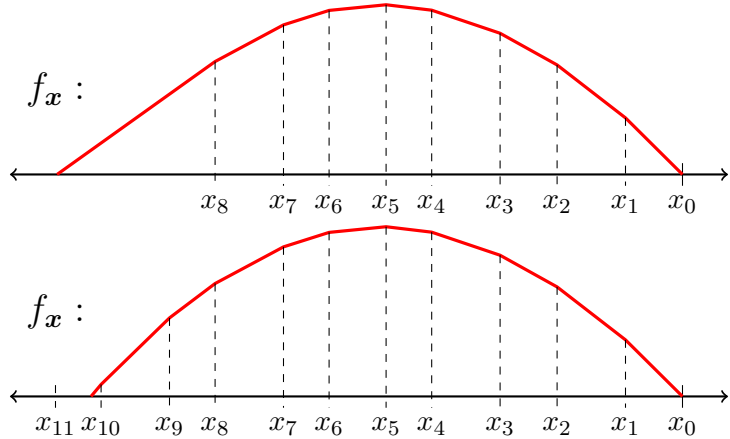

The function $f_{\bm{x}}$ consists of the non-negative parts of the linear functions with possible slopes 
\[c_-+\beta, \ \ \mbox{on $(-\infty,x_{K}]$},  \text{ and } -(1-j\beta) \ \ \mbox{on $[x_{j+1},x_j]$} \ \ \text{ for  $j \in \{0\}\cup[K-1]$} .\] 

 Because of (\ref{eq:BoundsC-}), the function $f_{\bm{x}}$ is concave on its support for all non-increasing vectors $\bm{x} \in \R^{K+1}$, and so $f_{\bm{x}} \in \mathcal{C}$ and $x_0 = U(f_{\bm{x}})$ is the upper edge of $f_{\bm{x}}$. We note from \eqref{eq:defT} that $f_{\bm{x}}$ is not necessarily defined by a unique vector $\bm{x}$, though as we state in the following lemma, we may determine when two vectors $\bm{x}$ and $\bm{x}'$ define the same function $f_{\bm{x}} = f_{\bm{x}'}$.

\begin{lem}\label{lem:EqClassf}
Let $\bm{x} \in \R^{K+1}$ be non-increasing. For any non-increasing $\bm{x}' \in \R^{K+1}$,\, $f_{\bm{x}'} = f_{\bm{x}}$ if and only if the following are satisfied:
\[ 
\begin{array}{cc}
x_j' = x_j & \text{ for all } x_j > L(f_{\bm{x}}), \\
x_j' \leq L(f_{\bm{x}}) & \text{ for all } x_j \leq L(f_{\bm{x}}).
\end{array}
\]
\end{lem}

\begin{proof}
The result follows immediately from the definition of $f_{\bm{x}}$ and $f_{\bm{x}'}$ in \eqref{eq:defT}.
\end{proof}

\begin{rmk}\label{rmk:supf=xk+1}
 The definition of $k$ entails that the function $f_{\bm{x}}$ is decreasing for $x > x_{k+1}$ and non-decreasing on $x \leq x_{k+1}$, and so 
\begin{equation}\label{eq:TSup}
\sup f_{\bm{x}}(x) = f_{\bm{x}}(x_{k+1}) = \sum_{j=0}^k(1-j\beta)(x_j - x_{j+1}).
\end{equation}
Therefore $L(f_{\bm{x}}) \leq x_{k+1}$ and so by Lemma \ref{lem:EqClassf},
\begin{equation}\label{eq:Samexk+1}
f_{\bm{x}} = f_{\bm{x}'}\,\,\, \Longrightarrow \,\,\, x_0=x_0',\,\, x_1=x_1',\,\, \ldots,\,\, x_{k+1}=x_{k+1}'.
\end{equation}
\end{rmk}

In Section \ref{sec:Toperators}, we introduce a vector operator $\psi_\sigma$ and prove that under some conditions on $\sigma$, we may translate the operator $s_\sigma \circ r$ acting on $f_{\bm{x}}$ as a transformation $\bm{x} \mapsto \psi_{\sigma}(\bm{x})$ so that $s_\sigma \circ r (f_{\bm{x}}) = f_{\psi_\sigma}(\bm{x})$. The main result of this section is Theorem \ref{thm:pwl}, which is contained in Section \ref{sec:pwl}. We show that for any $g \in \mathcal{D}$, after finitely many applications of $s \circ r$, the resulting function is a piecewise linear function belonging to the following class of functions:

\begin{defn}\label{def:Tgamma}
A function $g$ belongs to $\mathcal{T}$ if and only if $g = f_{\bm{x}}$ for some non-increasing vector $\bm{x} \in \R^{K+1}$ and $\sup g = \gamma$.
\end{defn}

We will assume throughout this section that $g^0\in{\cal D}$ with $\sup g^0=\gamma$. As in the dynamics of Theorem \ref{thm:main} we recursively define for all $t \geq 0$,
\begin{equation}\label{eq:sigmatDef}
\sigma^t:= \sigma^\ast(r(g^t)), \qquad g^{t+1} := s \circ r (g^t) = s_{\sigma^t} \circ r (g^t),
\end{equation}
and let $L^t := L(g^t), \, U^t := U(g^t)$ be the lower and upper edges of $g^t$.

\begin{rmk}
    A careful reading of the proofs in this section will reveal that our arguments are valid for a much broader class of functions than ${\cal D}$. Namely, it is enough to assume that $g^0: \R \to \R_+\cup\{-\infty\}$ have bounded support with $\sup g^0 = \gamma$.
\end{rmk}

\subsection{Operators on piecewise linear log-profiles}\label{sec:Toperators}

Starting from $f_{\bm{x}}$ for non-increasing $\bm{x} \in \R^{K+1}$, we show that $r(f_{\bm{x}})$ is a piecewise linear function, and that for a range of values of $\sigma$, $s_\sigma \circ r (f_{\bm{x}}) = f_{\bm{y}}$ for some non-increasing vector $\bm{y} \in \R^{K+1}$. We evaluate how the operator $s_\sigma \circ r$ acts on $f_{\bm{x}}$ by evaluating the affects on the vector $\bm{x}$. More precisely, we introduce in Lemma \ref{lem:operators} below a vector operator $\psi_{\sigma}$ and prove that under some conditions on $\sigma$, $s_\sigma \circ r(f_{\bm{x}}) = f_{\psi_\sigma(\bm{x})}$.

\bigskip

We start with a description of how the operator $r$ behaves when $h$ satisfies \eqref{eq:hCond3}. For any $g \in \mathcal{D}$, we bound $r(g)$ from above and show that $r(g)$ is a line with slope $-1$ on the interval $[U(g), U(r(g))]$. We also provide a direct computation for $r(f_{\bm{x}})$. 

\begin{lem}\label{lem:repLin}
For $g \in \mathcal{D}$, define $\omega(g):= \sup_{z \in \R}(g(z) + z)$. The following hold:
\begin{itemize}
\item[(i)] For all $x \in \R$, 
\[ r(g)(x) \leq \pi\left(1 -\gamma + \omega(g) - x\right),\]
with equality for all $x \geq U(g)$, that is, $r(g)$ is a line with slope $-1$ on $[U(g), U(r(g))]$.
\item[(ii)] For $\bm{x} \in \R^{K+1}$ non-increasing, let $L' := L(f_{\bm{x}})\vee x_{K-1}$ when $K\in \N$ and $L' = L(f_{\bm{x}})$ when $K = \infty$. Then $\omega(f_{\bm{x}}) = U(f_{\bm{x}}) = x_0$ and \begin{equation}\label{eq:rForT} 
r(f_{\bm{x}})(x) = \begin{cases}
\pi(1-\gamma - (x-x_0)) & x > x_0,\\
1-\gamma + f_{\bm{x}}(x) & x \in [L', x_0],\\
\pi\left(1 - \gamma + f_{\bm{x}}(L') - c_-(L' - x)\right) & x < L'.
\end{cases}
\end{equation}
\end{itemize}
\end{lem}

\begin{proof}
Let $U=U(g)$ and $L=L(g)$. Note that $h(x-z) \leq z-x$ for all $z,x \in \R$, and so immediately
\begin{equation}\label{eq:rIneq} 1 - \gamma + \sup_{z \in \R}(g(z) + h(x-z)) \leq 1- \gamma + \sup_{z \in \R}(g(z) +z - x) = 1 - \gamma + \omega(g) - x.
\end{equation}
Applying $\pi$ to both sides proves the inequality of (i). For the equality, since $g(z) = -\infty$ for all $z > U$, then $\sup_{z \in \R}(g(z) + h(x-z)) = \sup_{z \in \R}(g(z) + z-x)$ for all $x \geq U$. Thus for all $x \geq U$, we replace the inequality of \eqref{eq:rIneq} with equality, and applying $\pi$ to both sides once more completes the proof of (i).  \\

For any non-increasing sequence $\bm{x} \in \R^{K+1}$, the slopes of $f_{\bm{x}}$ are all at least $-1$. As a consequence, we see that $\omega(f_{\bm{x}}) = U(f_{\bm{x}}) = x_0$, and the first case of \eqref{eq:rForT} follows from (i). Let $f:=f_{\bm{x}}$ and suppose $x \in [L', x_0]$. If $x \in [x_{j+1}, x_j]$ for $j=0,1,\ldots,K-2$, then by the definition of $f_{\bm{x}}$, 
\[f(y) \leq f(x) -(1-j\beta)(y-x) \text{ for all } y \in \R,\]
and since $-1 \leq -(1-j\beta) \leq c_-$ for these choices of $j$, $f(z) + h(x-z)$
reaches its maximum at $x$ and 
\[ r(f)(x) = 1 - \gamma + \sup_{z \in \R}(f(z) + h(x-z)) = 1 - \gamma + f(x).\]
Next for $x < L'$ and recalling \eqref{eq:BoundsC-}, we see that $f(x) \leq f(L') -c_-(L' - x)$, and so 
\[ r(f)(x) = \pi\left(1-\gamma + \sup_{z \in \R}(f(z) + h(x-z))\right) = \pi\left(1-\gamma + f(L') - c_-(L' - x)\right),\]
completing the proof.
\end{proof}

\medskip

We now focus on introducing the vector operator $\psi_\sigma$ that satisfies $s_\sigma \circ r(f_{\bm{x}}) = f_{\psi_\sigma(\bm{x})}$. We have already seen above in Lemma \ref{lem:repLin} that $r(f_{\bm{x}})$ is once more a piecewise linear function, while the addition of $\beta(x - \sigma)_-$ changes the slopes of $r(f_{\bm{x}})$ by $\beta$ for all $x < \sigma$ and leaves the slopes unchanged for $x \geq \sigma$. Therefore, defining $\psi_\sigma$ is a matter of finding in which linear interval of $r(f_{\bm{x}})$ lies $\sigma$, or if $\sigma$ lies outside the support of $r(f_{\bm{x}})$. 

Defining $\psi_\sigma$ for all possible $\sigma$ would be quite cumbersome. We therefore restrict ourselves to values of $\sigma$ that we will encounter in later proofs. For $g \in \mathcal{D}$, define
\begin{equation}\label{eq:sigmaRange}
a_{-}(g):=L(g) + \frac{1-\gamma}{1\wedge \beta} \qquad \text{and} \qquad
a_{+}(g):= \omega(g) - \gamma + \frac{1-\gamma}{1\wedge \beta}.
\end{equation} 
The values above are not arbitrary, indeed we will see later (in Lemma \ref{lem:sigmatBounds}) that $a_-(g^t) \leq \sigma^t \leq a_+(g^t)$ for all $t \geq 0$. 

For a non-increasing vector $\bm{x} \in \R^{K+1}$, define
\begin{equation}\label{eq:rho}
\rho(\bm{x}) \in \R^{K+1} := \begin{cases}
(x_0 + 1 -\gamma, x_1, \ldots, x_{K-1}, x_{K-1}) & K \in \N \\
(x_0 + 1 - \gamma, x_1, x_2,  \ldots) & K = +\infty.
\end{cases}
\end{equation}
The vector $\rho(\bm{x})$ can be seen as a description of $r(f_{\bm{x}})$. From \eqref{eq:rForT}, we see that $r(f_{\bm{x}})$ may have a linear section with slope $c_-$, and so $r(f_{\bm{x}})$ may not necessarily be expressed as $f_{\bm{y}}$ for some non-increasing vector $\bm{y}$ (as defined in Definition \ref{def:T}). However, it is true that the linear intervals of $r(\bm{x})$ are bounded by the values $L\vee \rho(\bm{x})_j$. Thus, as discussed above, the different cases of $\psi_\sigma$ depend on $\sigma$ and $\rho(\bm{x})$. 

We refer the reader to Figure \ref{fig:operators} for a visual aid in understanding the definition of $\psi_\sigma$ below. Lemma \ref{lem:operators} may be proved by a case by case analysis for the different values of $\sigma$, though by examining the integral definition of $f_{\bm{x}}$ in \eqref{eq:defT}, the proof becomes somewhat more straightforward.

\begin{figure}[h!]
\centering

\begin{tikzpicture}[xscale = 3.5, yscale = 2.5, domain=1:1]
\draw[<->, thick] (-3,0) -- (0.75,0);

\draw[very thick, red] (0,0) -- (-2/10,2/10) -- (-5/10,41/100) -- (-9/10,57/100) -- (-12/10,6/10)-- (-16/10,52/100) -- (-19/10,37/100)-- (-21/10,21/100)-- (-23/10, 0);

\node at (-2.3,0.25) {\color{red} $f_{\bm{x}}$};

\draw (0,0) -- (0,-0.05) node[below]{\color{red} $x_0$};
\draw (-2/10,0) -- (-2/10,-0.05) node[below]{\color{red} $x_1$};
\draw (-5/10,0) -- (-5/10,-0.05) node[below]{\color{red} $x_2$};
\draw (-9/10,0) -- (-9/10,-0.05) node[below]{\color{red} $x_3$};
\draw (-12/10,0) -- (-12/10,-0.05) node[below]{\color{red} $x_4$};
\draw (-16/10,0) -- (-16/10,-0.05) node[below]{\color{red} $x_5$};
\draw (-19/10,0) -- (-19/10,-0.05) node[below]{\color{red} $x_6$};
\draw (-21/10,0) -- (-21/10,-0.05) node[below]{\color{red} $x_7$};

\draw[very thick, black] (0.4,0) -- (-2/10,6/10) -- (-5/10,81/100) -- (-9/10,97/100) -- (-12/10,10/10)-- (-16/10,92/100) -- (-19/10,77/100)-- (-2.92666,0);

\node at (-2.5,0.6) {$r(f_{\bm{x}})$};

\draw (0.4,0) -- (0.4,-0.15);
\node at (0.4,-0.25) {$\rho_0$};
\node at (-2/10,-0.25) {$\rho_1$};
\node at (-5/10,-0.25) {$\rho_2$};
\node at (-9/10,-0.25) {$\rho_3$};
\node at (-12/10,-0.25) {$\rho_4$};
\node at (-16/10,-0.25) {$\rho_5$};
\node at (-19/10,-0.25) {$\rho_7,\rho_6$};


\draw (0.6,0) -- (0.6,-0.05) node[below]{\color{orange} $\sigma_1$};

\draw[very thick,orange] (0.314,0.0) -- (-2/10,34/100) -- (-5/10,48/100) -- (-9/10,51/100) -- (-12/10,45/100) -- (-16/10,25/100)-- (-19/10,1/100) -- (-1.91,0);

\node at (-1.2,0.2) {\small \color{orange} $s_{\sigma_1}(r(f_{\bm{x}})) = f_{\bm{y}}$};

\draw (0.314,0) -- (0.314,-0.3);
\node at (0.285,-0.4) {\color{orange} $y_1\!,\!y_0$};
\node at (-2/10,-0.4) {\color{orange} $y_2$};
\node at (-5/10,-0.4) {\color{orange} $y_3$};
\node at (-9/10,-0.4) {\color{orange} $y_4$};
\node at (-12/10,-0.4) {\color{orange} $y_5$};
\node at (-16/10,-0.4) {\color{orange} $y_6$};
\node at (-19/10,-0.4) {\color{orange} $y_7$};

\node at (-2.5,-0.1) {\color{red} $\bm{x}:$};
\node at (-2.5,-0.25) {$\rho(\bm{x}):$};
\node at (-2.5,-0.4) {\color{orange} $\bm{y}\!=\!\psi_{\sigma_1}(\bm{x}):$};

\end{tikzpicture}

\medskip

\begin{tikzpicture}[xscale = 3.5, yscale = 2.5, domain=1:1]
\draw[<->, thick] (-3,0) -- (0.75,0);

\draw[very thick, red] (0,0) -- (-2/10,2/10) -- (-5/10,41/100) -- (-9/10,57/100) -- (-12/10,6/10)-- (-16/10,52/100) -- (-19/10,37/100)-- (-21/10,21/100)-- (-23/10, 0);

\node at (-2.3,0.25) {\color{red} $f_{\bm{x}}$};

\draw (0,0) -- (0,-0.05) node[below]{\color{red} $x_0$};
\draw (-2/10,0) -- (-2/10,-0.05) node[below]{\color{red} $x_1$};
\draw (-5/10,0) -- (-5/10,-0.05) node[below]{\color{red} $x_2$};
\draw (-9/10,0) -- (-9/10,-0.05) node[below]{\color{red} $x_3$};
\draw (-12/10,0) -- (-12/10,-0.05) node[below]{\color{red} $x_4$};
\draw (-16/10,0) -- (-16/10,-0.05) node[below]{\color{red} $x_5$};
\draw (-19/10,0) -- (-19/10,-0.05) node[below]{\color{red} $x_6$};
\draw (-21/10,0) -- (-21/10,-0.05) node[below]{\color{red} $x_7$};

\draw[very thick, black] (0.4,0) -- (-2/10,6/10) -- (-5/10,81/100) -- (-9/10,97/100) -- (-12/10,10/10)-- (-16/10,92/100) -- (-19/10,77/100)-- (-2.92666,0);

\node at (-2.5,0.6) {$r(f_{\bm{x}})$};

\draw (0.4,0) -- (0.4,-0.15);
\node at (0.4,-0.25) {$\rho_0$};
\node at (-2/10,-0.25) {$\rho_1$};
\node at (-5/10,-0.25) {$\rho_2$};
\node at (-9/10,-0.25) {$\rho_3$};
\node at (-12/10,-0.25) {$\rho_4$};
\node at (-16/10,-0.25) {$\rho_5$};
\node at (-19/10,-0.25) {$\rho_7,\rho_6$};

\draw (0.3,0.05) -- (0.3,-0.05); 
\node at (0.3,-0.125) {\color{blue} $\sigma^*$};
\draw (0.3,-0.2) -- (0.3,-0.3);

\draw[very thick,blue] (0.4,0) -- (0.3,0.1) -- (-2/10,45/100) -- (-5/10,57/100) -- (-9/10,60/100) -- (-12/10,54/100) -- (-16/10,34/100)-- (-19/10,1/10) -- (-1.9952,0);

\node at (-1.3,0.2) {\small \color{blue} $s(r(f_{\bm{x}})) = f_{\bm{y}}$};

\node at (0.425,-0.4) {\color{blue} $y_0$};
\node at (0.275,-0.4) {\color{blue} $y_1$};
\node at (-2/10,-0.4) {\color{blue} $y_2$};
\node at (-5/10,-0.4) {\color{blue} $y_3$};
\node at (-9/10,-0.4) {\color{blue} $y_4$};
\node at (-12/10,-0.4) {\color{blue} $y_5$};
\node at (-16/10,-0.4) {\color{blue} $y_6$};
\node at (-19/10,-0.4) {\color{blue} $y_7$};

\node at (-2.5,-0.1) {\color{red} $\bm{x}:$};
\node at (-2.5,-0.25) {$\rho(\bm{x}):$};
\node at (-2.5,-0.4) {\color{blue} $\bm{y}\!=\!\psi_{\sigma^*}(\bm{x}):$};

\end{tikzpicture}

\medskip

\begin{tikzpicture}[xscale = 3.5, yscale = 2.5, domain=1:1]
\draw[<->, thick] (-3,0) -- (0.75,0);

\draw[very thick, red] (0,0) -- (-2/10,2/10) -- (-5/10,41/100) -- (-9/10,57/100) -- (-12/10,6/10)-- (-16/10,52/100) -- (-19/10,37/100)-- (-21/10,21/100)-- (-23/10, 0);

\node at (-2.3,0.25) {\color{red} $f_{\bm{x}}$};

\draw (0,0) -- (0,-0.05) node[below]{\color{red} $x_0$};
\draw (-2/10,0) -- (-2/10,-0.05) node[below]{\color{red} $x_1$};
\draw (-5/10,0) -- (-5/10,-0.05) node[below]{\color{red} $x_2$};
\draw (-9/10,0) -- (-9/10,-0.05) node[below]{\color{red} $x_3$};
\draw (-12/10,0) -- (-12/10,-0.05) node[below]{\color{red} $x_4$};
\draw (-16/10,0) -- (-16/10,-0.05) node[below]{\color{red} $x_5$};
\draw (-19/10,0) -- (-19/10,-0.05) node[below]{\color{red} $x_6$};
\draw (-21/10,0) -- (-21/10,-0.05) node[below]{\color{red} $x_7$};

\draw[very thick, black] (0.4,0) -- (-2/10,6/10) -- (-5/10,81/100) -- (-9/10,97/100) -- (-12/10,10/10)-- (-16/10,92/100) -- (-19/10,77/100)-- (-2.92666,0);

\node at (-2.5,0.6) {$r(f_{\bm{x}})$};

\draw (0.4,0) -- (0.4,-0.15);
\node at (0.4,-0.25) {$\rho_0$};
\node at (-2/10,-0.25) {$\rho_1$};
\node at (-5/10,-0.25) {$\rho_2$};
\node at (-9/10,-0.25) {$\rho_3$};
\node at (-12/10,-0.25) {$\rho_4$};
\node at (-16/10,-0.25) {$\rho_5$};
\node at (-19/10,-0.25) {$\rho_7,\rho_6$};


\draw (-0.7,0) -- (-0.7,-0.05) node[below] {\color{violet} $\sigma_2$};
\draw (-0.7,-0.2) -- (-0.7,-0.3);
\draw[very thick,violet] (0.4,0) -- (-2/10,6/10) -- (-5/10,81/100)--(-7/10,89/100) -- (-9/10,90/100) -- (-12/10,84/100) -- (-16/10,64/100)-- (-19/10,40/100) -- (-2.281,0);

\node at (-0.9,0.75) {\small \color{violet} $s_{\sigma_2}(r(f_{\bm{x}})) = f_{\bm{y}}$};

\node at (0.4,-0.4) {\color{violet} $y_0$};
\node at (-2/10,-0.4) {\color{violet} $y_1$};
\node at (-5/10,-0.4) {\color{violet} $y_2$};
\node at (-7/10,-0.4) {\color{violet} $y_3$};
\node at (-9/10,-0.4) {\color{violet} $y_4$};
\node at (-12/10,-0.4) {\color{violet} $y_5$};
\node at (-16/10,-0.4) {\color{violet} $y_6$};
\node at (-19/10,-0.4) {\color{violet} $y_7$};

\node at (-2.5,-0.1) {\color{red} $\bm{x}:$};
\node at (-2.5,-0.25) {$\rho(\bm{x}):$};
\node at (-2.5,-0.4) {\color{violet} $\bm{y}\!=\!\psi_{\sigma_2}(\bm{x}):$};

\end{tikzpicture}

\caption{Let $\beta = 0.3$, $\gamma = 0.6$, and $c_- = 0.75$, and so $k=4$ and $K=7$.
The functions $f_{\bm{x}}$, $r(f_{\bm{x}})$ and $s_\sigma(r(f_{\bm{x}})) = f_{\bm{y}}$ are illustrated along with the positions of the entries of the vectors $\bm{x}$, $\rho(\bm{x})$, and $\bm{y} = \psi_{\sigma}(\bm{x})$ for $\sigma \in \{\sigma_1, \sigma^\ast, \sigma_2\}$, where $\sigma^\ast := \sigma^\ast(r(f_{\bm{x}}))$. 
}
\label{fig:operators}
\end{figure}
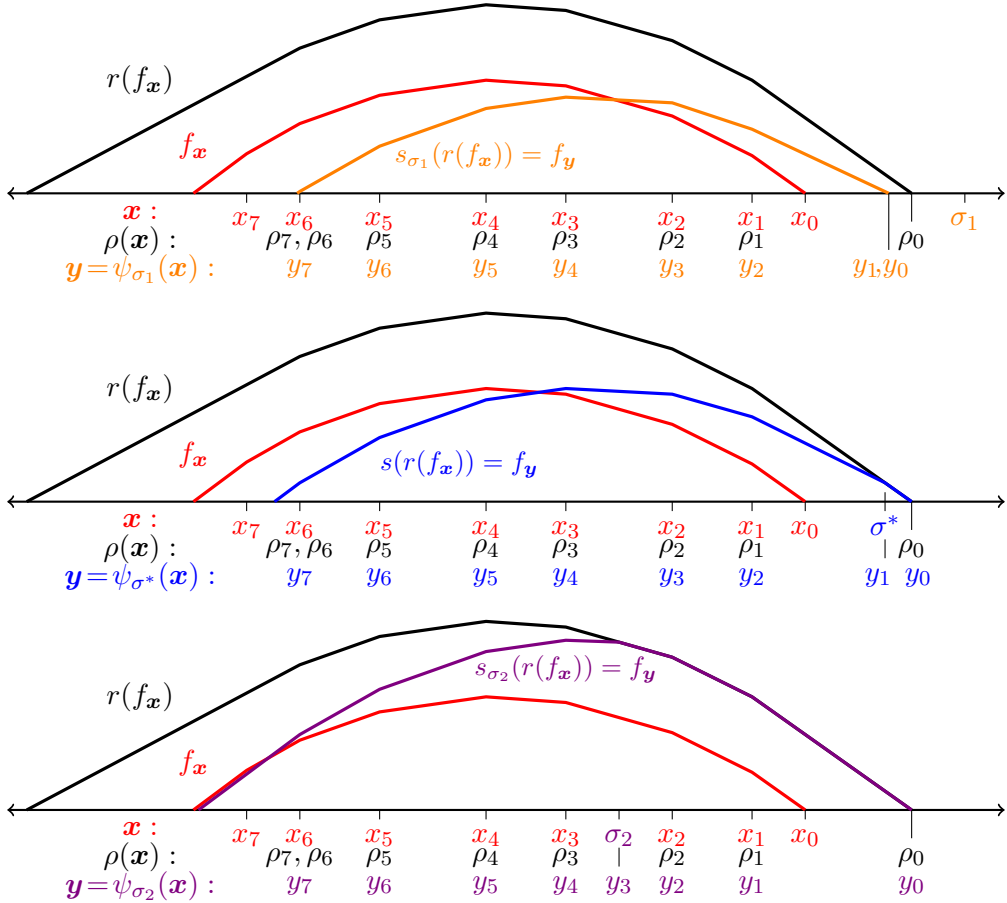

\begin{lem}\label{lem:operators}
Recall the definition of $a_-(g)$ and $a_+(g)$ from \eqref{eq:sigmaRange}. If $\bm{x} \in \R^{K+1}$ is non-increasing and $\sigma \in \R$ satisfies 
\begin{equation}\label{eq:operatorsSigmaCond}
a_-(f_{\bm{x}})\vee x_{k+1} \leq \sigma \leq a_+(f_{\bm{x}}),
\end{equation}
then $s_{\sigma}(r (f_{\bm{x}})) = f_{\psi_\sigma(\bm{x})},$
where
\begin{equation}\label{eq:operators}\psi_\sigma(\bm{x}) \!:= \!\begin{cases}
\left(\frac{\rho(\bm{x})_0 - \beta \sigma}{1-\beta}, \frac{\rho(\bm{x})_0 - \beta \sigma}{1-\beta}, x_1, x_2, x_3, \ldots, x_{K-1}\right) & \rho(\bm{x})_0 < \sigma, \\
\left(\rho(\bm{x})_0, x_1, \ldots, x_i, \sigma, x_{i+1}, x_{i+2}, \ldots, x_{K-1}\right) & \rho(\bm{x})_{i+1}\! <\! \sigma \!\leq \!\rho(\bm{x})_i \text{ for } 0\! \leq\! i\! \leq\! k.
\end{cases}
\end{equation}
\end{lem}

\begin{rmk}\label{rmk:operators}
Note from \eqref{eq:sigmaRange}, if $\beta \geq 1$, then $a_+(f_{\bm{x}}) = x_0 + 1 - 2\gamma < \rho(\bm{x})_0$ and so the first case of \eqref{eq:operators} does not apply. Therefore, \eqref{eq:operators} is well defined for all $\sigma$ satisfying \eqref{eq:operatorsSigmaCond}. 
\end{rmk}

\begin{proof}
Denote $f := f_{\bm{x}}$ and $L:= L(f), U:=U(f)$. We start by assuming that $x_{K-1} \geq L$, and so $L' = x_{K-1}$ as defined in Lemma \ref{lem:repLin} (ii). From \eqref{eq:rForT} and the formula $f_{\bm{x}}$ from \eqref{eq:defT}, we may write 
\begin{equation}\label{eq:newrForT}
r(f)(x) = \pi\left(\int_x^{\rho(\bm{x})_0}\!\! -c_- \vee\left(1 - \beta\sum_{j=1}^{K-1} \1_{\{z < x_j\}}(z)\right)dz\right).
\end{equation}
Since $\beta(x - \sigma)_- = -\int_x^\infty \beta\1_{\{z < \sigma\}}(z)dz$, we see that 
\begin{equation}\label{eq:srForT}
s_\sigma \circ r (f)(x) = \pi\left(\int_x^{\rho(\bm{x})_0}\!\! -c_- \vee\left(1 - \beta\sum_{j=1}^{K-1} \1_{\{z < x_j\}}(z)\right)dz - \int_x^\sigma \!\!\beta\1_{\{z < \sigma\}}(z)dz\right).
\end{equation}
We need only show that the expression in \eqref{eq:srForT} is equal to the expression of $f_{\psi_\sigma(\bm{x})}$ from \eqref{eq:defT}. Along with the restriction $\sigma \geq x_{k+1}$ from \eqref{eq:operatorsSigmaCond}, this follows when $\sigma \leq \rho(\bm{x})_0$ by replacing the upper bound of the second integral in \eqref{eq:srForT} with $\rho(\bm{x})_0$. Now let $\rho(\bm{x})_0 < \sigma \leq a_+(f)$, which by Remark \ref{rmk:operators} can only occur if $\beta < 1$. Let $y_0 = U(s_\sigma \circ r (f))$ be the upper edge of $s_\sigma \circ r (f)$. Then $y_0 \leq \rho(\bm{x})_0 < \sigma$ and so $\beta(x - \sigma)_-$ can be rewritten as
\begin{equation*}
\int_x^\sigma \beta \1_{\{z < \sigma\}}(z)dz = \int_x^{y_0} \beta \1_{\{z < y_0\}}(z)dz + \int_{y_0}^\sigma \beta \1_{\{z < \sigma\}}(z)dz.
\end{equation*}
From Lemma \ref{lem:repLin}(ii) and the restriction $\sigma \leq a_+(f) < x_0 + \frac{1-\gamma}{\beta}$, we see that 
\[s_\sigma \circ r (f)(x_0) = r(f)(x_0) + \beta(x_0 - \sigma) = 1 -\gamma + \beta x_0 - \beta \sigma > 0,\]
and so $y_0 \geq x_0$. This implies the integrand of the first integral in \eqref{eq:srForT} is simply 1 for all $x \in [y_0,\rho(\bm{x})_0]$. Thus, so long as 
\begin{equation}\label{eq:Whatisy0}
\int_{y_0}^{\rho(\bm{x})_0}dz - \int_{y_0}^\sigma \beta \1_{\{z < \sigma\}}(z)dz = 
\rho(\bm{x})_0 - \beta\sigma -(1-\beta)y_0 = 0,
\end{equation}
then \eqref{eq:srForT} can be rewritten as 
\begin{equation}\label{eq:srForTStrong}
s_\sigma \circ r (f)(x) = \pi\left(\int_x^{y_0}-c_- \vee \left(1 - \beta \sum_{j=1}^{K-1}\1_{\{z < x_j\}}(z)\right)dz - \int_x^{y_0}\beta\1_{\{z < y_0\}}(z)dz \right).
\end{equation}
Solving \eqref{eq:Whatisy0} yields $y_0 = \frac{\rho(\bm{x})_0 - \beta \sigma}{1-\beta}$, and we see that \eqref{eq:srForTStrong} is indeed equal to the integral expression for $f_{\psi_\sigma(\bm{x})}$ from \eqref{eq:defT}. \\

 Now suppose $L \geq x_{K-1}$. If we let $\bm{x}' = (x_0',x_1',\ldots,x_K')$ where $x_j' = L \vee x_j$, then $f_{\bm{x}'} = f$ by Lemma \ref{lem:EqClassf}. Then we may revert to the case above for $f_{\bm{x}'}$ and see that 
 \[ s_\sigma \circ r (f) = s_\sigma \circ r(f_{\bm{x}'}) = f_{\psi_\sigma(\bm{x}')}.\]
Since we assume from \eqref{eq:operatorsSigmaCond} that $\sigma > L$, we see directly from \eqref{eq:operators} that $\psi_\sigma(\bm{x}')_j = L \vee \psi_\sigma(\bm{x})_j$. Therefore, from Lemma \ref{lem:EqClassf}, if $L(f_{\psi_\sigma(\bm{x}')}) \geq L$, then $f_{\psi_\sigma(\bm{x})} = f_{\psi_\sigma(\bm{x}')}$.

Since $L \geq x_{K-1}$, then $f(L) = 0$ and by \eqref{eq:rForT}, $r(f)(x) < r(f)(L) = 1 - \gamma$ for all $x < L$. Since we restrict $\sigma \geq a_-(f) = L + \frac{1-\gamma}{1 \wedge\beta}$, then 
 \begin{equation*}\label{eq:x<L}
 x < L \,\, \Longrightarrow \,\, s_\sigma \circ r (f)(x) = \pi\left(\underbrace{r(f)(x) + \beta(x - \sigma)_-}_{<0}\right)= -\infty,
 \end{equation*}
and so $L(f_{\psi(\bm{x}')}) \geq L$.
\end{proof}

\subsection{The space of piecewise linear log-profiles is attractive}\label{sec:pwl}

The goal of this section is to prove that $\mathcal{T}$ is attractive in the following sense:
\begin{thm}\label{thm:pwl}
There exists a number $t_\ell := t_\ell(\gamma, \beta)$ (depending only on $\gamma$ and $\beta$) such that for all $g^0\in{\cal D}$, then $g^t \in \mathcal{T}$ for all $t \geq t_\ell$.
\end{thm}

In proving Theorem \ref{thm:pwl}, we will first prove the following comparison principle:

\begin{lem}[comparison principle]\label{lem:recursion}
Define $\bm{x}^0 = (\omega(g^0),L^0,L^0,\ldots, L^0) \in \R^{K+1}$. The recursion 
\[ 
\bm{x}^{t} := \psi_{\sigma^{t-1}}(\bm{x}^{t-1}), \qquad t \geq 1\]
is well defined, in that $\sigma^{t-1}$ satisfies the conditions \eqref{eq:operatorsSigmaCond} to apply $\psi_{\sigma^{t-1}}$ to $\bm{x}^{t-1}$. Furthermore, for all $t \geq 1$, $g^t \leq f_{\bm{x}^t}$ and $g^t(x) = f_{\bm{x}^t}(x)$ for all $x \geq U^0$.
\end{lem}

To complete the argument of Theorem \ref{thm:pwl}, we show in the proof of Theorem \ref{thm:pwl} that the supports of the profiles $(g^t)_{t \geq 0}$ are increasing. Specifically we show that the upper edges $(U^t)_{t \geq 0}$ are always increasing by at least some constant $C$. By applying the comparison principle described above, this implies that more and more of the fronts of the profiles $(g^t)_{t \geq 0}$ are piecewise linear until eventually, $g^t = f_{\bm{x}^t} \in \mathcal{T}$ for all $t$ large enough.

\medskip

 Before proving Lemma \ref{lem:recursion}, we begin with a technical lemma that provides bounds for $\sigma^t$. These bounds will imply that $\sigma^t$ satisfies the conditions of Lemma \ref{lem:operators} so that we may apply $\psi_{\sigma^t}$ to $\bm{x}^t$ and guarantee that $s_{\sigma^t} \circ r (f_{\bm{x}^t}) = f_{\bm{x}^{t+1}}$. 

 \begin{lem}\label{lem:sigmatBounds} Recall the sequence $(\sigma^t)_{t \geq 0}$ from \eqref{eq:sigmatDef}. The following hold:
 \begin{itemize}
\item[(i)] Recall $a_-$ and $a_+$ from \eqref{eq:sigmaRange}. For all $t \geq 0$, $\omega(g^t) \in [U^t, U^t + \gamma]$
\[  a_-(g^t) \leq \sigma^t \leq a_+(g^t).\]
\item[(ii)] The sequence $(\sigma^t)_{t=0}^\infty$ is non-decreasing.
\end{itemize}
 \end{lem}

 \begin{proof}
By definition of $U^t$ being the upper edge of $g^t$ and since $\sup g^t =\gamma$, we see immediately that $U^t \leq \omega(g^t) \leq U^t + \gamma$.
Recall that since $\sigma^t := \sigma^\ast(r(g^t))$, then $s \circ r (g) = s_{\sigma^t}(r(g^t))$. 

Starting with the lower bound of $\sigma^t$, we note that for all $x \geq \sigma^t$, since $h(0) = 0$,
\[ \gamma \geq s \circ r (g^t)(x) = r(g^t)(x) \geq 1 - \gamma + g^t(x),\]
and so $g^t(x) \leq 2\gamma - 1 < \gamma$. Since $\sup g^t = \gamma$, this implies that $\sup_{z \in [L^t,\sigma^t]}g^t(z) = \gamma$, and
\begin{multline*} \gamma \geq 0 \vee s \circ r(g^t)(\sigma^t) = 0 \vee r(g^t)(\sigma^t) \geq 1 - \gamma + \sup_{z \in [L^t,\sigma^t]}(g^t(z) + h(\sigma^t - z)) \\
\geq 1 - \gamma + \sup_{z \in [L^t,\sigma^t]}g^t(z) + L^t - \sigma^t = 1 + L^t - \sigma^t,
\end{multline*}
as well as
\begin{multline*}
\gamma \geq \sup_{z \in [L^t,\sigma^t]}(s \circ r (g^t)(z))= \sup_{z \in [L^t,\sigma^t]}(r(g^t)(z) + \beta(z - \sigma^t)) \\
\geq 1 - \gamma + \sup_{z \in [L^t,\sigma^t]}g^t(z) + \beta(L^t-\sigma^t) = 1 + \beta(L^t - \sigma^t).
\end{multline*}
The lower bound of (ii) then follows by solving for $\sigma^t$ in the inequalities above and taking the maximum of the two values. 

Now onto the upper bound of $\sigma^t$. When $\sup g^t = \gamma$, then $\sup r(g^t) =1$ which along with Lemma \ref{lem:repLin}(i) implies that for all $x \in \R$,
\[ r(g^t)(x) \leq r^t_{\max}(x) := 1\wedge \pi(1 - \gamma + \omega(g^t) - x).\]
We can verify that
\[ \sup s_{a_+(g^t)}(r^t_{\max}) = \sup_{x \in \R}\left( r^t_{\max}(x) + \beta(x -a_+(g^t))_-\right) = \gamma.\] 
Therefore $\sigma^t \leq \sigma^\ast(r^t_{\max}) \leq a_+(g^t)$ since $r(g^t) \leq r^t_{\max}$.\\

The proof of (ii) follows by contradiction, and so suppose $\sigma^{t+1} < \sigma^t$. For all $x \in \R$,
\[ g^{t+1}(x) = \pi\left(r(g^t)(x) + \beta(x - \sigma^t)_-\right) \leq 1 + \beta(x - \sigma^t)_- \,\, \Longrightarrow \,\, \beta(x - \sigma^t)_- \geq g^{t+1}(x) - 1.\]
Let $\eps = (1-\gamma) \wedge \beta(\sigma^t - \sigma^{t+1})$. Since $\sup g^{t+1} = \gamma$, the above implies that for all $x \in \R$,
\[ \beta(x - \sigma^{t+1})_- \geq g^{t+1}(x) - 1 + \eps.\]
By evaluating $g^{t+2}$ and recalling that $h(0) = 0$, we then see that 
\begin{multline*}
\sup g^{t+2} = \sup_{x \in \R} \left(r(g^{t+1})(x) + \beta(x - \sigma^{t+1})_-\right) \\ \geq 
\sup_{x \in \R} \left(g^{t+1}(x) + 1-\gamma + g^{t+1}(x) - 1 + \eps\right) = 2\sup g^{t+1} - \gamma + \eps = \gamma + \eps,
\end{multline*}
a contradiction to $\sup g^{t+2} = \gamma$.
\end{proof}

Recalling the conditions imposed on $\sigma$ to apply Lemma \ref{lem:operators}, the following is a corollary to Lemma \ref{lem:sigmatBounds}(i):

\begin{cor}\label{cor:operators}
If $g = f_{\bm{x}} \in \mathcal{T}$ for some non-increasing $\bm{x} \in \R^{K+1}$, then $\sigma^\ast := \sigma^\ast(r(g))$ satisfies the conditions \eqref{eq:operatorsSigmaCond} of Lemma \ref{lem:operators} and 
\[ s \circ r(g) = f_{\psi_{\sigma^\ast}(\bm{x})} \in \mathcal{T}.\]
\end{cor}

\begin{proof}
Since $g \in \mathcal{T}$, then $g(x_{k+1}) = \gamma$. By using Lemma \ref{lem:repLin}(ii) for $r(g)$, for all $\sigma < x_{k+1}$, 
\[ s_{\sigma} \circ r (g)(x_{k+1}) = r(g)(x_{k+1}) = 1 - \gamma + g(x_{k+1}) = 1 > \gamma.\]
Since $\sup s \circ r (g) = \gamma$, we conclude that $\sigma^{\ast}(r(g)) \geq x_{k+1}$. The result then follows from Lemma \ref{lem:sigmatBounds}(i) by setting $g^0 = g$.
\end{proof}

We now move on to the proof of our comparison principle Lemma \ref{lem:recursion}.

\begin{proof}[Proof of Lemma \ref{lem:recursion}]
We prove the statement by induction. We see that $g^0 \leq f_{\bm{x}^0}$, and so $r(g^0) \leq r(f_{\bm{x}^0})$ and by Lemma \ref{lem:sigmatBounds}(i), 
\[a_-(f_{\bm{x}^0}) \leq a_-(g^0) \leq \sigma^0 \leq a_+(g^0) = a_+(f_{\bm{x}^0}). \]
Therefore, $\sigma^0$ satisfies the conditions \eqref{eq:operatorsSigmaCond} and $s_{\sigma^0} \circ r (f_{\bm{x}^0}) = f_{\bm{x}^1}$. By Lemma \ref{lem:repLin} we also see that $r(g^0)(x) = r(f_{\bm{x}^0})(x)$ for all $x \geq U^0$. This equality as well as the inequality $g^0 \leq f_{\bm{x}^0}$ are maintained when applying the operator $s_{\sigma^0}$ to both $r(g^0)$ and $r(f_{\bm{x}^0})$, which completes the base case $t=1$.

\medskip

Now suppose the claim holds for some $t \geq 1$, and so $g^t \leq f_{\bm{x}^t}$, with equality for $x \geq U^0$. We then clearly have that 
\begin{equation}\label{eq:pwlIneq}
\sup_{z \in \R}(g^t(z) + h(x-z)) \leq \sup_{z \in \R}(f_{\bm{x}^t}(z) + h(x-z)),
\end{equation}
and so $r(g^t) \leq r(f_{\bm{x}^t})$; applying the operator $s_{\sigma^t}$ to both functions maintains the inequality, and so $g^{t+1} \leq s_{\sigma^t} \circ r (f_{\bm{x}^t})$. 

\medskip

By Lemma \ref{lem:repLin}(ii), $\omega(f_{\bm{x}^t}) = x_0^t$, while $x_0^t = U^t \leq \omega(g^t) \leq \omega(f_{\bm{x}^t})$, and so by Lemma \ref{lem:repLin}, $r(g^t)(x) = r(f_{\bm{x}^t})(x)$ for all $x \geq x_0^t$. For all $x\leq x_0^t$, from the slopes of $f_{\bm{x}^t}$, we know that $f_{\bm{x}^t}(x) \leq f_{\bm{x}^t}(y) + y-x$ for all $y > x$. In particular, this implies (along with the induction hypothesis) that for all $x \in [U^0, x_0^t]$, 
\begin{multline*}
r(g^t)(x) \geq \pi\left(1 -\gamma + \sup_{z \geq x}(g^t(z) + h(x-z))\right)\\
= \pi\left(1-\gamma + \sup_{z \geq x}(f_{\bm{x}^t}(z) + h(x-z))\right) = r(f_{\bm{x}^t})(x),
\end{multline*}
The above along with \eqref{eq:pwlIneq} implies that $r(g^t)(x) = r(f_{\bm{x}^t})(x)$ for all $x \geq U^0$; the equality is maintained by applying the operation $s_{\sigma^t}$ to both functions, and so $g^{t+1}(x) = s_{\sigma^t} \circ r (f_{\bm{x}^t})(x)$ for $x \geq U^0$. 

\medskip

The inductive step is now complete if we show that $\sigma^{t}$ satisfies \eqref{eq:operatorsSigmaCond}, and so $s_{\sigma^t} \circ r (f_{\bm{x}^t}) = f_{\bm{x}^{t+1}}$. We have already argued that $\omega(g^t) = \omega(f_{\bm{x}^t}) = x^t_0$ while $L(f_{\bm{x}^t}) \leq L^t$ is evident since $g^t \leq f_{\bm{x}^t}$, and applying Lemma \ref{lem:sigmatBounds}(i) to $g^t$ guarantees 
\[ a_-(f_{\bm{x}^t}) \leq a_-(g^t) \leq \sigma^t \leq a_+(g^t) = a_+(f_{\bm{x}^t}).\]
All that is left is to show that $\sigma^t \geq x_{k+1}^t$. By examining \eqref{eq:operators} of Lemma \ref{lem:operators}, the induction hypothesis implies that $\sigma^{t-1} \geq x_{k+1}^t$, while $\sigma^{t} \geq \sigma^{t-1}$ is guaranteed by Lemma \ref{lem:sigmatBounds}(ii).
\end{proof}

\bigskip

With the comparison principle of Lemma \ref{lem:recursion} in hand, all that is left is to show that the front edges of $g^t$ are increasing, and so more and more of $g^t$ is piecewise linear until $g^t \in \mathcal{T}$ for all $t$ large enough.

\begin{proof}[Proof of Theorem \ref{thm:pwl}]
Recall the vectors $(\bm{x}^t)_{t=0}^\infty$ from Lemma \ref{lem:recursion}. We have that $x_0^t = \omega(g^0) \geq U^0$, and since $\bm{x}^{t+1} = \psi_{\sigma^t}(\bm{x}^t)$ and $\sigma^t$ satisfies the upper bound of \eqref{eq:operatorsSigmaCond}, a quick calculation via \eqref{eq:operators} of Lemma \ref{lem:operators} reveals
\begin{equation}\label{eq:FrontEdge} x_0^{t+1} - x_0^t = \psi_{\sigma^t}(\bm{x}^t)_0 - x_0^t \geq 
\begin{cases} (1-\gamma) \wedge \left(\frac{\beta \gamma}{1 - \beta}\right) & \beta < 1 \\
1-\gamma & \beta \geq 1.
\end{cases}
\end{equation} 
By repeatedly applying \eqref{eq:FrontEdge}, there exists $\tau_1 := \tau_1(\gamma, \beta)$ such that for all $t \geq \tau_1$, 
\begin{equation}\label{eq:U^tTooBig}
x_0^t > U^0 + \frac{\gamma}{1-k\beta}.
\end{equation}
Between $x_{k+1}^t$ and $x_0^t$, the function $f_{\bm{x}^t}$ has slopes $-1 + j\beta$ for $j=0, \ldots, k$, and so 
\begin{equation}\label{eq:pwlbound}
x \in [x_{k+1}^t, x_0^t) \,\, \Longrightarrow \,\, f_{\bm{x}^t}(x) \geq (1-k\beta)(x_0^t - x).
\end{equation}
Since $f_{\bm{x}^t}$ and $g^t$ coincide for all $x \geq U^0$, if $x_{k+1}^t \leq U^0$ for some $t \geq \tau_1$, then \eqref{eq:U^tTooBig} and \eqref{eq:pwlbound} would imply $g^t(U^0) = f_{\bm{x}^t}(U^0) > \gamma$, contradicting $\sup g^t = \gamma$. Thus for all $t \geq \tau_1$, 
\[ x_{k+1}^t > U^0.\]
Since $f_{\bm{x}^t} \geq g^t$, then $\sup f_{\bm{x}^t} \geq \sup g^t = \gamma$. Additionally, since the two functions coincide for $x \geq U^0$ and $f_{\bm{x}^t}$ reaches its maximum at $x_{k+1}^t$, we see that for all $t \geq \tau_1$, 
\begin{equation}\label{eq:pwlbound2}
\sup f_{\bm{x}^t} = f_{\bm{x}^t}(x_{k+1}^t) = g(x_{k+1}^t) \leq \gamma,
\end{equation}
thus $\sup f_{\bm{x}^t} = \gamma$ and $f_{\bm{x}^t} \in \mathcal{T}$ for all $t \geq \tau_1$. 

\medskip

Let $t \geq \tau_1$ for the rest of the proof. From Lemma \ref{lem:recursion}, $\sigma^t$ satisfies the lower bound of \eqref{eq:operatorsSigmaCond}. Therefore $\sigma^t \geq x_{k+1}^t$, and since $U^0 < x_{k+1}^t$, 
\begin{equation}\label{eq:U^tTooBig2}
g^{t+1}(U^0) \leq \pi\left(r(g^t)(U^0) + \beta(U^0 - \sigma^t)\right) \leq \pi\left(1 + \beta(U^0 - \sigma^t)  \right)\leq \pi\left(1 + \beta(U^0 - x_{k+1}^t)\right).
\end{equation}
From \eqref{eq:pwlbound} and \eqref{eq:pwlbound2}, we conclude that
\[ x_{k+1}^t \geq x_0^t- \frac{\gamma}{1-k\beta}.\]
By once more repeatedly applying \eqref{eq:FrontEdge}, the above inequality yields that for all $t \geq t_\ell$ for some $t_\ell := t_\ell(\gamma, \beta) \geq \tau_1$, 
\begin{equation}\label{eq:x_{k+1}^tTooBig}
x_{k+1}^t > U^0 + \frac{1}{\beta}.
\end{equation}
Substituting \eqref{eq:x_{k+1}^tTooBig} into \eqref{eq:U^tTooBig2}, $f_{\bm{x}^t}(U^0) = g^t(U^0) = -\infty$ for all $t \geq t_\ell$. Since $f_{\bm{x}^t} \in \mathcal{C}$ is concave, this implies that $f_{\bm{x}^t}(x) = -\infty$ for all $x < U^0$, and since $g^t \leq f_{\bm{x}^t}$, we see that the two functions coincide on the entire real line for all $t \geq t_\ell$. Since $f_{\bm{x}^t} \in \mathcal{T}$, so is $g^t \in \mathcal{T}$. 
\end{proof}

\section{Traveling Wave Solutions}\label{sec:tws}

For a non-increasing vector $\bm{x} \in \R^{K+1}$, recall the parameters $k$ and $K$ from \eqref{eq:defK}, the definition of $f_{\bm{x}}$ from Definition \ref{def:T}, as well as the definition of $\mathcal{T}$ from Definition \ref{def:Tgamma}. For $g = f_{\bm{x}} \in \mathcal{T}$, we define
\begin{equation}\label{eq:defv(g)}
\bm{v}(g) := (x_1 - x_2, \ldots, x_k - x_{k+1}) \in \R^k.
\end{equation}
For $j=1,\ldots,k$, \, $\bm{v}(g)_j$ consists of the linear segment of $g$ with slope $-(1-j\beta)$, and by \eqref{eq:TSup}, the length of the linear segment with slope $-1$ is given by 
\begin{equation}\label{eq:v0}
x_0 - x_1 = \gamma - \left\langle \bm{u}, \bm{v}(g) \right\rangle, \qquad \text{where} \qquad \bm{u} := (1-j\beta)_{j=1}^k.
\end{equation}
Since all other slopes of $g$ are increasing, $\bm{v}(g)$ and \eqref{eq:v0} describe $g$ from the point where it reaches its maximum $\gamma$ to its upper edge. 
We also note by \eqref{eq:Samexk+1}, that $\bm{v}(g)$ is uniquely defined even if there are potentially several $\bm{x}$'s such that $g = f_{\bm{x}}$. 
When $\beta \geq 1$ (and so $k =0$), then $\bm{v}(g)$ is simply the empty vector; we will assume in this case that $\left\langle \bm{u}, \bm{v}(g) \right\rangle = 0$, and so \eqref{eq:v0} simplifies to $x_0 - x_1= \gamma$.

\bigskip

As is stated in the following lemma, any traveling wave solution must be a piecewise linear function $g \in \mathcal{T}$ for which $\bm{v}(g)$ is invariant under the operation $s \circ r$.

\begin{lem}\label{lem:operatorsFixedPoint}
If $(g, \nu)$ is a traveling wave solution, then $g \in \mathcal{T}$ and $\bm{v}(s \circ r (g)) = \bm{v}(g)$. 
\end{lem}

\begin{proof}
The first part of the lemma follows from Theorem \ref{thm:pwl}. Let $\bm{x}$ be a non-increasing vector such that $g = f_{\bm{x}}$ and let $\sigma^\ast := \sigma^\ast(r(g))$. By Corollary \ref{cor:operators} and the definition of traveling waves, 
\[f_{\psi_{\sigma^\ast}(\bm{x})} = s \circ r (g) = g( \cdot - \nu) = f_{\bm{x}}(\cdot - \nu).\]
The result now follows by Lemma \ref{lem:EqClassf} and the definitions of $\bm{v}(g)$ and $\bm{v}(s \circ r(g))$. 
\end{proof}

In seeking traveling wave solutions, the following lemma implies that we can in fact limit our search to a subclass of functions $g \in \mathcal{T}$.  

\begin{lem}\label{lem:DeltaClasses}
For $g \in \mathcal{T}$, if $\bm{v}(s \circ r (g)) = \bm{v}(g)$, then 
\begin{equation}\label{eq:DeltaClassesBound}
\sigma^\ast := \sigma^\ast(r(g) > U(g) + \left\langle \bm{u}, \bm{v}(g) \right\rangle - \gamma.
\end{equation}
\end{lem}

\begin{proof}
Let $g = f_{\bm{x}}$ for a non-increasing vector $\bm{x} \in \R^{K+1}$ and denote $\bm{v} := \bm{v}(g)$. Recalling \eqref{eq:v0}, if \eqref{eq:DeltaClassesBound} does not hold, then
\[ \sigma^{\ast}  \leq x_0 - \gamma + \left\langle \bm{u}, \bm{v} \right\rangle= x_1 = \rho(\bm{x})_1.\]
Therefore, from Lemma \ref{lem:operators} and \eqref{eq:v0} again,
\[ \gamma - \left\langle \bm{u},\bm{v}(s \circ r (g))\right\rangle = \psi_{\sigma^\ast}(\bm{x})_0 - \psi_{\sigma^\ast}(\bm{x})_1 = x_0 + 1 - \gamma - x_1 > x_0 - x_1 = \gamma - \left\langle \bm{u},\bm{v}(g) \right\rangle,\]
and so $\bm{v}(s \circ r (g)) \neq \bm{v}(g)$. 
\end{proof}

\bigskip

We will show in Section \ref{sec:affine} that the mapping $\bm{v} \mapsto \bm{v}(s \circ r (g))$ is a piecewise affine transformation that can be derived from Lemma \ref{lem:operators}. In light of Lemma \ref{lem:DeltaClasses}, we will only prove the transformations for $g \in \mathcal{T}$ satisfying \eqref{eq:DeltaClassesBound}, though we provide without proof the mapping $\bm{v} \mapsto \bm{v}(s \circ r (g))$ in terms of affine transformations for all $g \in \mathcal{T}$ (see Remark \ref{rmk:affine}). 

\medskip

The proof of Theorem \ref{thm:transition} is contained in Section \ref{sec:proofthmTransition}. We prove convergence results for the affine transformations of Section \ref{sec:affine}, the limits of which corresponding to potential vectors $\bm{v}(g)$ of functions $g \in \mathcal{T}$ satisfying $\bm{v}(s \circ r (g)) = \bm{v}(g)$. These limiting vectors inform the parameters $\nu(\gamma, \beta)$ and $\chi(\gamma, \beta)$ of Theorem \ref{thm:transition}

\medskip

After stating the continuity of certain maps in Section \ref{sec:proofthmSpeed}, Theorem \ref{thm:speed} follows by translating the vector convergence results of Section \ref{sec:proofthmTransition} to the convergence of translations of $g^t$ to the traveling wave solution. The continuity lemmas needed have technical proofs that are left to Section \ref{sec:ContinuityLemmas}.

\subsection{Piecewise affine transformations of the front}\label{sec:affine}

From \eqref{eq:TSup}, since $\sup g = \gamma$ for $g \in \mathcal{T}$, we see that $\bm{v}(g)$ must belong to the following subset of $\mathbb{R}^k$:
\begin{equation}\label{eq:defDelta}
\Delta := \left\{ \bm{v} \in \R^k :\: v_i \geq 0 \text{ and } \left\langle \bm{u}, \bm{v} \right\rangle \leq \gamma\right\}, \qquad \text{where } \bm{u} = (1-j\beta)_{j=1}^k.
\end{equation}
We wish to express $\bm{v}(g) \mapsto \bm{v}(s \circ r (g))$ in terms of vector transformations (independent of $g$), and to that end, we start with expressing $\sigma^\ast(r(g)) - U(g)$ in terms of $\bm{v}(g)$ alone. 

\begin{lem}\label{lem:sigma_l}
Recall $\bm{u}$ from \eqref{eq:defDelta} and define the piecewise linear operator $\Sigma : \R^k \to \R$ by
\begin{equation}\label{eq:sigmaV}
\Sigma(\bm{v}) := \begin{cases}
\left\langle \bm{u},\bm{v} \right\rangle - \gamma - \sum_{i=1}^{k-1} v_i - \frac{1-k\beta}{1\wedge\beta} v_{k} + \frac{1 - \gamma}{1\wedge\beta} & \text{ if } \ \  v_{k} < \frac{1-\gamma}{1-k\beta}, \\ \\
\left\langle \bm{u},\bm{v} \right\rangle - \gamma - \sum_{i=1}^{k} v_i + \frac{1 - \gamma}{1-k\beta} & \text{ otherwise}.
\end{cases}
\end{equation}
For all $g \in \mathcal{T}$,
\begin{equation}
\label{eq:defsigma_l}
\sigma^*(r(g))  = U(g) + \Sigma(\bm{v}(g)). 
\end{equation}
\end{lem}

\begin{proof}
Let $\bm{x}\in\R^{K+1}$ such that $f_{\bm{x}}=g$, and let $\sigma^\ast := \sigma^\ast(r(g))$. From Corollary \ref{cor:operators}, we have $s \circ r (g) = f_{\psi_{\sigma^\ast}(\bm{x})} \in \mathcal{T}$.\\

We start with $k=0$ (and so $\beta \geq 1$), then $\bm{v}$ is empty and both cases of \eqref{eq:sigmaV} simplify to $1 - 2\gamma$. By \eqref{eq:v0}, $\gamma = x_0 - x_1$ and 
\[
\gamma = \psi_{\sigma^\ast}(\bm{x})_0 - \psi_{\sigma^\ast}(\bm{x})_1 = x_0 + 1 - \gamma - \sigma^\ast.
\]
Solving for $\sigma^\ast$ yields $\sigma^\ast = 1 - 2\gamma$, completing the argument in this case.\\

For the remainder of the proof, assume $\beta \in (0,1)$. Recalling \eqref{eq:TSup}, $g$ achieves its maximum of $\gamma$ at $x_{k+1}$ and $s \circ r (g))$ achieves its maximum at $\psi_{\sigma^\ast}(\bm{x})_{k+1}$. Since $g$ has slope $-(1-k\beta)$ on $(x_{k+1}, x_k)$,
\begin{equation}\label{eq:sigmafDef} 
x \in (x_{k+1}, x_k] \,\, \Longrightarrow \,\, g(x) = \gamma - (1-k\beta)(x - x_{k+1}).
\end{equation}
By Lemma \ref{lem:operators}, if $\sigma^\ast > x_{k}$, then $\psi_{\sigma^\ast}(\bm{x})_{k+1} = x_k$ and by using Lemma \ref{lem:repLin}(ii) to calculate $r(g)(x_k)$,
\begin{equation}\label{eq:defsigma_2case1}
\gamma = s \circ r (g)(x_k) = r(g)(x_k) + \beta(x_k - \sigma^\ast) = 1-\gamma + g(x_k) + \beta(x_k - \sigma^\ast),
\end{equation}
while if $\sigma^\ast \in (x_{k+1}, x_k]$, then $\psi_{\sigma^\ast}(\bm{x})_{k+1} = \sigma^\ast$ and
\begin{equation}\label{eq:defsigma_2case2}
\gamma = s\circ r (g)(\sigma^\ast) = r(g)(\sigma^{\ast}) = 1 - \gamma +g(\sigma^\ast).
\end{equation}
Substituting \eqref{eq:sigmafDef} into \eqref{eq:defsigma_2case1} and \eqref{eq:defsigma_2case2}, and solving for $\sigma^\ast$ yields
\begin{equation}\label{eq:defsigma_2}
\sigma^\ast = \begin{cases}
x_k - \frac{1-k\beta}{\beta}(x_k - x_{k+1}) + \frac{1 - \gamma}{\beta} & \mbox{if} \ \  \sigma^\ast > x_{k} \\
x_{k+1} + \frac{1 - \gamma}{1-k\beta} & \mbox{if} \ \ \sigma^\ast \leq x_k.
\end{cases}
\end{equation}
By evaluating both cases of \eqref{eq:defsigma_2}, we see that
\begin{align*}
\sigma^\ast > x_k \,\, &\Longrightarrow \,\, x_k - x_{k+1} = \frac{\beta}{1-k\beta}\left(\sigma^\ast - x_k + \frac{1-\gamma}{\beta}\right) < \frac{1-\gamma}{1-k\beta}, \\
\sigma^\ast \leq x_k \,\, &\Longrightarrow \,\, x_k - x_{k+1} \geq \sigma^\ast - x_{k+1} = \frac{1-\gamma}{1-k\beta}.
\end{align*}
Thus $v_k = x_k - x_{k+1} < \frac{1-\gamma}{1-k\beta}$ if and only if $\sigma^\ast > x_k$. From \eqref{eq:defv(g)} and \eqref{eq:v0},
\begin{equation}\label{eq:xj=sumvj}
\text{for } i \geq 1, \qquad x_i = x_0 - (\gamma - \left\langle \bm{u}, \bm{v}\right\rangle) - \sum_{j=1}^{i-1} v_i,
\end{equation}
and replacing $x_j$ with \eqref{eq:xj=sumvj} in \eqref{eq:defsigma_2} completes the proof. 
\end{proof}

From Lemma \ref{lem:DeltaClasses} and applying Lemma \ref{lem:sigma_l} to express $\sigma^\ast(r(g))$, we see that if $g \in \mathcal{T}$ satisfies $\bm{v}(s \circ r (g)) = \bm{v}(g)$, then 
\begin{gather}\label{eq:DeltaClasses}
\begin{aligned}
\bm{v}(g) &\in \Delta_{-1} := \Big\{ \bm{v} \in \Delta :\: 1 - \gamma < \Sigma(\bm{v})\Big\}, \qquad \text{ or }\\
\bm{v}(g) & \in \, \Delta_0 \,\,\,:= \Big\{ \bm{v} \in \Delta :\: \left\langle \bm{u},\bm{v} \right\rangle - \gamma < \Sigma(\bm{v}) \leq 1 - \gamma\Big\}.
\end{aligned}
\end{gather}
Thus we focus on expressing the transformation $\bm{v}(g) \to \bm{v}(s \circ r (g))$ as piecewise linear transformations when $\bm{v}(g) \in \Delta_{-1} \cup \Delta_0$. Define the matrices
\[ A_{-1} := \left(\begin{array}{ccccc}
\frac{\beta}{1-\beta} & \frac{\beta}{1-\beta} & \cdots & \frac{\beta}{1-\beta} & \frac{1-k\beta}{1-\beta} \\ 
1 & 0 & \cdots & 0 & 0 \\
\vdots & & \ddots & & \vdots \\
0 & 0 & \cdots & 1 & 0
\end{array}\right)
\qquad 
A_{0} := \left(\begin{array}{ccccc}
-1 & -1 & \cdots & -1 & \frac{-(1-k\beta)}{\beta} \\
1 & 0 & \cdots & 0 & 0 \\
\vdots && \ddots && \vdots \\
0 & 0 & \cdots & 1 & 0
\end{array}\right),\]
and let $\left\{ \bm{e}_i \right\}_{i=1}^k$ be the standard basis of $\R^k$. 

\begin{lem}\label{lem:affine}
For $g \in \mathcal{T}$, let $\bm{v} := \bm{v}(g)$. 
\begin{itemize}
\item[(i)] If $\bm{v} \in \Delta_{-1}$, then 
\begin{equation}\label{eq:v+1-1}
\bm{v}(s \circ r(g)) = A_{-1}\bm{v} + \frac{\gamma - \left\langle \bm{u},\bm{v} \right\rangle}{1 - \beta}\bm{e}_1 \,\, \text{ and } \,\,
U(s \circ r (g)) = U(g) + \frac{1 - \gamma - \beta\Sigma(\bm{v})}{1-\beta}.
\end{equation}
\item[(ii)] If $\bm{v} \in \Delta_0$, then 
\begin{equation}\label{eq:v+10}
\bm{v}(s \circ r (g)) = A_0 \bm{v} + \frac{1-\gamma}{\beta}\bm{e}_1, \,\,\text{ and } \,\,
U(s \circ r (g))= U(g)+ 1 - \gamma.
\end{equation}
\end{itemize}
\end{lem}

\begin{proof}
Let $g = f_{\bm{x}}$ and $\sigma^\ast := \sigma^\ast(r(g))$. From Lemma \ref{lem:sigma_l}, the case $\bm{v}(g) \in \Delta_{-1}$ is equivalent to $\rho(\bm{x})_0 < \sigma^\ast$, and the case $\bm{v}(g) \in \Delta_0$ is equivalent to $\rho(\bm{x})_1 < \sigma^\ast \leq \rho(\bm{x})_0$. We apply Corollary \ref{cor:operators} and Lemma \ref{lem:operators}. In both cases, we see that $\psi_{\sigma^\ast}(\bm{x})_{j+1} = x_j$ for $j=1,\ldots,k$, and so 
\[\text{ for } j=2,\ldots, k, \qquad  \bm{v}(s \circ r (g))_{j} = \psi_{\sigma^\ast}(\bm{x})_j - \psi_{\sigma^\ast}(\bm{x})_{j+1}= x_{j-1} - x_j = \bm{v}(g)_{j-1}.\]
For $\bm{v}(s \circ r (g))_1 = \psi_{\sigma^\ast}(\bm{x})_1 - \psi_{\sigma^\ast}(\bm{x})_2$, if $\rho(\bm{x})_0 < \sigma^\ast$, then
\[
\bm{v}(s \circ r (g))_1
= \frac{x_0 + 1 - \gamma - \beta(x_0 + \Sigma(\bm{v}))}{1-\beta} - x_1 = \frac{\gamma - \left\langle \bm{u},\bm{v}(g) \right\rangle}{1-\beta} + \frac{\beta}{1-\beta}\sum_{i=1}^{k-1} v_i + \frac{1-k\beta}{1-\beta}v_k,\]
and if $\rho(\bm{x})_1 < \sigma^\ast \leq \rho(\bm{x})_0$, then 
\[ \bm{v}(s \circ r (g))_1
= \sigma^\ast - x_1 = x_0 - x_1 + \Sigma(\bm{v}) = -\sum_{i=1}^{k-1} v_i - \frac{1-k\beta}{\beta}v_k + \frac{1-\gamma}{\beta}.\]
Finally the results for $U(s \circ r (g)) - U(g)$ follow directly from $\psi_{\sigma^\ast}(\bm{x})_0 - x_0$.
\end{proof}

\begin{rmk}\label{rmk:affine}
The affine transformations of Lemma \ref{lem:affine} are the only ones needed for future proofs, but it may be of interest to the reader that affine transformations analogous to those in Lemma \ref{lem:affine} exist if $g \in \mathcal{T}$ and $\bm{v}(g)$ belongs to neither $\Delta_{-1}$ nor $\Delta_0$. We partition $\Delta$ into $\Delta = \Delta_{-1} \cup \Delta_0 \cup \Delta_1 \cup \cdots \cup \Delta_k$ by further defining
\[\text{ for } i=1,\ldots,k, \qquad \Delta_i := \left\{ \bm{v} \in \Delta :\:  - \sum_{j=1}^i v_i < \Sigma(\bm{v}) - \left\langle \bm{u}, \bm{v} \right\rangle + \gamma \leq - \sum_{j=1}^{i-1}v_i\right\}.\]
For $i=1,\ldots,k-1$, let $I_i \in \R^{i \times i}$ be the identity matrix and define the $k \times k$ matrix $A_i$ by
\[ A_i := \left(\begin{array}{cc}
I_{i-1} & 0 \\
0 & M_i \end{array} \right)
\qquad \text{where} \qquad 
M_i := \left(\begin{array}{cccccc}
1 & 1 & 1 & \cdots & 1 & \frac{1-k\beta}{\beta} \\
0 & -1 & -1 & \cdots & -1 & \frac{-(1-k\beta)}{\beta} \\
0 & 1 & 0 & \cdots & 0 & 0 \\
0 & 0 & 1 & \cdots & 0 & 0 \\
& \vdots & & \ddots & & \vdots \\
0 & 0 & 0 & \cdots & 1 & 0
\end{array}\right).\]
From case by case calculations using Lemmas \ref{lem:operators} and \ref{lem:sigma_l}, one may show that if $\bm{v}(g) \in \Delta_i$ for $i=1,\ldots,k-1$, then
\[
\bm{v}(s \circ r(g)) = A_i \bm{v} + \frac{1-\gamma}{\beta}\left(\bm{e}_{i+1} - \bm{e}_i\right)
\,\, \text{ and } U(s \circ r (g)) = U(g) + 1 - \gamma,\]
while if $\bm{v}(g) \in \Delta_k$, then 
\[
\bm{v}(s \circ r(g)) = I_k\bm{v} -\frac{1-\gamma}{1-k\beta}\bm{e}_k \,\, \text{ and } U(s \circ r (g)) = U(g) + 1 - \gamma.
\]
\end{rmk}

\subsection{Proof of Theorem \ref{thm:transition}}\label{sec:proofthmTransition}

The aim of this section is to prove the following proposition which establishes the existence and uniqueness of a traveling wave solution; the proof of Theorem \ref{thm:transition} will then follow as a consequence.

\begin{prop}\label{prop:unique}
 There is a  unique traveling wave solution $(G,\nu)$, which is given by
 \begin{equation}\label{eq:defy}
 G = f_{\bm y}, \ \ \text{ where }
\bm{y} := \left( 0,\left(-(\chi\vee 0) - (j-1)\nu)\right)_{j=1}^K\right),
\end{equation}
and $\nu = \nu(\gamma,\beta)$, $\chi = \chi(\gamma,\beta)$ are defined in the statement of Theorem \ref{thm:transition}.
\end{prop}

We postpone the proof of Proposition \ref{prop:unique} until the end of the section. Recalling Lemmas \ref{lem:operatorsFixedPoint} and \ref{lem:DeltaClasses}, we first seek vectors that are fixed points to the affine transformations of Lemma \ref{lem:affine}. In the following two lemmas, we prove that starting from any non-negative vector $\bm{v} \in \R^k$, after successive applications of \eqref{eq:v+1-1} or \eqref{eq:v+10}, the resulting sequence of vectors converge at geometric rate to a unique fixed point.

\begin{lem}\label{lem:AF}
Let $\beta \in (0,1)$, and define
\[
\bm{\mu}_L := (\nu_L, \nu_L, \ldots, \nu_L) \in \R^k, \text{ where } \nu_L=\frac{\gamma}{\sum_{j=1}^k(1-j\beta)} = \frac{2\gamma}{k(2 - (k+1)\beta)}.
\]
For $\bm{v} \in \R^k$, define $\bm{v}^t$ recursively by 
\begin{equation}\label{eq:A-1rec} \bm{v}^0= \bm{v}, \qquad \text{and} \qquad \forall t\geq 1, \,\, \bm{v}^t = A_{-1}\bm{v}^{t-1} + \frac{\gamma - \left\langle \bm{u},\bm{v}^{t-1} \right\rangle}{1-\beta}\bm{e}_1.
\end{equation}
There exists constants $C > 0$ and $0 < r < 1$ such that for all non-negative $\bm{v} \in \R^{k}$, 
\[\forall t\geq 1, \qquad  ||\bm{v}^t - \bm{\mu}_L||_\infty \leq Cr^t ||\bm{v} - \bm{\mu}_L||_\infty.\]
\end{lem}

\begin{proof}
From \eqref{eq:A-1rec}, we see that for all $t \geq 0$,
\[ \gamma - \left\langle \bm{u}, \bm{v}^{t+1}\right\rangle = \gamma - \beta\sum_{i=1}^{k-1}v_i^t - (1-k\beta)v_k^t -\gamma + \left\langle \bm{u},\bm{v}^t \right\rangle - \sum_{i=1}^{k-1}(1-(i+1)\beta)v_i^t = 0\]
We may therefore assume that $\left\langle \bm{u},\bm{v} \right\rangle = \gamma$ in \eqref{eq:A-1rec} (and so $\bm{v}^t = A_{-1}^t\bm{v}$). Note that $A_{-1}$ is an irreducible and aperiodic Leslie matrix with leading eigenvalue $\lambda _1 = 1$; $A_{-1}$ is a row stochastic matrix and since it's first row is positive, $A_{-1}$ is also primitive (and so aperiodic). The vectors $\bm{u}$ and $\bm{\mu}_L$ are left and right eigenvectors associated with $\lambda_1$ respectively. It follows from standard Perron-Froebenius Theory (see for example \cite{HOJO:13}) that if $\bm{v}$ is a non-negative vector satisfying $\left\langle \bm{u}, \bm{v}\right\rangle = \gamma$, then 
$A_{-1}^{t} \bm{v}$ converges to $\bm{\mu}_L$
at a geometric rate, completing the proof.
\end{proof}

\begin{lem}\label{lem:A0}
Let $\beta \in (0,1)$ and define
\[
\bm{\mu}_F := (1-\gamma, \ldots, 1-\gamma) \in \R^k.
\]
For $\bm{v} \in \R^{k}$, define $\bm{v}^t$ recursively by
\begin{equation}\label{eq:A0rec}
\bm{v}^0 = \bm{v}, \qquad \text{ and } \qquad \forall t \geq 1,\,\, \bm{v}^{t} =  A_{0} \bm{v}^{t-1} + \frac{1-\gamma}{\beta}\bm{e}_1.
\end{equation}
If $\beta^{-1} \not\in \N$, then there exists constants $C>0$ and $0 < r < 1$ such that for all $\bm{v} \in \R^{k}$,
\[ \forall t\geq 1,
\ \ \ ||\bm{v}^t - \bm{\mu}_F||_\infty \leq Cr^t|| \bm{v} - \bm{\mu}_F||_\infty.\]
\end{lem}

\begin{proof} We can directly verify that 
\begin{equation}\label{eq:A0mu}
\bm{\mu}_F = A_0 \bm{\mu}_F + \frac{1-\gamma}{\beta}\bm{e}_1.
\end{equation}
Define $\bm{\eps}^t := \bm{v}^t - \bm{\mu}_F$. We can then rework \eqref{eq:A0rec} with the aid of \eqref{eq:A0mu} to get
\begin{equation}\label{eq:epsGen}
\bm{\eps}^{t+1} = A_0\bm{\eps}^t.
\end{equation}
The matrix $A_0$ is a permuted polynomial companion matrix, and so by reading off the first row of $A_{0}$ we see immediately that its characteristic polynomial is 
\begin{equation}\label{eq:CharPoly}
\lambda^k + \lambda^{k-1} + \cdots + \lambda + \frac{1-k\beta}{\beta}.
\end{equation}
From \eqref{eq:CharPoly}, we see that $1$ is not an eigenvalue of $A_0$, and therefore the unique solution to $\bm{\eps} = A_0\bm{\eps}$ is the zero vector; and so from \eqref{eq:epsGen}, $\bm{\mu}_F$ is the unique solution to \eqref{eq:A0mu}. 

\medskip

Suppose $\beta \not\in \N$. From the Enestr\"om-Kakeya Theorem \cite{ENES:20, KAKE:12}, if $a_0 + a_1x + \cdots + a_nx^n$ is a polynomial with real coefficients such that $0 \leq a_{0} \leq \cdots \leq a_n$, then all of it's roots lie in the unit disk $|z| \leq 1$. Therefore, the eigenvalues $\lambda$ of 
$A_{0}$ satisfy $|\lambda| \leq 1$. Let $c = \frac{1-k\beta}{\beta}$, suppose $0 < c < 1$ (which holds when $\beta^{-1} \not\in \N$), and assume there is an eigenvalue with $|\lambda| = 1$, i.e., $\lambda = e^{i\theta}$ for some $\theta$. Then 
\[\sum_{n=0}^k e^{in\theta} = \lambda^k + \lambda^{k-1} + \cdots + \lambda + c + (1-c) = 1-c,\]
and it can be worked out that 
\[ \sum_{n=0}^k e^{in\theta} = \frac{\sin\left(\frac{1}{2}(k+1)\theta\right)}{\sin\left(\frac{1}{2}\theta\right)} e^{\frac{ik\theta}{2}}.\]
Since $1-c$ is real, this means that in the exponent of $e$, $k\theta = 2M\pi$ for some integer $M$. Then $\sin\left(\frac{1}{2}(k+1)\theta\right) = \sin\left(M\pi + \frac{1}{2}\theta\right) = \sin\left(\frac{1}{2}\theta\right)$, which then means $\sum_{n=0}^k e^{in\theta} = \pm 1 \neq 1-c,$ and so $e^{i\theta}$ cannot be an eigenvalue of $A_{0}$. Therefore, $|\lambda| < 1$ for all eigenvalues of $A_{0}$. Since $A_0'$ has spectral radius less than 1, we see that recursive applications of \eqref{eq:epsGen} yields that $||\eps^t||_\infty$ vanishes at a geometric rate. Since $\bm{v}^t = \bm{\mu}_F + \bm{\eps}^t$, the statement of the lemma follows.
\end{proof}

\begin{rmk}
If $\beta^{-1} \in \N$, then $\frac{1-k\beta}{\beta} = 1$, and $\lambda^k + \cdots + \lambda + 1$ has as its roots all the $(k+1)^{th}$ roots of unity except 1, and so $|\lambda| = 1$ for all eigenvalues of $A_{0}$. In fact in this case, we may verify that $A_0$ has period $k$, and so as $t \to \infty$, we see that $\bm{\eps}^t$ will simply cycle through $\bm{\eps}^{0}, \ldots, \bm{\eps}^{k}$. Thus $\bm{v}^t$ does not converge in this case. 
\end{rmk}

\medskip

From Lemmas \ref{lem:AF} and \ref{lem:A0}, we have shown that there is at most one fixed point for each of the two affine transformations of Lemma \ref{lem:affine}. To prove Proposition \ref{prop:unique}, we need to show that these solutions exist by showing that either 
\begin{equation}\label{eq:FPorSP}
\bm{\mu}_L \in \Delta_{-1} \qquad \text{or} \qquad \bm{\mu}_F \in \Delta_0.
\end{equation}
We show in the next lemma that $\gamma_c(\beta)$ as defined in \eqref{eq:gam} is the critical value of $\gamma$ for which we transition from the first inclusion of \eqref{eq:FPorSP} holding (when $\gamma < \gamma_c(\beta)$) to the second inclusion of \eqref{eq:FPorSP} holding (when $\gamma \geq \gamma_c(\beta)$).

\begin{lem}\label{lem:dis}
For $\gamma, \beta \in(0,1)$,
\begin{itemize}
\item[(i)]  $\bm{\mu}_L \in \Delta_{-1}$ iff $\gamma < \gamma_c(\beta)$, in which case $\Sigma(\bm{\mu}_L) = \beta^{-1}\left((\beta - 1)\nu_L + 1 - \gamma \right).$
\item[(ii)]  $\bm{\mu}_F\in \Delta_{0}$ iff $\gamma \geq\gamma_{c}(\beta)$, in which case $\Sigma(\bm{\mu}_F) = \left\langle \bm{u},\bm{\mu}_F \right\rangle + 1 - 2\gamma$.
\end{itemize}
\end{lem}

\begin{proof}
The key observation underlying the proof is that our choice of $\gamma_c := \gamma_c(\beta)$ entails that 
\begin{equation}\label{eq:disGamCrit} 
\text{for } \gamma \in (0,1), \qquad \begin{array}{ccccc}
\nu_L < 1 - \gamma & & \text{if and only if} & & \gamma < \gamma_c(\beta), \\
\gamma \geq \left\langle \bm{u},\bm{\mu}_F \right\rangle & & \text{if and only if} & & \gamma \geq \gamma_c(\beta),
\end{array}
\end{equation}
which follows by rewriting \eqref{eq:gam} as 
\[ \gamma_c = \frac{\sum_{j=1}^k(1-j\beta)}{1 + \sum_{j=1}^k(1-j\beta)} \,\, \Longrightarrow\,\,  (1-\gamma_c)\sum_{j=1}^k(1-j\beta) = \gamma_c.\]

Let $\gamma < \gamma_c(\beta)$. Then by \eqref{eq:disGamCrit}, $\left\langle \bm{u}, \bm{\mu}_F \right\rangle > \gamma$, and so does not belong to $\Delta_{0} \subset \Delta$.
In order to prove that $\bm{\mu}_L\in\Delta_{-1}$, we need to show that $\Sigma(\bm{\mu}_L)> 1-\gamma$.
From \eqref{eq:disGamCrit}, the first condition in \eqref{eq:sigmaV} is satisfied and a direct computation shows that 
\begin{equation}\label{eq:sigmaStrong}
\Sigma(\bm{\mu}_L) = \frac{(\beta - 1)\nu_L + 1-\gamma}{\beta}.
\end{equation}
Since  $\nu_L < 1-\gamma$ and $\beta-1<0$, this yields $\Sigma(\bm{\mu}_L) > 1 - \gamma$ and 
$\bm{\mu}_L\in\Delta_{-1}$. 

\medskip

Let us now take  
$\gamma \geq \gamma_c(\beta)$. For $\bm{\mu}_F$, The first condition of \eqref{eq:sigmaV} is satisfied when $\beta^{-1} \not\in \N$, while the second condition is satisfied when $\beta^{-1} \in \N$. Both cases yield 
\[\Sigma(\bm{\mu}_F) = -\gamma + \left\langle \bm{u},\bm{\mu}_F \right\rangle + 1- \gamma,\]
and so $-\gamma + \left\langle \bm{u},\bm{\mu}_F \right\rangle < \Sigma(\bm{\mu}_F) \leq 1-\gamma$ and $\bm{\mu}_F \in \Delta_0$. It remains to prove that $\bm{\mu}_L\not\in\Delta_{-1}$. 
If the first condition of \eqref{eq:sigmaV} is satisfied, we have the same expression as \eqref{eq:sigmaStrong}, and since $\nu_L \geq 1-\gamma$, it follows that $\Sigma(\bm{\mu}_L) \leq 1-\gamma$. If the second condition of \eqref{eq:sigmaV} is satisfied, then $\nu_L \geq \frac{1-\gamma}{1-k\beta}$, and $\Sigma(\bm{\mu}_L) \leq -\nu_L + \frac{1-\gamma}{1-k\beta} < 0$. Both conditions yield $\Sigma(\bm{\mu}_L) \leq 1-\gamma$, and so $\bm{\mu}_L \not\in \Delta_{-1}$.
\end{proof}

\bigskip

The proofs of Proposition \ref{prop:unique} and Theorem \ref{thm:transition} now follow from the previous lemmas. 

\begin{proof}[Proof of Proposition 
\ref{prop:unique}]
For $\beta \in (0,1)$, Lemmas \ref{lem:operatorsFixedPoint}, \ref{lem:DeltaClasses}, \ref{lem:affine}, along with Lemmas \ref{lem:AF}--\ref{lem:dis} imply that
\begin{equation}\label{eq:v(g)=mu}
\bm{v}(G) = \begin{cases}
\bm{\mu}_L & \text{ if } \gamma < \gamma_c(\beta),\\
\bm{\mu}_F & \text{ if } \gamma \geq \gamma_c(\beta),
\end{cases}
\end{equation}
while $\bm{v}(G) \in \Delta_0$ is the empty vector when $\beta \geq 1$. By Lemma \ref{lem:affine} and Lemma \ref{lem:dis}, a quick calculation yields
\[\nu := U(s\circ r (G)) - U(G) = 
\begin{cases}
\nu_L & \text{ if } \gamma < \gamma_c(\beta), \\
1-\gamma & \text{ if } \gamma \geq \gamma_c(\beta),
\end{cases}\] 
and so $\nu = \nu(\gamma, \beta)$ as defined in Theorem \ref{thm:transition}. Recalling $\bm{\mu}_L$ from Lemma \ref{lem:AF} and $\bm{\mu}_{F}$ from Lemma \ref{lem:A0}, we thus have $\bm{v}(G) = (\nu, \ldots, \nu) \in \R^k$ and 
\[ \gamma - \left\langle \bm{u}, \bm{v}(G) \right\rangle = \begin{cases}
0 & \text{ if } \gamma < \gamma_c(\beta), \\
-\chi & \text{ if } \gamma \geq \gamma_c(\beta),
\end{cases}\]
where $\chi := \chi(\gamma,\beta)$ is defined in \eqref{eq:defchi}. Thus if $G = f_{\bm{y}}$, we see from \eqref{eq:defv(g)} and \eqref{eq:v0} that $y_0 = 0$ and $y_j = -(\chi \vee 0) - (j-1)\nu$ for $j=1,\ldots,k+1$. We need only find $y_j$ for $j > k+1$. By Corollary \ref{cor:operators} and Lemma \ref{lem:operators}, we see that $\psi_{\Sigma(\bm{v}(G))}(\bm{y})_{j} = y_{j-1}$ for all $j= k+1, \ldots$. Since $(G,\nu)$ is a traveling wave, this implies $\psi_{\Sigma(\bm{v}(G))}(\bm{y})_{j} = y_j + \nu$, and so 
\begin{equation*}
\bm{y} = \left( 0,\left(-(\chi\vee 0) - (j-1)\nu)\right)_{j=1}^K\right),
\end{equation*}
completing the proof. 
\end{proof}

\begin{proof}[Proof of Theorem \ref{thm:transition}]
The statements on the speed of the wave (a.1) and (b.1) follow directly from Proposition \ref{prop:unique}. Let $g^0 = G$ and $g^1 = s \circ r (g^0)$; the statements (a.2) and (b.2) follow by evaluating $\sigma^\ast = \Sigma(\bm{v}(G))$ with the help of Lemma \ref{lem:dis} and \eqref{eq:v(g)=mu}. If $\gamma < \gamma_c$ then $\Sigma(\bm{v}(G)) > 1-\gamma$ and so $g^1 = r(g^0)(x) + \beta(x - \sigma^\ast) < r(g^0)(x)$ for all $x \in \text{Supp}(g^1)$, and if $\gamma \geq \gamma_c$ then $\Sigma(\bm{v}(G)) = -\chi + 1- \gamma$ and so $g^1(x) = r(g^0)(x)$ for all $x \in \text{Supp}(g^1)\cap[-\chi + 1 - \gamma,\infty) = [-\chi + 1 -\gamma, 1 - \gamma]$.
\end{proof}

\subsection{Proof of Theorem \ref{thm:speed}}\label{sec:proofthmSpeed}

Recall the definition of $\phi$:
\[ \phi(g_1,g_2) := \sup_{x \in \R}\left\lvert g_1(x)_+ - g_2(x)_+\right\rvert \vee |U(g_1) - U(g_2)| \vee |L(g_1) - L(g_2)|.\]
Throughout this section, we will denote $\Sigma := \Sigma(\bm{v}(G))$, where $G = f_{\bm{y}}$ is the traveling wave solution defined in \eqref{eq:defy} of Proposition \ref{prop:unique}. To prove Theorem \ref{thm:speed}, we will make use of the following continuity lemmas, the proofs of which follow by technical applications of previous lemmas and are relegated to Section \ref{sec:ContinuityLemmas}.

\begin{lem}\label{lem:cont}
The transformation $s \circ r$ is continuous (with respect to $\phi$) at $G$. 
\end{lem}

\begin{lem}\label{lem:Equiv}
There exists a constant $C > 0$ such that for any $g, g' \in \mathcal{T}$, 
\[ ||\bm{v}(g) - \bm{v}(g')||_\infty \leq C \phi(g, g').\]
\end{lem}

\begin{lem}\label{lem:Contf}
Let $\bm{x} \in \R^{K+1}$ be non-increasing. There exists a constant $C$ such that 
\begin{itemize}
\item[(i):] For $K \in \N$ and $||\bm{y} - \bm{x}||_\infty$ small enough, 
\[ \phi(G, f_{\bm{x}}) \leq C||\bm{y} - \bm{x}||_\infty.\]
\item[(ii):] For $K = \infty$ with $M = \min\{j \in \N : \: y_{j+1} < L(G)\}$ and $\displaystyle{\max_{j=0, \ldots, M}|y_j - x_j|}$ small enough,
\[ \phi(G, f_{\bm{x}}) \leq C \max_{j=0, \ldots, M}|y_j - x_j|.\]
\end{itemize}
\end{lem}

To outline the proof of Theorem \ref{thm:speed}, Lemmas \ref{lem:cont} and \ref{lem:Equiv} allow us to guarantee that once $g^{t_\ell} \in \mathcal{T}$ (recall Theorem \ref{thm:pwl}), then $\bm{v}(g^{t_\ell})$ is close to $\bm{v}(G)$ whenever $\phi(G, g^0)$ is small enough. We then apply the convergence Lemmas \ref{lem:AF} and \ref{lem:A0} of Section \ref{sec:proofthmTransition} to show that $\bm{v}(g^t)$ converges at a geometric rate to $\bm{v}(G)$. To apply these convergence lemmas we need (from Lemma \ref{lem:affine}) that $\bm{v}(g^t)$ and $\bm{v}(G)$ belong to the same set $\Delta_i$ for all $t \geq t_\ell$; this condition prevents us from being able to prove a global convergence result in this article, but does hold if $\bm{v}(G)$ is an interior point of either $\Delta_{-1}$ or $\Delta_0$ and $\bm{v}(g^{t_\ell})$ is close enough to $\bm{v}(G)$. Since $g^t \in \mathcal{T}$ then $g^t = f_{\bm{x}^t}$ for some $\bm{x}^t$, and we show that since $\bm{v}(g^t)$ converges to $\bm{v}(G)$ at a geometric rate, so does $\bm{x}^t$ converge to a translation of $\bm{y}$ at a geometric rate. The result then follows by applying Lemma \ref{lem:Contf}. 

\begin{proof}[Proof of Theorem \ref{thm:speed}]
Throughout, let $\bm{y}^t := (y_0 + t\nu, y_1 + t\nu, \ldots, y_K + t\nu)$ so that $G^t := f_{\bm{y}^t} = G(\cdot - t\nu)$, and recall $t_\ell$ from Theorem \ref{thm:pwl}.

By applying Lemma \ref{lem:cont} (and translations thereof), we see that 
\[ g^0 \mapsto g^{t_\ell}\]
is a composition of finitely many continuous mappings (w.r.t. $\phi$) around $G= G^0$, and since $G^{t_\ell},g^{t_\ell} \in \mathcal{T}$, we may apply Lemma \ref{lem:Equiv} and guarantee that $||\bm{v}(G^{t_\ell}) - \bm{v}(g^{t_\ell})||_\infty$ is arbitrarily small for all $\phi(G,g^0)$ small enough.

\medskip

For the rest of the proof, let $t \geq t_\ell$ so that $g^t \in \mathcal{T}$. 
The next step is to prove the existence of constants $C_1 >0$ and $0<r<1$ such that
\begin{equation}
||\bm{v}(g^t) - \bm{v}(G)||_\infty < C_1 r^t. \label{eq:SpeedV}
\end{equation}
When $\beta \geq 1$, \eqref{eq:SpeedV} is vacuously true since $\bm{v}(G)$ and $\bm{v}(g^t)$ are both empty, so now let $\beta \in (0,1)$. 
From Lemma \ref{lem:sigma_l} (and since $1-j\beta < 1$ for $j=0,\ldots,k$) we can see from (\ref{eq:sigmaV}) that
there exists a universal constant $C'$ such that
\begin{equation}\label{eq:SigmasSpeed}
\left\lvert\Sigma(\bm{v}(g^t)) - \Sigma\right\rvert \leq C'||\bm{v}(g^t) - \bm{v}(G)||_\infty.
\end{equation}
Thus, recalling the definition of $\Delta_i$ from \eqref{eq:DeltaClasses}, if $||\bm{v}(G) - \bm{v}(g^t)||_\infty$ is small enough, then \eqref{eq:SigmasSpeed} guarantees that $\bm{v}(g^t)$ and $\bm{v}(G)$ belong to the same set $\Delta_i$ (recall that we specify that $\gamma \neq \gamma_c$, and so $\bm{v}(G)$ is an interior point of the set $\Delta_i$ it belongs to). Lemma \ref{lem:affine} and Lemma \ref{lem:dis}, along with Lemma \ref{lem:AF} when $\gamma < \gamma_c$ and Lemma \ref{lem:A0} when $\gamma > \gamma_c$, \, $\beta^{-1} \not\in \N$ guarantee that if $||\bm{v}(g^{t_\ell}) - \bm{v}(G)||_\infty$ is small enough, then $||\bm{v}(g^t) - \bm{v}(G)||_\infty$ remains small enough such that for all $t \geq t_\ell$, $\bm{v}(g^t)$ and $\bm{v}(G)$ belong to the same set $\Delta_i$ and that \eqref{eq:SpeedV} holds. 

\medskip

For all values of $\beta > 0$ and $t \geq t_\ell$, since $\bm{v}(G)$ and $\bm{v}(g^t)$ belong to the same set $\Delta_i$, we see from Lemma \ref{lem:affine} and \eqref{eq:SigmasSpeed} that $|(U^{t+1} - U^t)-\nu| =0$ when $\gamma \geq \gamma_c(\beta)$ and is bounded by a constant times $||\bm{v}(g^t) - \bm{v}(G)||_\infty$ when $\gamma < \gamma_c(\beta)$.
From \eqref{eq:SpeedFront}, if we let 
\[c := \sum_{t=0}^\infty\left( (U^{t+1} - U^t) - \nu\right) < \infty,\]
then there exists a constant $C_2 > 0$ such that 
\begin{equation}\label{eq:SpeedFront} 
\left\lvert (U^t - c) - t\nu \right\rvert =\left\lvert \sum_{s=t+1}^\infty \left( (U^{t+1} - U^t) - \nu\right) \right\rvert < C_2 r^t.
\end{equation}

\medskip

Since $g^t \in \mathcal{T}$ for all $t \geq t_\ell$, let $\bm{x}^t \in \R^{K+1}$ be a non-increasing vector for which $g^t = f_{\bm{x}^t}$. For all $i=0, 1, \ldots, k+1$, from \eqref{eq:SpeedV}, \eqref{eq:SpeedFront} and \eqref{eq:v0}, along with the definition of $\bm{v}(g^t)$ in \eqref{eq:defv(g)}, 
\begin{equation}\label{eq:BoundxFromv}
|(x_i^t-c) -y_i^t | \leq |(U^t - c) - t\nu| + \left\lvert \left\langle \bm{u}, \bm{v}(g^t) - \bm{v}(G) \right\rangle \right\rvert + \sum_{j=0}^i|\bm{v}(g^t)_j - \bm{v}(G)_j| \leq C_3r^t 
\end{equation}
for a constant $C_3$. Since $\bm{v}(g^t)$ belongs to  $\Delta_{-1}$ or $\Delta_0$ (and so $\sigma^\ast(r(g^t)) > x_1$), by Lemma \ref{lem:operators}, 
\begin{equation}\label{eq:Boundx_jFromx_j}
x_{j+1}^{t+1} = x_j^t \,\, \text{ for } \,\, j =1,2,\ldots,K-1.
\end{equation}
Starting with the case $K \in \N$, we see from \eqref{eq:BoundxFromv} and \eqref{eq:Boundx_jFromx_j} that for all $t$ large enough, $||(\bm{x}^t - c(1,1,\ldots,1)) - \bm{y}^t||_\infty \leq C'r^t$, and applying Lemma \ref{lem:Contf} and a translation yields
\[ \phi\left(g^t(\cdot + t\nu + c), G\right) \leq Cr^t\]
for some constant $C$. For $K = \infty$, the argument is similar with $K$ replaced by $M$ as defined in Lemma \ref{lem:Contf}.  
\end{proof}

\subsubsection{Proofs of Continuity Lemmas}\label{sec:ContinuityLemmas}

In this section we provide proofs to the continuity lemmas \ref{lem:cont} -- \ref{lem:Contf}.

\begin{proof}[Proof of Lemma \ref{lem:Contf}]
We start with $K \in N$. For $||\bm{x} - \bm{y}||_\infty$ small enough, we may guarantee that on $\mbox{Supp}(G) \cap \mbox{Supp}(f_{\bm{x}})$, the slopes of the linear sections of $G$ and $f_{\bm{x}}$ differ by at most $\beta$, in which case the slopes only differ on the intervals $(y_j \wedge x_j, y_j \vee x_j)$. Thus for $x \in \mbox{Supp}(G) \cap \mbox{Supp}(f_{\bm{x}})$, 
\begin{multline*}
|G(x) - f_{\bm{x}}(x)| \leq \int_x^{x_0 \wedge y_0}|G'(z) - f'_{\bm{x}}(z)|dz + \int_{x_0 \wedge y_0}^{x_0\vee y_0} dz \leq \beta\sum_{j=1}^K|y_j - x_j| + |y_0 - x_0|.
\end{multline*}
For $x \not\in \mbox{Supp}(G) \cap \mbox{Supp}(f_{\bm{x}})$, since both $G, f_{\bm{x}}$ are concave and $0\vee G,0\vee f_{\bm{x}}$ are continuous, we see immediately that $|0\vee G(x) - 0\vee f_{\bm{x}}(x)|$ is bounded by the maximum value of this difference in $\mbox{Supp}(G) \cap \mbox{Supp}(f_{\bm{x}})$, completing the argument. \\

For $K = \infty$, the argument is similar to the one above but with $K$ replaced by $M$. The condition that $\displaystyle{\max_{j=0, \ldots, M}|y_j - x_j|}$ be small enough is used to guarantee that $L(f_{\bm{x}}) > x_M$, and so $0\vee f_{\bm{x}}$ is continuous.
\end{proof}

\begin{proof}[Proof of Lemma \ref{lem:cont}]
Before dwelling into the proof, the careful reader may notice that the transformation $s\circ r $ is not continuous everywhere in $\mathcal{D}$. However, around $G$, we will sandwich the function with two piecewise linear functions to control the effect of the transformation.

For conciseness, we assume that $K<\infty$. The continuity of the following maps follow from Lemmas \ref{lem:operators} and \ref{lem:Contf}.

\medskip

\noindent (1): $(\sigma, \bm{x}) \mapsto f_{\psi_\sigma(\bm{x})}$ is continuous at $(\Sigma, \bm{y})$.

\medskip

\noindent (2): For non-increasing $\bm{x} \in \R^{K+1}$ let $\sigma_{\bm{x}} := \sigma^\ast(r(f_{\bm{x}}))$, then $\bm{x} \mapsto \sigma_{\bm{x}}$ is continuous at $\bm{y}$. 

\medskip

Let $\sigma^\ast := \sigma^\ast(r(g))$. To proceed, it is evident that for all $\delta$ we may find non-increasing vectors $\bm{y}^{\pm\delta} \in \R^{K+1}$ near $\bm{y}$ such that $f_{\bm{y}^{-\delta}} \leq g \leq f_{\bm{y}^{+\delta}}$ for all $\phi(G,g) < \delta$. 
For all $\sigma \in \R$, the operator $s_\sigma \circ r $ maintains the inequality
\begin{equation}\label{eq:Contsrfeta}
s_{\sigma} \circ r (f_{\bm{y}^{-\delta}}) \leq s_\sigma \circ r(g) \leq s_{\sigma} \circ r (f_{\bm{y}^{+\delta}}),
\end{equation}
from which we may also conclude that $\sigma_{\bm{y}^{-\delta}} \leq \sigma^\ast \leq \sigma_{\bm{y}^{+\delta}}$. By applying (2) to $\bm{y}^{\pm\delta}$, we see that we may choose $\delta$ small enough to guarantee that $|\Sigma - \sigma^\ast|$ is arbitrarily small. From this arbitrarily small bound on $|\Sigma - \sigma^\ast|$ as well as applying (1) to $(\sigma^\ast, \bm{y}^{\pm \delta})$, we conclude from \eqref{eq:Contsrfeta} the statement of the lemma.
\end{proof}

\begin{proof}[Proof of Lemma \ref{lem:Equiv}]
Let $\bm{x}, \bm{x}' \in \R^{K+1}$ such that $g = f_{\bm{x}}$ and $g' = f_{\bm{x}'}$. The upper edges of the functions differ by $|x_0 - x_0'|$, and so 
\[\phi(f_{\bm{x}}, f_{\bm{x}'}) \geq |x_0 - x_0'|.\]
Fix $j=1, \ldots, k+1$, and assume without loss of generality that $x_j < x_j'$. We have by definition of the functions that for all $x_j < z \leq x_0\wedge x_j'$, 
\[f_{\bm{x}'}(x_j) - f_{\bm{x}'}(z) \leq (1-j\beta)(z - x_j), \qquad f_{\bm{x}}(x_j) - f_{\bm{x}}(z) \geq (1-(j-1)\beta)(z - x_j),\]
Thus, 
\begin{multline}\label{eq:Equiv1}
2\phi(f_{\bm{x}}, f_{\bm{x}'}) \geq |f_{\bm{x}}(z) - f_{\bm{x}'}(z)| + |f_{\bm{x}}(x_j) - f_{\bm{x}'}(x_j)| \\ \geq \left(f_{\bm{x}}(x_j) - f_{\bm{x}}(z)\right) - \left(f_{\bm{x}'}(x_j) - f_{\bm{x}'}(z)\right) \geq \beta(z - x_j).
\end{multline}
If $x_j' \leq x_0$, then \eqref{eq:Equiv1} implies $\phi(f_{\bm{x}}, f_{\bm{x}'}) \geq \beta(x_j' - x_j)/2$. If $x_0 < x_j'$, then with \eqref{eq:Equiv1},
\begin{multline*}
3\phi(f_{\bm{x}}, f_{\bm{x}'}) \geq |f_{\bm{x}}(x_0) - f_{\bm{x}'}(x_0)| + |f_{\bm{x}}(x_j) - f_{\bm{x}'}(x_j)| + |x_0' - x_0| \\
\geq \beta(x_0 - x_j) + (x_j' - x_0) \geq (1\wedge\beta)(x_j' - x_j),
\end{multline*}
and so $\phi(f_{\bm{x}}, f_{\bm{x}'})\geq (1\wedge \beta)(x_j' - x_j)/3$. We therefore have a constant $c$ such that
\[\phi(f_{\bm{x}}, f_{\bm{x}'}) \geq c \max_{j=0,\ldots,k+1}|x_j - x_j'|.\]
The statement of the lemma then follows from the above, recalling the definition of $\bm{v}(g)$ from \eqref{eq:defv(g)}, and the triangle inequality.
\end{proof}

\section*{Acknowledgement}

E.S. gratefully acknowledges support from the FWF project PAT3816823 “Waves in Population Genetics”.

\end{document}